\documentclass[11pt]{article}

\usepackage{style}

\begin{document}

\title{Characterizing the limiting critical Potts measures\\ on locally regular-tree-like expander graphs}
\author{Hang Du\footnote{Department of Mathematics, Massachusetts Institute of Technology} \and Yanxin Zhou\footnote{Department of Statistics, Stanford University}}
\date{\today}

\maketitle

\begin{abstract}
    For any integers $d,q\ge 3$, we consider the $q$-state ferromagnetic Potts model with an external field on a sequence of expander graphs that converges to the $d$-regular tree $\ttT_d$ in the Benjamini-Schramm sense. We show that along the critical line, any subsequential local weak limit of the Potts measures is a mixture of the free and wired Potts Gibbs measures on $\ttT_d$. Furthermore, we show the possibility of an arbitrary extent of strong phase coexistence: for any $\alpha\in [0,1]$, there exists a sequence of locally $\ttT_d$-like expander graphs $\{G_n\}$, such that the Potts measures on $\{G_n\}$ locally weakly converges to the $(\alpha,1-\alpha)$-mixture of the free and wired Potts Gibbs measures. Our result extends results of \cite{HJP23} which restrict to the zero-field case and also require $q$ to be sufficiently large relative to $d$, and results of \cite{BDS23} which restrict to the even $d$ case. We also confirm the phase coexistence prediction of \cite{BDS23}, asserting that the Potts local weak limit is a genuine mixture of the free and wired states in a generic setting. We further characterize the subsequential local weak limits of random cluster measures on such graph sequences, for any cluster parameter $q>2$ (not necessarily integer).
\end{abstract}


\section{Introduction}
Graphical models—such as the Ising and Potts models—on sparse locally-tree like graphs arise in many problems in statistical physics, combinatorics and theoretical computer science \cite{MM09}. Physicists have long conjectured that graphical models on these locally tree-like ensembles exhibit thermodynamic behavior very similar to those on regular lattices, especially with respect to phase transitions and the structure of Gibbs measures (e.g.\ \cite{abou1973selfconsistent, chalupa1979bootstrap, thouless1986spin}). At the same time, the absence of short loops and the recursive nature of their local neighborhoods often make these models more analytically tractable—and in many cases exactly solvable. For example, the Bethe approximation predicts thermodynamic quantities by locating the critical point(s) of a variational free energy functional (the ``Bethe functional''). Moreover, in the large-graph limit, the corresponding spin system is expected to be governed by the dominant critical point(s)—i.e., the global maximizer(s)—of the Bethe functional. We refer to \cite{DM10b} for a through review on the statistical physics insights of graphical models on locally tree-like graphs.

Building on physics predictions, over the past fifteen years, rigorous studies have examined graphical models on locally tree-like graphs, focusing on various aspects including the limiting free‐energy density \cite{DM10, DMS13, SS14, DGH14, DMSS14, Can17, Can19, BBC23, GGHJM25, CH25}, the local structure of Gibbs measures \cite{MMS12, BD17, HJP23, BDS23, YP24}, and related algorithmic questions \cite{Sly10, SS14, GSVY16, HJP23, CGGSV23, GSS25}. As noted above, a central theme of these studies is the analysis of critical points of the Bethe functional—all of which correspond to fixed points of the so-called belief propagation iteration—with particular emphasis on the dominant critical point(s). However, despite their shared structural features, the details of these analyses remain highly model‐dependent.

In this paper, we study the $q$-state ferromagnetic Potts model with an external field on graphs that locally resemble a $d$-regular tree, for integers $d,q\ge3$. The phase diagram in the large-system limit was characterized in \cite{DMSS14,CH25}; of particular interest to us is the critical line along which two different dominant critical points coexist. Along this line, the model is conjectured to exhibit strong coexistence of phases governed by these dominant critical points, and the corresponding limiting Gibbs measure is expected to decompose as a mixture of the associated pure states. In this work, we investigate these predictions in detail.

\subsection{Main results}\label{subsec-first}
In this subsection, we present the basic setup and state our main results. For the sake of clarity and brevity, some of the notation and definitions will be deferred to the next subsection. We begin with the precise definition of Potts model on a finite graph $G=(V,E)$. Fix an integer $q\ge 3$ as the number of states in the Potts model. Denote $[q]=\{1,2,\dots,q\}$. 

\begin{definition}[The Potts measure on a finite graph]
For two parameters $\beta,B>0$, the $q$-state Potts measure with inverse temperature $\beta$ and external field $B$ on $G$ is defined as follows:
\begin{equation}
\label{eq-def-potts}
    \mu_G^{\beta,B}[\sigma]:=\frac{1}{\mathcal Z_G^{\Po,\beta,B}}\exp\Bigg(\beta\sum_{(u,v)\in E}\delta_{\sigma(u),\sigma(v)}+B\sum_{v\in V}\delta_{\sigma(v),1}\Bigg)\,,\quad \forall \sigma\in [q]^{V}\,,
\end{equation}
where 
$$\calZ_G^{\Po,\beta,B}:=\sum_{\sigma\in [q]^{V}}\exp\Bigg(\beta\sum_{(u,v)\in E}\delta_{\sigma(u),\sigma(v)}+B\sum_{v\in V}\delta_{\sigma(v),1}\Bigg)$$ is the Potts partition function.  
\end{definition}

Fix another integer $d\ge 3$ and let $\ttT_d=\ttT_d(o)$ be the $d$-regular tree rooted at $o$. In this work, we focus on Potts measures on locally $\ttT_d$-like graphs, as defined next. For a finite graph $G=(V,E)$, a vertex $v\in V$, and $r\in\mathbb{N}$, let $\NB_r(G,v)$ be the subgraph of $G$ induced by all vertices within distance $r$ of $v$. 

\begin{definition}[Locally $\ttT_d$-like graphs]
    For a sequence of finite graphs $\{G_n\}$, let $v_n$ be a uniformly random vertex of $G_n$. We say that $\{G_n\}$ is \textit{uniformly sparse}, if 
    \[
    \lim_{L\to\infty}\lim_{n\to\infty}\mathbb{E}_{v_n\sim \operatorname{Uni}(V_n)}\big[\operatorname{deg}_{G_n}(v_n)\mathbf{1}\{\operatorname{deg}_{G_n}(v_n)>L\}\big]=0\,.
    \]
    
For a sequence of graphs $\{G_n\}_{n\ge1}$, we say that $G_n$ \emph{converges in the Benjamini-Schramm sense} (or converges locally) to $\ttT_d$ if the graphs are uniformly sparse and, for every $r\in\mathbb{N}$,
\[
\mathbb{P}_{v_n\sim\operatorname{Uni}(V_n)}\bigl[\operatorname{N}_r(G_n,v_n)\cong\operatorname{N}_r(\ttT_d,o)\bigr]\to1\,,
\quad\text{as }n\to\infty.
\]
By convention, the notation $
G_n \stackrel{\mathrm{loc}}{\longrightarrow} \ttT_d
$
means precisely that $\{G_n\}$ is a uniformly sparse sequence that converges locally to $\ttT_d$. We refer to such a sequence as \emph{locally} $\ttT_d$-\emph{like} graphs.
\end{definition}      

Given a sequence of graphs $G_n \stackrel{\operatorname{loc}}{\longrightarrow} \ttT_d$, we study the local weak limits of the Potts measures on $\{G_n\}$. Roughly speaking, local weak convergence of measures asserts that, in the large-graph limit, the measures locally resemble a limiting measure on $\ttT_d$ (see Definition~\ref{def-lwc-unified} for the precise formulation). Insights from statistical physics suggest that these local weak limits are governed by the dominant measure(s) emerging from the Bethe prediction (see Section~\ref{subsec-bg} for more background). This belief was fully-confirmed in the off-critical regime by \cite{BDS23}, where the dominant measure is unique. 

Our main interest lies in the critical regime: when $(\beta,B) \in \mathsf{R}_c$ (the Potts critical line as defined in \eqref{eq-R_c} below), there exist two different dominant measures, $\mu^{\free,\beta,B}$ and $\mu^{\wired,\beta,B}$, corresponding to the infinite Potts Gibbs measures on $\ttT_d$ with free and wired boundary conditions, respectively. It is expected that any local weak limit of $\mu_{G_n}^{\beta,B}$ is composed of mixtures of these two dominant measures, which has been proved in \cite{BDS23} for the case when $d$ is even. Even further, when the graphs are sufficiently well-connected, it is believed that the Gibbs measures $\mu_{G_n}^{\beta,B}$ can be well approximated by a mixture of the two dominant phases, leading to a stronger form of local weak convergence, known as \emph{local weak convergence in probability} (see Definition~\ref{def-lwc-unified}). This notion of well-connectedness is naturally formalized via graph expansion, which we now define.

\begin{definition}[Uniform edge-expander graphs]\label{def-expander}
    We say a sequence of graphs $G_n=(V_n,E_n)$ are \textit{uniform edge-expander} graphs, if for any $\varepsilon>0$, it holds that 
    \[
    \liminf_{n\to\infty}\inf_{\substack{S\subset V_n\\\varepsilon|V_n|\le |S|\le |V_n|/2}}\frac{|E_n(S;S^c)|}{|V_n|}>0\,,
    \]
where $E_n(S;S^c) \subset E_n$ is the set of edges with one endpoint in $S$, and another one in $S^c$.
\end{definition}
Assuming uniform edge-expansion of the sequence $\{G_n\}$, the strong form of local weak convergence of $\mu_{G_n}^{\beta,B}$ can be deduced from \cite{HJP23} for the case $B=0$ and $q \ge d^{400d}$. Our main contribution is to  present an alternative conceptual route to establishing the local weak convergence in probability of $\mu_{G_n}^{\beta,B}$ for general $(\beta,B) \in \mathsf{R}_c$ and arbitrary $d, q \ge 3$.

To state our main results, we first introduce some notations. For any two integers $d,q\ge 3$, following the notations in \cite{CH25}, we define the Potts critical line
\begin{align}\label{eq-R_c}
\mathsf R_c:=\{(\beta,B):\beta=\log(1+w_c(B)),0<B<B_+\}\,,
\end{align}
where
\begin{equation}\label{eq-w_c(B)}
w_c(B):=\ell_{q,d}((q-1)e^{-B})\,,\quad \ell_{q,d}(x)\equiv \frac{x^{2/d}-1}{1-\tfrac{1}{q-1}x^{2/d}}\,,
\end{equation}
and $B_+=B_+(d,q)>0$ is the unique positive solution of
\begin{equation}\label{eq-B_+}
(1+w_c(B))(1+w_c(B)/(q-1))=\left(\frac{d}{d-2}\right)^2\,.
\end{equation}
We refer to Definition~\ref{def-lwc-unified} for the precise definitions of local weak convergence and related terms.

\begin{theorem}\label{thm-main-main}
 Fix any two integers $d,q\ge 3$. Assume that $G_n\stackrel{\text{loc}}{\longrightarrow}\ttT_d$ are uniform edge-expander graphs. For $(\beta,B)\in \mathsf R_c$, it holds that any convergent subsequence of $\{\mu_{G_n}^{\beta,B}\}$ locally weakly converges in probability to a mixture of $\mu^{\free,\beta,B}$ and $\mu^{\wired,\beta,B}$. 
 In other words, any subsequential local weak limit point of $\mu_{G_n}^{\beta,B}$ is supported on a single measure in $\mathcal M$ as defined in \eqref{eq-mathcal-M} below. 
\end{theorem}
Inspired by \cite{BDS23}, our approach to proving Theorem~\ref{thm-main-main} relies on the Edwards-Sokal coupling that relates the Potts model with the random cluster model as defined in Definition~\ref{def-RCM}.
The proof of Theorem~\ref{thm-main-main} is built on an analogous characterization of local weak limits of the random cluster measure. Notably, our result applies not only when $q$ is an integer, but extends to all real $q>2$. For any $d\ge 3$ and $q>2$, define the random-cluster critical line
\begin{align}
\label{eq-R_c'}&\mathsf R_c':=\{(w,B):w=w_c(B),0<B<B_+\}\,,
\end{align}
where $w_c(B)$ and $B_+$ are defined as in \eqref{eq-w_c(B)} and \eqref{eq-B_+}, respectively. 

\begin{theorem}\label{thm-main-RCM}
     Fix any integer $d\ge 3$ and a real $q>2$. Assume that $G_n\stackrel{\text{loc}}{\longrightarrow}\ttT_d$ are uniform edge expander graphs. For $(w,B)\in \mathsf R_c'$, it holds that any convergent subsequence of $\{\varphi_{G_n^*}^{w,B}\}$ converges locally weakly in probability to a mixture of $\varphi^{\free,\beta,B}$ and $\varphi^{\wired,\beta,B}$. In other words, any subsequential local weak limit point of $\varphi_{G_n^*}^{w,B}$ is supported on a single measure in $\mathcal N$ as defined in \eqref{eq-mathcal-N} below. 
\end{theorem}

The proofs of Theorems~\ref{thm-main-main} and~\ref{thm-main-RCM} contain the primary conceptual contributions of this work. Our approach combines several distinct ideas: transitioning from the Potts model to the random cluster model to exploit monotonicity \cite{BDS23}; applying and extending the rank-2 approximation techniques from \cite{BBC23, CH25} to a local setting; analyzing a transformed Ising model via the Edwards-Sokal coupling; and establishing exponential deviation estimates for (FK-Ising) percolation clusters on expander graphs \cite{KLS20}. Importantly, none of these components can be used as a black box—each step in the argument involves additional technical innovations to adapt and integrate these tools effectively.

Knowing that every local weak limit of the critical Potts measures on $\{G_n\}$ is a mixture of the free and wired Gibbs measures, \cite{BDS23} further asked whether this mixture can be genuinely nontrivial. We give an affirmative—and stronger—answer: for any integers $d,q\ge3$, one can construct locally $\ttT_d$-like expander graphs on which the critical Potts measure exhibits any desired mixture weight between the disordered and ordered phases.
%\footnote{For the random cluster model, we only prove this result when the cluster parameter $q\in \mathbb N$. However, we believe the same claim holds for all real $q>2$.}

\begin{theorem}\label{thm-main}
    Fix any two integers $d,q\ge 3$. For any $(\beta,B)\in \mathsf R_c$ and any $\alpha\in [0,1]$, there exists a sequence of uniform edge-expander graphs $G_n\stackrel{\loc}{\longrightarrow} \ttT_d$ with $\operatorname{girth}(G_n)\to\infty$, such that $$\mu_{G_n}^{\beta,B}\stackrel{\operatorname{lwcp}}{\longrightarrow} \alpha{\mu^{\free,\beta,B}}+(1-\alpha){\mu^{\wired,\beta,B}}\,.$$
\end{theorem}
A crucial building block in the proof of Theorem~\ref{thm-main} is the analysis of local weak limits of Potts measures on random $d$-regular graphs. However, we find that typical instances of such graphs are insufficient for establishing the theorem. In fact, the main challenge in the proof lies in analyzing how certain modifications to a typical random $d$-regular graph affect its Potts partition function. A more detailed technical overview of these difficulties and our approach is provided in Section~\ref{subsec-main-results}.

\subsection{Background and Related Work}\label{subsec-bg}
In this subsection, we provide additional technical background and review relevant prior work.

\vspace{0.5\baselineskip}

\noindent\textbf{Bethe prediction.}  
Fix a sequence of graphs $G_n\stackrel{\loc}{\longrightarrow}\ttT_d$ and an integer $q\ge 3$. A fundamental question about the $q$-state Potts models on $\{G_n\}$ concerns the limiting free energy density  
\begin{equation}
\label{eq-free-energy-limit}
\Psi^{\Po}(\beta,B)
:=\lim_{n\to\infty}\frac{1}{|V_n|}\log Z_{G_n}^{\Po,\beta,B}\,.
\end{equation}
The explicit value of $\Psi^{\Po}(\beta,B)$ was first predicted by statistical-physics heuristics and later rigorously established.  Intuitively, since $G_n$ locally resembles $\ttT_d$, the free energy per site in $G_n$ coincides in the limit with the free energy at the root of $\ttT_d$.  Building on a well-believed replica-symmetric assumption, the latter can be computed explicitly via a cavity-method calculation, yielding a precise characterization of the limiting free energy density for the Potts model on $G_n$.

To be precise, fix parameters $\beta$ and $B$, and let $\Delta^*$ denote the set of probability measures on $[q]$.  The Bethe functional $\Psi: \Delta^* \to \mathbb{R}$ is defined by
\begin{equation}
\label{eq-bethe-functional}
\Psi(\nu)
:=\log\Biggl(\sum_{i=1}^q e^{B\delta_{i,1}}\Bigl(\sum_{j=1}^q e^{\beta\delta_{i,j}}\nu(j)\Bigr)^d\Biggr)
\;-\;\frac d2\log\Biggl(\sum_{i,j=1}^q e^{\beta\delta_{i,j}+B(\delta_{i,1}+\delta_{j,1})}\nu(i)\nu(j)\Biggr).
\end{equation}
The Bethe prediction asserts that, for any sequence satisfying $G_n\stackrel{\loc}{\longrightarrow}\ttT_d$, the limit in \eqref{eq-free-energy-limit} exists and is given by $
\Psi^{\Po}(\beta,B)
=\sup_{\nu\in\Delta^*}\Psi(\nu)$.

This prediction has now been completely verified by a series of works \cite{DM10, DMSS14, HJP23, CH25}. Specifically, \cite{DM10} first verifies the Bethe prediction for the Ising case $q=2$ (their results apply to more general locally tree-like graphs). After that, \cite{DMSS14} resolves the general $q\ge 3$ cases for $d\in 2\mathbb N$ relying on an analysis fact that is specific to even $d$. Later, \cite{HJP23} proves the prediction for the case of zero field and general $d$ when $q\ge d^{400d}$, using the method of cluster expansion. Finally, building on connections to the random cluster model and the Ising model on locally $\ttT_d$-like graphs developed in \cite{BBC23}, \cite{CH25} completes the proof of Bethe prediction for the most general case that $d,q\ge 3$ are arbitrary. 
\vspace{0.5\baselineskip}

\noindent\textbf{BP fixed points.}  
Another key contribution of \cite{DMSS14} is the identification of the maximum of the Bethe functional. To that end, define the belief propagation operator $\operatorname{BP}:\Delta^*\to\Delta^*$ by
\begin{equation}
\label{def:BP}
\operatorname{BP}(\nu)(i)\;\propto\;e^{B\delta_{i,1}}\Bigl(\sum_{j=1}^q e^{\beta\delta_{i,j}}\nu(j)\Bigr)^{d-1}
\quad\forall i\in[q].
\end{equation}
A measure $\nu\in\Delta^*$ satisfying $\operatorname{BP}(\nu)=\nu$ is called a \emph{BP fixed point}.  One checks by direct calculation that the interior stationary points of the Bethe functional $\Psi(\cdot)$ coincide exactly with these BP fixed points.

On the other hand, although there may be many BP fixed points, \cite{DMSS14} shows that the global maximum of $\Psi$ can only be one of two candidates. Let $\nu^{\free,\beta,B}$ and $\nu^{\wired,\beta,B}$ be the limit points obtained by iterating the BP map from the initial measures $\pi^{\free}=\mathrm{Unif}([q])$ and $\pi^{\wired}=\delta_1$, respectively. Clearly, both $\nu^{\free,\beta,B}$ and $\nu^{\wired,\beta,B}$ are BP fixed points. \cite{DMSS14} proves that these two measures are the only candidates for the global maximizer of $\Psi$, so the Bethe prediction reduces to
\begin{equation}\label{eq-Bethe-prediction-2}
\lim_{n\to\infty}\frac{1}{|V_n|}
\log Z_{G_n}^{\Po,\beta,B}
=\max\{\Psi(\nu^{\free,\beta,B}),\,\Psi(\nu^{\wired,\beta,B})\}.
\end{equation}
Notably, it might be true that $\nu^{\free,\beta,B}=\nu^{\wired,\beta,B}$, and we refer to such regime of the parameters $\beta,B$ as the uniqueness regime.\footnote{This uniqueness regime is less restrictive than the classical “strong uniqueness regime” for factor models on $\ttT_d$, as it only requires the infinite wired Gibbs measure to coincide with the infinite free Gibbs measure, rather than demanding that all infinite‐volume Gibbs measures agree.} 
In particular, when $(\beta,B)$ lies in the uniqueness regime, the limit in \eqref{eq-Bethe-prediction-2} simply equals
$
\Psi(\nu^{\free,\beta,B}) \;=\; \Psi(\nu^{\wired,\beta,B}).
$
\vspace{0.5\baselineskip}

\noindent\textbf{The Potts phase diagram.}  
An important feature that makes the Potts model significantly more challenging than the Ising model is its two-dimensional phase diagram, which captures both uniqueness vs.\ non-uniqueness and ordered vs.\ disordered transitions.  While the Ising model’s uniqueness regime covers the entire first quadrant $\{(\beta,B):\beta,B>0\}$, for $q\ge3$ the Potts model exhibits a nontrivial non-uniqueness regime.  Within this regime, the measures $\nu^{\free,\beta,B}$ and $\nu^{\wired,\beta,B}$ differ, giving rise to two subregimes—known as the disordered regime and the ordered regime—depending on whether $\nu^{\free,\beta,B}$ or $\nu^{\wired,\beta,B}$ maximizes~$\Psi$ (recall that $\Psi$ depends on $(\beta,B)$).

To make this more precise, fix integers $d,q\ge3$. There exists a threshold $B_+=B_+(d,q)>0$ (see \eqref{eq-B_+} below for the precise definition) such that for all $B\ge B_+$, $\nu^{\free,\beta,B}=\nu^{\wired,\beta,B}$. However, whenever $0<B<B_+$, there are three critical inverse‐temperature parameters $0<\beta_{\mathrm{ord}}(B)<\beta_{\operatorname{crit}}(B)<\beta_{\mathrm{dis}}(B)<\infty$, which delineate the uniqueness, disordered, critical, and ordered regimes.\footnote{It may appear counterintuitive from the notations that $\beta_{\mathrm{ord}}(B)<\beta_{\mathrm{dis}}(B)$. Here, $\beta_{\mathrm{ord}}(B)$ marks the emergence of the ordered phase, while $\beta_{\mathrm{dis}}(B)$ marks the disappearance of the disordered phase. Because one of these phases must dominate at any given temperature, it follows naturally that $\beta_{\mathrm{ord}}(B)<\beta_{\mathrm{dis}}(B)$.}
Specifically, we have 

\begin{enumerate}[label = (\roman*)]
\label{def:critical_temp}
    \item In the uniqueness regime,
    $\mathsf R_{=}:=\{(\beta,B):B\ge B_+\text{ or }0<B<B_+,\beta\in (0,\beta_{\operatorname{ord}}(B))\cup(\beta_{\operatorname{dis}}(B),\infty)\}$, we have $\nu^{\free,\beta,B}=\nu^{\wired,\beta,B}$.
    \item In the nonuniqueness regime, $\mathsf R_{\neq}:= (0,\infty)^2 \backslash \mathsf R_{=}$, we have $\nu^{\free,\beta,B}\neq\nu^{\wired,\beta,B}$. Moreover, the region $\mathsf R_{\neq}$ can be partitioned into three distinct subregions depending on whether $\nu^{\free,\beta,B}$ or $\nu^{\wired,\beta,B}$ dominates the functional $\Psi$, as follows:
    \begin{enumerate}[label=\alph*)]
        \item In the disordered regime, $ \mathsf R_{\free}:=\{(\beta,B):0<B<B_+,\beta_{\operatorname{ord}}(B)\le\beta<\beta_{\operatorname{crit}}(B)\}$, $\nu^{\free,\beta,B}$ dominates, i.e. $\Psi(\nu^{\free,\beta,B}) > \Psi(\nu^{\wired,\beta,B})$.
        \item In the ordered regime, $\mathsf R_{\wired}:=\{(\beta,B):0<B<B_+,\beta_{\operatorname{crit}}(B)<\beta\le\beta_{\operatorname{dis}}(B)\}$, $\nu^{\wired,\beta,B}$ dominates, i.e. $\Psi(\nu^{\wired,\beta,B}) > \Psi(\nu^{\free,\beta,B})$.
        \item On the critical line, $\mathsf R_{c}:=\{(\beta,B): \beta = \beta_{\operatorname{crit}}(B), 0 < B < B_+\}$, $\nu^{\wired,\beta,B}$ and $\nu^{\free,\beta,B}$ are both dominant, i.e. $\Psi(\nu^{\wired,\beta,B}) = \Psi(\nu^{\free,\beta,B})$.
    \end{enumerate}
    
\end{enumerate}

See Figure~\ref{fig:phase_diagram} for an illustration. In this terminology, $\mathsf R_{=}$ denotes the uniqueness regime, while $\mathsf R_{\free}$ and $\mathsf R_{\wired}$ correspond to the disordered and ordered regimes, respectively, in which the disordered or ordered phase dominates. Moreover, along the critical line $\mathsf R_c$, the two dominant phases coexist, giving rise to strong ordered/disordered phase coexistence, which we discuss later in the paper.
\vspace{0.5\baselineskip}

\begin{figure}
    \centering
    \includegraphics[width=0.6\linewidth]{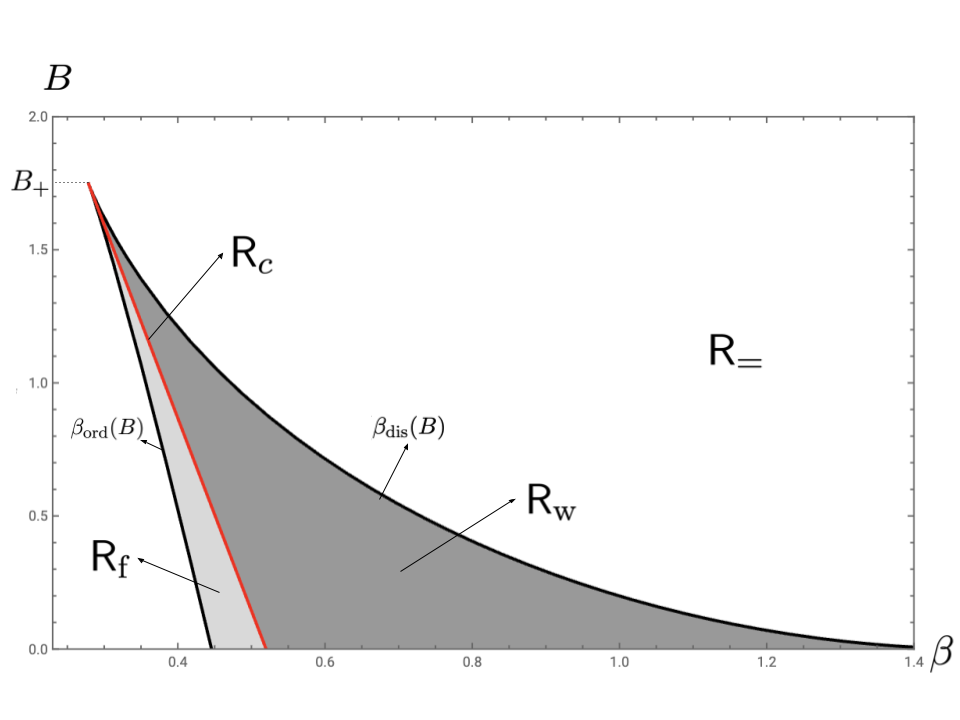}
    \caption{Phase diagram of the Potts model on $\ttT_d$ with $q = 45, d= 15$. The red line indicates the critical line $\mathsf R_c$, while the shaded region on the left (resp. right) is the disordered regime $\mathsf R_{\free}$ (resp. the ordered regime $\mathsf R_{\wired}$). The remaining white region is the uniqueness regime $\mathsf R_=$.}
    \label{fig:phase_diagram}
\end{figure}

\noindent\textbf{Connections to the random cluster model.}  
Through the well-known Edwards-Sokal coupling, the Potts model is closely related to the random cluster model, which we now define.

Fix a real parameter $q\ge1$, not necessarily integer.  For a finite graph $G=(V,E)$, let $v^*$ be a new ``ghost'' vertex and set  
\begin{equation}\label{eq-augmented-graph}
V^* = V\cup\{v^*\}, 
\qquad 
E^* = E\cup\bigl\{(v,v^*):v\in V\bigr\}.
\end{equation}
We write $G^*=(V^*,E^*)$ for this augmented graph in which $v^*$ is adjacent to every vertex of $G$.

\begin{definition}[The random cluster measure on a finite graph]\label{def-RCM}
Fix a graph $G=(V,E)$. For two parameters $w,B>0$, let $w_e=w$ for $e\in E$ and $w_e=e^B-1$ for $e\in E^*\setminus E$. The random cluster measure with weight parameters $w,B$ on $G^*$ is defined as follows:
    \[
    \varphi_{G^*}^{w,B}[\eta]:=\frac{1}{\mathcal Z_{G^*}^{\RC,w,B}}\prod_{e\in E^*}w_e^{\eta_e}\cdot q^{|\operatorname{C}(\eta)|-1}\,,\quad \eta\in \{0,1\}^{E^*}\,,
    \]
    where $\operatorname{C}(\eta)$ is the set of connected components on $V$ induced by the edge configuration $\eta$, and
\[
\mathcal Z_{G^*}^{\RC,w,B}:=\sum_{\eta\in \{0,1\}^{E^*}}\prod_{e\in E^*}w_e^{\eta_e} q^{|\operatorname{C}(\eta)|-1}
\]
is the random cluster partition function.\footnote{The definition is simply the extension of the common definition as appeared in \cite{CH25} to the augmented graph.}
\end{definition}
\begin{comment}
\begin{remark}\label{rmk-relation-varphi_G-varphi_G*}
A more common definition of the random cluster measure on $G$ is that 
\[
\varphi_G^{\beta,B}[\eta]:=\frac{1}{\mathcal Z_{G}^{\RC,w,B}}w^{|E(\eta)|}\prod_{C\in \operatorname{C}(\eta)}(1+(q-1)e^{-B|C|})\,,\quad\forall \eta\in \{0,1\}^E\,,
\]
where $E(\eta):=\{e\in E:\eta_e=1\}$, and $\mathcal Z_{G}^{\RC,w,B}$ is again the normalization constant. Indeed, $\varphi_{G}^{\beta,B}$ is exactly the marginal of $\varphi_{G^*}^{w,B}$ on $\{0,1\}^{E}$, and it is straightforward to check that
    \begin{equation}\label{Z_{G}-Z_{G^*}-relation}
        \quad \mathcal Z_{G^*}^{\RC,w,B}=e^{B|V|}\cdot \mathcal Z_G^{\RC,w,B}\,.
    \end{equation}
    We opt to focus on the measure $\varphi_{G^*}^{\RC,\beta,B}$ as it enjoys a certain domain Markov property. 
\end{remark}
\end{comment}

There is an analogous Bethe prediction for the free‐energy density of the random cluster model on locally $\ttT_d$‐like graphs.  In particular, when $q\ge2$ is an integer, the random cluster Bethe prediction corresponds directly to that of the Potts model via the Edwards–Sokal coupling (see Section~\ref{subsec-ES-coupling}).\footnote{In fact, parts of the Potts Bethe prediction were first established by proving the corresponding statements for the random cluster model and then transferring them via the Edwards–Sokal coupling.}  For noninteger $q>2$ with $B=0$, \cite{BBC23} introduced a rank‐2 approximation of the Potts partition function to characterize the limiting free‐energy density.  This method was extended in \cite{CH25} to $B>0$, yielding the free‐energy density of the random cluster model with an external field on locally $\ttT_d$‐like graphs.  That work also uncovers a first‐order phase transition for $q>2$ and provides an explicit description of the critical line $\mathsf R_c'$ (see \eqref{eq-R_c'} below).  
\vspace{0.5\baselineskip}

\noindent\textbf{Local weak convergence of measures.}  
We now turn to the main theme of this paper. Let $G_n=(V_n,E_n)$ be a sequence of graphs satisfying $G_n\stackrel{\loc}{\longrightarrow}\ttT_d$. Beyond the Bethe prediction for the limiting free‐energy density, statistical‐physics heuristics further suggest that the Potts measure $\mu_{G_n}^{\beta,B}$ (resp. the random cluster measure $\varphi_{G_n^*}^{w,B}$) locally resembles certain Gibbs measures on $\ttT_d$. This is formalized through the local weak convergence of measures. To state the precise definition, we first introduce some notations.

For simplicity, write $\ttT_d=(\ttV,\ttE)$ and, for each $r\in\mathbb{N}$, denote its $r$-neighborhood $\NB_r(\ttT_d,o)$ by $(\ttV_r,\ttE_r)$.  Likewise, let $
\ttT_d^* = (\ttV\cup\{v^*\},\,\ttE^*)$ and $
\NB_r^*(\ttT_d,o) = (\ttV_r\cup\{v^*\},\,\ttE_r^*)
$
be the corresponding augmented graphs.  Fix two finite sets $\mathcal X$ and $\mathcal Y$, and let 
$
\mathcal P(\mathcal X^{\ttV}\times\mathcal Y^{\ttE^*})$ (resp. $\mathcal P(\mathcal X^{\ttV_r}\times\mathcal Y^{\ttE_r^*})$)
denote the space of probability measures on $\mathcal X^{\ttV}\times\mathcal Y^{\ttE^*}$ (resp.\ $\mathcal X^{\ttV_r}\times\mathcal Y^{\ttE_r^*}$), each equipped with the weak topology.  For any probability measure $\mathtt p$ on $\mathcal P(\mathcal X^{\ttV}\times\mathcal Y^{\ttE^*})$ and $r\in\mathbb{N}$, let $\mathtt p_r$ be its projection measure  on $\mathcal P(\mathcal X^{\ttV_r}\times\mathcal Y^{\ttE_r^*})$.

\begin{definition}[Local weak convergence of measures]\label{def-lwc-unified}
For $G_n=(V_n,E_n)$ satisfying $G_n\stackrel{\loc}{\longrightarrow}\ttT_d$, let $G_n^*=(V_n\cup\{v^*\},E_n^*)$ be the augmented graph. For $n,r\in\mathbb{N}$ and $v\in V_n$, denote $\NB_r(G_n,v)=(V_{n,r}(v),E_{n,r}(v))$ and let $\NB_r^*(G_n,v)=(V_{n,r}(v)\cup\{v^*\},E_{n,r}^*(v))\subset G_n^*$ be its augmented graph. Let $\lambda_n$ be a probability measure on $\mathcal X^{V_n}\times\mathcal Y^{E_n^*}$. For each $v\in V_n$, let $\mathbf{P}_{n,r}(v)$ be the distribution of $(\NB_r(G_n,v),\sigma\!\mid_{V_{n,r}(v)},\eta\!\mid_{E_{n,r}^*(v)})$ when $(\sigma,\eta)\sim\lambda_n$.

Let $v_n$ be a uniformly random vertex of $G_n$. We say that $\lambda_n$ \emph{converges locally weakly} to a probability measure $\mathtt p$ on $\mathcal P(\mathcal X^{\ttV^*}\times\mathcal Y^{\ttE^*})$, written $\lambda_n\stackrel{\mathrm{lwc}}{\longrightarrow}\mathtt p$, if for every $r\in\mathbb{N}$ the law of $\mathbf{P}_{n,r}(v_n)$ converges weakly to $\delta_{\NB_r(\ttT_d,o)}\otimes\mathtt p_r$ as $n\to\infty$. Moreover, if $\mathtt p=\delta_\lambda$ for some $\lambda\in\mathcal P(\mathcal X^{\ttV^*}\times\mathcal Y^{\ttE^*})$, we say $\lambda_n$ \emph{converges locally weakly in probability} to $\lambda$, denoted $\lambda_n\stackrel{\mathrm{lwcp}}{\longrightarrow}\lambda$. Note that $\lambda_n\stackrel{\operatorname{lwcp}}{\longrightarrow}\lambda$ is equivalent to that
\begin{equation}\label{eq-characterization-lwcp}
\forall r\in\mathbb N\,,\varepsilon>0\,,\quad
\lim_{n\to\infty}\mathbb{P}_{v_n\sim\mathrm{Unif}(V_n)}\bigl[\mathrm{TV}(\mathbf{P}_{n,r}(v_n),\delta_{\NB_r(\ttT_d,o)}\otimes\lambda\!\mid_{[q]^{\ttV_r}\times \{0,1\}^{\ttE_r^*}})\ge\varepsilon\bigr]=0.
\end{equation}
\end{definition}

\begin{remark}
    There are two sources of randomness in the above definition: (i) the sampling of a configuration $(\sigma,\eta)\sim \lambda_n$, and (ii) the sampling of a uniform vertex $v_n\in V_n$. For a fixed $v\in V$, $\mathbf{P}_{n,r}(v)$ concerns source (i) of randomness, whereas ``the law of $\mathbf{P}_{n,r}(v_n)$'' involves that of source (ii). We refer to \cite[Definition 2.3]{MMS12} for a more detailed discussion on this.
\end{remark}

The above definition provides a unified framework for local weak convergence of the Potts and random cluster measures by choosing $\mathcal X=[q]$, $\mathcal Y=\{\mathbf{o}\}$ for the former and $\mathcal X=\{\mathbf{o}\}$, $\mathcal Y=\{0,1\}$ for the latter. Moreover, Definition~\ref{def-lwc-unified} will prove especially useful once we introduce the Edwards–Sokal coupling between the Potts and random cluster measures (see Section~\ref{subsec-ES-coupling}).

It is known that any sequence of Potts or random cluster measures on locally $\ttT_d$‐like graphs admits a locally weakly convergent subsequence (see, e.g., \cite[Lemma 2.9]{BDS23}). We now define the candidate local weak limit points of these measures, beginning with the Potts case. Recall that $\pi^{\free}=\mathrm{Unif}([q])$ and $\pi^{\wired}=\delta_1$. For each $r\ge1$, consider the Potts model on $\NB_r(\ttT_d,o)=(\ttV_r,\ttE_r)$ with free or wired boundary conditions on $\partial\ttV_r$. For $\dagger\in\{\free,\wired\}$, the Gibbs measure $\mu_{r}^{\dagger,\beta,B}$ on configurations $\sigma\in[q]^{\ttV_r}$ is given by
\begin{equation*}
\label{def:free_wired_gibbs}
\mu_{r}^{\dagger,\beta,B}[\sigma]
\;\propto\;
\exp\Biggl(\beta\sum_{(u,v)\in \ttE_r}\delta_{\sigma(u),\sigma(v)}
+ B\sum_{v\in \ttV_{r-1}}\delta_{\sigma(v),1}\Biggr)
\prod_{u\in \partial\ttV_r}\pi^{\dagger}(\sigma(u))\,.  
\end{equation*}
Additionally, let $\check{\nu}_{r}^{\dagger,\beta,B}$ denote the marginal of $\mu_{r}^{\dagger,\beta,B}$ at the root, i.e.\ on $\sigma(o)$, for $\dagger\in\{\free,\wired\}$. 

It is well known that as $r\to\infty$, $\mu_{r}^{\dagger,\beta,B}$ converges weakly to a limiting measure $\mu^{\dagger,\beta,B}$ on $[q]^{\ttV}$. Consequently, $\check{\nu}_{r}^{\dagger,\beta,B}$ converges weakly to the marginal $\check{\nu}^{\dagger,\beta,B}$ of $\mu^{\dagger,\beta,B}$ on $\sigma(o)$. One checks that $\check{\nu}^{\dagger,\beta,B}\in\Delta^*$ satisfies
\begin{equation}\label{eq-check-nu}
\check{\nu}^{\dagger,\beta,B}(i)\;\propto\;e^{B\delta_{i,1}}\Bigl(\sum_{j=1}^q e^{\beta\delta_{i,j}}\nu^{\dagger}(j)\Bigr)^{d},
\quad\forall\,i\in[q].
\end{equation}

It is natural to expect that if $G_n\stackrel{\loc}{\longrightarrow}\ttT_d$, then the Potts measure $\mu_{G_n}^{\beta,B}$ converges locally weakly in probability to the infinite Potts Gibbs measure on $\ttT_d$ corresponding to the dominant phase.  More precisely, recalling the three regions $\mathsf R_{=}$, $\mathsf R_{\free}$, and $\mathsf R_{\wired}$ in the Potts phase diagram: If $(\beta,B)\in\mathsf R_{=}$, then $\mu^{\free,\beta,B}=\mu^{\wired,\beta,B}$, and one expects  
  \begin{equation}\label{eq-lwc-R_=}
  \mu_{G_n}^{\beta,B}\stackrel{\mathrm{lwcp}}{\longrightarrow}\mu^{\free,\beta,B}=\mu^{\wired,\beta,B}\,,
  \end{equation}
  and if $(\beta,B)\in\mathsf R_{\dagger}$ for $\dagger\in\{\free,\wired\}$, then $\mu^{\dagger,\beta,B}$ is the unique dominant measure, and one expects  
  \begin{equation}\label{eq-lwc-R_dagger}
  \mu_{G_n}^{\beta,B}\stackrel{\mathrm{lwcp}}{\longrightarrow}\mu^{\dagger,\beta,B}.
  \end{equation}
Indeed, \eqref{eq-lwc-R_=} and \eqref{eq-lwc-R_dagger} were proved in \cite{MMS12, BDS23}, with \cite{MMS12} handling the Ising case $q=2$ and \cite{BDS23} covering the general case $q\ge3$.  

There is a similar story for the random cluster measures on locally $\ttT_d$-like graphs. First, one defines the infinite free and wired random cluster measures on $\ttT_d$, denoted by $\varphi^{\free,w,B}$ and $\varphi^{\wired,w,B}$, respectively (see Section~\ref{subsec:stoch_dom} for details). Let $\mathsf R'_{=}$, $\mathsf R'_{\free}$, $\mathsf R'_{\wired}$, and $\mathsf R'_c$ be the corresponding uniqueness regime, disordered regime, ordered regime and critical line of the random cluster model on locally $\ttT_d$-like graphs.\footnote{When $q\in \mathbb{N}$, by the Edwards-Sokal coupling, the four regimes coincide with the images of the regimes $\mathsf R_{=}$, $\mathsf R_{\free}$, $\mathsf R_{\wired}$, and $\mathsf R_c$ under the mapping $(\beta,B)\mapsto(e^\beta-1,B)$.} It is shown in \cite{BDS23} that when $(w,B)\in\mathsf R'_{=}$, 
$\varphi_{G_n^*}^{w,B}\stackrel{\mathrm{lwcp}}{\longrightarrow}\varphi^{\free,w,B}=\varphi^{\wired,w,B}$,
and when $(w,B)\in\mathsf R'_{\dagger}$ for $\dagger\in\{\free,\wired\}$, $
\varphi_{G_n^*}^{w,B}\stackrel{\mathrm{lwcp}}{\longrightarrow}\varphi^{\dagger,w,B}$.

Along the critical lines, however, as there are two distinct dominant phases, the situations become more subtle in both the Potts and the random cluster models.  It is conjectured that on the critical line $\mathsf R_c$ (resp. $\mathsf R_c'$), the local weak limit of $\mu_{G_n}^{\beta,B}$ (resp. $\varphi_{G_n^*}^{w,B}$) is a mixture of the two dominant Gibbs measures.  Partial confirmations appear in \cite{HJP23, Shr23, BDS23}.  In particular, leveraging the same even-$d$ analysis from \cite{DMSS14}, \cite{BDS23} proves that for $d\in 2\mathbb N$ and $(\beta,B)\in\mathsf R_c$, every subsequential local weak limit of $\mu_{G_n}^{\beta,B}$ (viewed as a random element of $\mathcal P([q]^{\ttV})$) is supported on  
\begin{equation}\label{eq-mathcal-M}
\mathcal M
:=\{\alpha\,\mu^{\free,\beta,B}+(1-\alpha)\,\mu^{\wired,\beta,B},\alpha\in[0,1]\}\,. 
\end{equation}
Similarly, for $d\in 2\mathbb N,q\ge 3$\footnote{\cite{BDS23} only states the result for $q\in \mathbb N$, but possibly their proof can be extended to real $q>2$.} and $(w,B)\in \mathsf R_c'$, any subsequential local weak limit of $\varphi_{G_n^*}^{w,B}$ is supported on the set
\begin{equation}\label{eq-mathcal-N}
\mathcal N := \{\alpha\,\varphi^{\free,w,B} + (1-\alpha)\,\varphi^{\wired,w,B} : \alpha\in[0,1]\}\,.
\end{equation}
A comparable result for $B=0$ on random $d$-regular graphs was obtained in \cite{Shr23}.

In general, the above is the best possible description. For example, if each $G_n$ is a disjoint union of $O(1)$ macroscopic components, then any subsequential local weak limit of $\mu_{G_n}^{\beta,B}$ (resp.\ $\varphi_{G_n^*}^{w,B}$) may indeed be supported on multiple points in $\mathcal{M}$ (resp.\ $\mathcal{N}$). However, when the graphs $G_n$ are well connected, e.g. if they are expander graphs, one expects that every convergent subsequence of $\mu_{G_n}^{\beta,B}$ (resp.\ $\varphi_{G_n^*}^{w,B}$) actually \emph{converges locally weakly in probability}—i.e., the limit is supported on a single point in $\mathcal{M}$ (resp.\ $\mathcal{N}$). Even stronger, it is believed that for these graphs, the Potts and random cluster measures decompose into two macroscopic phases—free-like and wired-like—but a typical configuration sampled from the Gibbs measure is locally uniform and resembles either entirely the free phase or entirely the wired phase.

The stronger form of local weak convergence in probability for Potts and random cluster measures on expander graphs was partially established in \cite{HJP23}. Under the assumptions that $G_n\stackrel{\mathrm{loc}}{\longrightarrow}\ttT_d$ are uniform edge-expander graphs and $q\ge d^{400d}$, \cite{HJP23} shows that at the critical weight $w=w_c$, every subsequential local weak limit of $\varphi_{G_n^*}^{w,0}$ is supported on a single measure in $\mathcal N$. More precisely, conditioning on the disordered and the ordered phases, $\varphi_{G_n^*}^{w,0}$ converges locally weakly in probability to $\varphi^{\free,w,0}$ and $\varphi^{\wired,w,0}$, respectively, mirroring the conditional convergence results for the zero-field Ising and Potts models on expander graphs (see \cite[Theorem 2.4]{MMS12} and \cite[Theorem 1.18]{BDS23}). Although not stated explicitly, their methods likewise imply the same strong local weak convergence of the zero-field Potts model whenever $q$ is sufficiently large relative to $d$. Prior to our work, the general question of strong local weak convergence for the Potts and random cluster measures—particularly in the presence of an external field or for arbitrary values of $d, q \ge 3$—remained open. Our results, Theorems~\ref{thm-main-main} and~\ref{thm-main-RCM}, resolve these gaps.

\vspace{0.5\baselineskip}

\noindent\textbf{Phase coexistence.}  
For the Potts model on locally $\ttT_d$-like graphs, when $(\beta,B)$ lies outside the uniqueness regime, two distinct phases—the disordered and the ordered ones—coexist.  Owing to its fundamental physical significance and intriguing algorithmic implications, this phase coexistence phenomenon has been studied in \cite{GSVY16, HJP23, CGGSV23} for details.

Of particular interest is the critical line $\mathsf R_c$, along which the ordered and disordered phases are both dominant. In this regime, the model is expected to exhibit strong phase coexistence—that is, both phases occupy a nontrivial fraction of the Potts Gibbs measure. This phenomenon is tied to the local weak limit of the Potts measures being a genuine mixture of the free and wired Gibbs measures, with mixing weights governed by the free and wired Potts partition functions—i.e., the subsums of the Potts partition function over configurations in the disordered and ordered phases.
 The partial partition functions were first analyzed in \cite{GSVY16}, who proved a weak form of strong coexistence on random $d$-regular graphs by showing that the ordered and disordered partial partition functions agree up to a logarithmic factor. In \cite{HJP23}, the authors establish true strong phase coexistence for the critical Potts model without external field on random $d$-regular graph--under the assumption that $q\ge d^{400d}$.  They also obtain the precise asymptotics of the free and wired Potts partition functions explicitly in terms of short‐cycle counts in the underlying graph.  

Despite these advances, prior to our result Theorem~\ref{thm-main}, it remained unclear whether the critical Potts model on random $d$-regular graphs—or more generally, on locally $\ttT_d$-like graphs—exhibits strong phase coexistence for arbitrary $d, q \ge 3$. Our result settles this question by characterizing the full set of possible local weak limits of Potts measures on locally $\ttT_d$-like expander graphs.

\subsection{Proof overview}\label{subsec-main-results}
In this subsection, we outline the proof of our main theorems.  We first sketch the proof of the local weak convergence in probability. By the Edwards-Sokal coupling, Theorem~\ref{thm-main-main} follows almost directly from Theorem~\ref{thm-main-RCM}, so we focus on the random cluster case.

Fix $(w,B)\in\mathsf R_c'$. Let $\psi^{\free}$ and $\psi^{\wired}$ denote the probabilities that the root $o$ of $\ttT_d$ is connected to $v^*$ under $\varphi^{\free,w,B}$ and $\varphi^{\wired,w,B}$, respectively. We show that for uniform edge‐expander graphs $G_n\stackrel{\mathrm{loc}}{\longrightarrow}\ttT_d$, a typical sample $\eta\in\{0,1\}^{E_n^*}$ from $\varphi_{G_n^*}^{w,B}$ has its giant component $C_*(\eta)$—the component containing $v^*$—of size close to either $\psi^{\free}|V_n|$ or $\psi^{\wired}|V_n|$ when $n$ is large.  Once this is established, the theorem follows by standard arguments: since $\varphi^{\free,w,B}$ and $\varphi^{\wired,w,B}$ are the minimal and maximal elements (under the natural stochastic ordering\footnote{This ordering is the main reason for passing from the Potts model to the random cluster model.}) among all subsequential local weak limits of $\varphi_{G_n^*}^{w,B}$, one deduces via stochastic domination that, conditioning on $|C_*(\eta)|\approx\psi^{\dagger}|V_n|$, the measure $\varphi_{G_n^*}^{w,B}$ converges locally weakly in probability to $\varphi^{\dagger,w,B}$ for each $\dagger\in\{\free,\wired\}$, thereby completing the proof of Theorem~\ref{thm-main-RCM}.

Turning to the proof of the claim, we show that for any $\psi\notin\{\psi^{\free},\psi^{\wired}\}$, $\varphi_{G_n^*}^{w,B}\bigl[\,|C_*|\approx \psi\,|V_n|\,\bigr]\to0$ as $n\to\infty$. By definition,
\begin{equation}\label{eq-goal-probability}
\varphi_{G_n^*}^{w,B}\bigl[\,|C_*|\approx \psi\,|V_n|\,\bigr]
=\frac{\sum_{\eta:\,|C_*(\eta)|\approx \psi|V_n|}\;\prod_{e\in E_n^*}w_e^{\eta_e}\,q^{|\operatorname{C}(\eta)|-1}}
{\mathcal Z_{G_n}^{\RC,w,B}}\,.
\end{equation}
The key input is the rank-2 approximation from \cite{BBC23,CH25}, which approximates $\mathcal Z_{G_n}^{\RC,w,B}$ by the partition function of a suitable two-spin model on $G_n$. We also establish a local rank-2 approximation, giving a corresponding approximation for the numerator in \eqref{eq-goal-probability} via a partial partition function of the same two-spin model. As observed in \cite{SS14}, any two-spin system can be transformed into an Ising model. Combining these steps, we reduce \eqref{eq-goal-probability} to the probability that the magnetization of a certain (ferromagnetic) Ising model on $G_n$ lies near an atypical value. Since each approximation incurs only an $\exp(o(|V_n|))$ multiplicative error, it remains to show that this Ising probability decays like $\exp(-\Theta(|V_n|))$, which completes the argument.

The critical condition $(w,B)\in\mathsf R_c'$ implies that the corresponding Ising model has no external field and lies in the non‐uniqueness regime. Although it is shown in \cite{MMS12} that for the zero‐field Ising model on expander graphs in this regime, the normalized magnetization concentrates on two typical values with high probability, the tightness arguments in \cite{MMS12} yield only an abstract $o(1)$ bound on the deviation probability.\footnote{We also note that a large‐deviation estimate for the normalized magnetization was obtained in the uniqueness regime by \cite{Can19}, but those techniques seem difficult to adapt to our setting, as the expander condition must be exploited here.}
 We instead obtain exponential deviation estimates by exploiting the Edwards–Sokal coupling. The problem then reduces to analyzing the asymptotics and fluctuations of large FK‐Ising percolation clusters on expander graphs, which we tackle using combinatorial‐probabilistic methods that are largely inspired from \cite{KLS20}. This concludes the proof of Theorems~\ref{thm-main-main} and \ref{thm-main-RCM}.

We remark that, although Theorems~\ref{thm-main-main} and \ref{thm-main-RCM} are stated for $B>0$, we expect the same conclusions to hold at zero field.  The $B>0$ assumption simplifies the control of non-giant clusters in $\eta\sim\varphi_{G_n^*}^{w,B}$ (see Lemma~\ref{lem-cluster-size}), but it is not essential.  In the case $B=0$, one can still obtain sufficient bounds on the cluster sizes in $\eta\sim\varphi_{G_n^*}^{w,0}$.  In particular, combining the arguments of Section~\ref{subsec-large-deviation} with the conditional local weak convergence results for zero‐field Potts models from \cite{BDS23} should yield the desired extensions.

To prove Theorem~\ref{thm-main}, we start with random $d$-regular graphs, which are locally $\ttT_d$-like expander graphs.  Adapting the small-graph conditioning method of \cite{GSVY16}, we obtain precise asymptotics for the free and wired Potts partition functions in terms of short-cycle counts (extending the large-$q$ results of \cite{HJP23}). However, numerical evidence shows that on typical random $d$-regular graphs the disordered phase retains weight bounded away from zero in every local weak limit. To complete the proof, we further construct modified graphs—by adding and deleting carefully chosen edges—to tune the ratio of free to wired partition functions via a cavity-method calculation.  Applying this modification procedure to appropriately sampled random $d$-regular graphs yields graph sequences realizing all mixture weights, completing the proof.  The same construction extends directly to the case $B=0$.

\vspace{0.5\baselineskip}

\noindent\textbf{Paper organization.}  
In Section~\ref{sec-pre} we collect the necessary preliminaries.  Section~\ref{sec-local-weak-limit} is devoted to the proofs of Theorems~\ref{thm-main-main} and \ref{thm-main-RCM}, and Section~\ref{sec-phase-coexistence} contains the proof of Theorem~\ref{thm-main}.  Proofs of some technical results are deferred to the appendix to preserve the flow of the presentation.

\vspace{0.5\baselineskip}

\noindent\textbf{Acknowledgement.} We warmly thank Amir Dembo for introducing this problem. We are also grateful to Anirban Basak, for sharing simulation codes, and to Amir Dembo, Allan Sly and Nike Sun, for having stimulating discussions and providing many helpful comments on an early draft of this paper. H. Du is partially supported by an NSF-Simons research collaboration grant (award number 2031883).

%Throughout the proof of our main results, we also obtain several byproducts that we believe to be of independent interest. The additional results are summarized as follow:\\
%(i) We consider the FK-Ising model on locally $\ttT_d$-like expander graphs and study the asymptotic behavior of its large components. We show that in the super-critical regime, with high probability there is a unique giant component and we characterize its asymptotic size. Moreover, we prove an exponential bound on the deviation probability of the giant component size.\\
%(ii) We study the marginal distribution of Potts measures on locally $\ttT_d$-like expander graphs. For any $m=O(1)$, we show that for $m$ ``typical'' vertices $v_1,\dots,v_m$, the Potts marginal on $(\sigma(v_1),\dots,\sigma(v_m))$ is nearly a mixture of $(\check{\nu}^{\free})^{\otimes p}$ and $(\check{\nu}^{\wired})^{\otimes p}$. We also obtain similar results for the random cluster measure.\\
%(iii) For random $d$-regular graphs, by employing the small graph conditioning method, we obtain precise asymptotics of the Potts free and wired partition functions in terms of the short cycle counting statistics. This characterizes the mixture ratio of the free and wired Potts Gibbs measures in the critical Potts model on random $d$-regular graphs, extending known results in \cite{HJP23} for large $q$. 
%More discussions on the additional results as well as their connection to the main results are provided in Section~\ref{subsec-proof-overview}.

\section{Preliminaries}\label{sec-pre}
In this section we provide necessary preliminaries of the proof of our main results. Throughout, we fix $d\ge 3,q>2$, and we do not assume $q$ is an integer when our consideration is restricted to the random cluster model. We also fix three parameters $\beta,w,B>0$ with $w=e^\beta-1$. We omit the dependence of $\beta,w,B$ in the notations when it is clear from the context.

\subsection{The Edwards-Sokal coupling}\label{subsec-ES-coupling}
We start with the well-known Edwards-Sokal coupling. Fix $\beta,B>0$.
For a finite graph $G=(V,E)$, let $G=(V^*,E^*)$ be its augmented graph defined as in \eqref{eq-augmented-graph}. For an edge $e\in E^*$, we write 
\begin{equation}\label{eq-pe}
    p_e:=\begin{cases}
        1-e^{-\beta}\,,&\text{ if }e\in E\,,\\
        1-e^{-B}\,,&\text{ if }e\in E^*\setminus E\,.
    \end{cases}
\end{equation}

\begin{definition}[The Edwards-Sokal measure on a finite graph]
\label{def:ES_measure}
    For a finite graph $G=(V,E)$, let $G^*=(V^*,E^*)$ be its augmented graph and $\varpi=\varpi_{G^*}^{\beta,B}$ be the Edwards-Sokal (ES) measure on $[q]^{V}\times \{0,1\}^{E^*}$, such that for any $(\sigma,\eta)\in [q]^{V}\times \{0,1\}^{E^*}$ (we use the convention that $\sigma(v^*)=1$ and $0^0=1$),
    \[
\varpi(\sigma,\eta)\propto \prod_{e=(u,v)\in E^*}(1-p_{e})^{1-\eta_{e}}(p_e\delta_{\sigma(u),\sigma(v)})^{\eta_e}\,.
    \]
\end{definition}
Under the Edwards–Sokal (ES) coupling, the marginal of $\varpi$ on $\sigma\in[q]^V$ coincides with the $q$-state Potts measure $\mu_G^{\beta,B}$, while its marginal on $\eta\in\{0,1\}^{E^*}$ coincides with the random cluster measure $\varphi_{G^*}^{w,B}$.  A direct consequence is the following relation between their partition functions.
\begin{lemma}\label{lem-Potts-RCM-partition-function}
    For any finite graph $G=(V,E)$ and parameters $\beta,w,B>0$ such that $w=e^\beta-1$, 
    \begin{equation}
        \label{eq-Potts-RCM-partition-function}
            \mathcal Z_{G}^{\Po,\beta,B}=\mathcal Z_{G^*}^{\RC,w,B}\,.
    \end{equation}
\end{lemma}
\begin{proof}
    Recall $p_e$ as defined in \eqref{eq-pe}. We consider the sum (using the same conventions as before)
    \begin{equation}\label{eq-sum-ES}
    \sum_{\sigma\in [q]^{V}}\sum_{\eta\in \{0,1\}^{E^*}}\prod_{e=(u,v)\in E^*}(1-p_{e})^{1-\eta_{e}}(p_e\delta_{\sigma(u),\sigma(v)})^{\eta_e}\,.
    \end{equation}
    On the one hand, by fixing $\sigma$ and summing over $\eta$, \eqref{eq-sum-ES} becomes
    \[
    \sum_{\substack{\sigma\in [q]^{V}}}\prod_{\substack{e=(u,v)\in {E^*}\\\sigma(u)\neq \sigma(v)}}(1-p_e)=\sum_{\sigma\in [q]^{V}}\exp\Bigg(-\beta\sum_{(u,v)\in E}(1-\delta_{\sigma(u),\sigma(v)})-B\sum_{v\in V}(1-\delta_{\sigma(v),1})\Bigg)\,,
    \]
    which equals $e^{-\beta|E|-B|V|}\mathcal Z_{G}^{\Po,\beta,B}$ by definition. On the other hand, if we exchange the ordering of the sum and fix $\eta$ and sum over $\eta$, we see \eqref{eq-sum-ES} also equals
    \[
    \sum_{\eta\in\{0,1\}^{E^*}}\prod_{e\in E^*}p_e^{\eta_e}(1-p_e)^{1-\eta_e}\cdot q^{|\operatorname{C}(\eta)|-1}=\prod_{e\in E^*}(1-p_e)\cdot \sum_{\eta\in\{0,1\}^{E^*}}\prod_{e\in E^*}w_e^{\eta_e}\cdot q^{|\operatorname{C}(\eta)|-1}\,,
    \]
    where $w_e=\tfrac{p_e}{1-p_e}$ for $e\in E^*$, meaning that $w_e=e^\beta-1=w$ for $e\in E$ and $w=e^B-1$ for $e\in E^*$. Thus, the above equals $e^{-\beta|E|-B|V|}\mathcal Z_{G^*}^{\RC,w,B}$. Altogether it yields \eqref{eq-Potts-RCM-partition-function}.
\end{proof}
Furthermore, the ES measure has the following property that makes it particularly useful.
\begin{proposition}\label{prop-ES-independent-smapling}
    For any finite graph $G$, under the ES coupling $\varpi$, conditioned on any particular realization of $\eta$, the spin configuration $\sigma$ can be sampled as follows: Let $\operatorname{C}(\eta)=\{C_*,C_1,\dots,C_\ell\}$ be the set of connected components induced by the edge configuration $\eta$ with $v^*\in C_*$. Let $c_*=1$ and $c_1,\dots,c_\ell \in [q]$ be chosen uniformly and independently at random, assign $\sigma(i)=c_k$ for any $k\in \{*,1,\dots,\ell\}$ and $i\in C_k$. 
\end{proposition}
The proof of Proposition~\ref{prop-ES-independent-smapling} can be found in, e.g., \cite{duminil2017lectures}. In particular, we have the following easy corollary.
\begin{corollary}\label{cor-Potts-RC}
    For graph $G$ with $G^*=(V^*,E^*)$ and parameters $\beta,w,B>0$ such that $w=e^\beta-1$, let $\mu_G=\mu_G^{\beta,B},\varphi_{G^*}=\varphi_{G^*}^{w,B}$ be the Potts and random cluster measures on $G,G^*$, respectively. For and any $u,v\in V^*$ it holds (let $\sigma(v^*)=1$, $\mu_G$-a.s.), 
    \[
    \mu_G\big[\sigma(u)=\sigma(v)\big]=\frac{q-1}{q}\varphi_{G^*}[u\leftrightarrow v]+\frac{1}{q}\,.
    \]
    Here, $u\leftrightarrow v$ denotes for that $u,v$ are connected in $\eta$, i.e. $u,v$ belong to the same cluster in $\operatorname{C}(\eta)$. 
\end{corollary}
\begin{proof}
    By Proposition~\ref{prop-ES-independent-smapling} we have in the ES coupling, for any $u,v\in V^*$, given that $u\leftrightarrow v$ in $\eta$, then $\sigma(u)=\sigma(v)$ almost surely; otherwise $\sigma(u)=\sigma(v)$ happens with conditional probability $q^{-1}$. Therefore, we have
    \begin{align*}
\mu_G\big[\sigma(u)=\sigma(v)\big]=&\ \varpi_G\big[\sigma(i)=\sigma(j)\big]=\varpi_G[\{\eta:u{\leftrightarrow}v\}]+\frac{1}{q}\varpi\big[\{\eta:u\not\leftrightarrow v\}\big]\\
=&\ \varphi_{G^*}[u\leftrightarrow v]+\frac{1}{q}\varphi_{G^*}[u\not\leftrightarrow v]=\frac{q-1}{q}\varphi_{G^*}[u\leftrightarrow v]+\frac{1}{q}\,.\qedhere
    \end{align*}
\end{proof}

\subsection{Stochastic ordering}
\label{subsec:stoch_dom}
The random cluster measure is more tractable because it satisfies a natural monotonicity property, which we now introduce. For $r\in\mathbb{N}$, let $\NB_r(\ttT_d,o)=(\ttV_r,\ttE_r)$ be the depth-$r$ $d$-regular tree rooted at $o$, and let its augmentation be $\NB_r^*(\ttT_d,o)=(\ttV_r^*,\ttE_r^*)$. Denote by $\partial \ttV_r$ the leaves of $\NB_r(\ttT_d,o)$, so $|\partial \ttV_r|=d(d-1)^{r-1}$ and each $v\in\partial \ttV_r$ is at distance $r$ from $o$. A \emph{boundary condition at level $r$} is a partition $\mathscr C=\{C_*,C_1,\dots,C_\ell\}$ of $\partial \ttV_r$, and we write $\mathfrak C(r)$ for the set of all such partitions. We now define the ES and random cluster measures on $\NB_r^*(\ttT_d,o)$ with boundary condition $\mathscr C\in\mathfrak C(r)$.

\begin{definition}\label{def-ES-coupling}
    Fix $r\in \mathbb{N}$ and a boundary condition $\mathscr C=\{C_*,C_1,\dots,C_\ell\}\in \mathfrak C(r)$. Consider the graph $\NB_r^*(\ttT_d,o)$ and $p_e$ as defined in \eqref{eq-pe}. For $\eta\in \{0,1\}^{\ttE_r^*}$, we define $\varpi_{r}^{\mathscr C}=\varpi_{r}^{\mathscr C,\beta,B}$ as the probability measure on $[q]^{\ttV_r}\times \{0,1\}^{\ttE_r^*}$ such that for all $(\sigma,\eta)\in [q]^{\ttV_r}\times \{0,1\}^{\ttE_r^*}$ (using the same conventions as before), 
    \begin{align*}
    \varpi_{r}^{\mathscr C}[\sigma,\eta]\propto 
   &\ \prod_{e=(u,v)\in \ttE_r^*}(1-p_{e})^{1-\eta_{e}}(p_e\delta_{\sigma(u),\sigma(v)})^{\eta_e}\cdot \\
   \times&\ \mathbf{1}\{\sigma(v)=1,\forall v\in C_*;\sigma(u)=\sigma(v),\forall u,v\in C_k,1\le k\le \ell\}\,.
    \end{align*}
  Additionally, we define $\varphi_{r}^{\mathscr C}=\varphi_{r}^{\mathscr C,w,B}$ as the probability measure on $\{0,1\}^{\ttE_r^*}$ such that 
    \[
    \varphi_{r}^{\mathscr C}[\eta]\propto \prod_{e\in \ttE_r^*}p_e^{\eta_e}(1-p_e)^{1-\eta_e}q^{|\operatorname{C}(\eta,\mathscr C)|}\,,\ \forall \eta\in \{0,1\}^{\ttE_r^*}\,,
    \]
    where $\operatorname{C}(\eta,\mathscr C)$ is the set of components in the graph on $\ttV_r^*$ induced by $\eta$ given that all the vertices in the same set of $\{v^*\}\cup C_*,C_1,\dots,C_\ell$ are viewed as connected. It is straightforward to check that $\varphi_{r}^{\mathscr C}$ is nothing but the marginal of $\varpi_{r}^{\mathscr C}$ on $\eta$. We will refer to the measure $\varpi_{r}^{\mathscr C}(\text{resp. }\varphi_{r}^{\mathscr C})$ as the ES measure (resp. the random cluster measure) on $\NB_r^*(\ttT_d,o)$ with boundary condition $\mathscr C$. 
\end{definition}

For any $r\in \mathbb{N}$, we define a partial ordering $\prec$ on $\mathfrak C(r)$: for two partitions $\mathscr C,\mathscr C'$ of $\partial \ttV_r$, we have $\mathscr C\preceq \mathscr C'$ if $C_*\subset C_*'$ and $\mathscr C$ is a refinement of $\mathscr C'$. The partial ordering $\prec$ induces stochastic domination relations on set the random cluster measures on $\NB_r^*(\ttT_d,o)$ with various boundary conditions, as given by the next proposition.
\begin{proposition}\label{prop-stochastic-domination}
    For any $r\in \mathbb{N}$ and any two boundary conditions $\mathscr C,\mathscr C'\in \mathfrak C(r)$ with $\mathscr C\preceq \mathscr C'$, we have $\varphi_{\beta,h,r}^{\mathscr C}\preceq_{\operatorname{st}}\varphi_{\beta,h,r}^{\mathscr C'}$, that is, there is a coupling $(\eta,\eta')$ of $\varphi_{\beta,h,r}^{\mathscr C},\varphi_{\beta,h,r}^{\mathscr C'}$ such that almost surely $\eta\le \eta'$ holds entrywise. %Moreover, if $\mathscr C\prec \mathscr C'$, then the stochastic domination is also strict, meaning that $\eta=\eta'$ cannot holds almost surely. 
\end{proposition}
Proposition~\ref{prop-stochastic-domination} is due to the positive association of the random cluster model, see, e.g. \cite{duminil2017lectures} for its proof. 

Now, consider two extreme boundary conditions $\mathscr C_{\free}$ where $C_*=\emptyset$ and all the $C_i$'s are singletons, and $\mathscr C_{\wired}$  where $C_*=\partial \ttV_r$, it holds that $\mathscr C_{\free}\preceq \mathscr C\preceq \mathscr C_{\wired}$ for any boundary condition $\mathscr C\in \mathfrak C(r)$. As a direct consequence of Proposition~\ref{prop-stochastic-domination}, abbreviating $\varphi_{r}^{\dagger}$ for $\varphi_{r}^{\mathscr C_\dagger},\dagger\in \{\free,\wired\}$, we see that the random cluster measure on $\NB_r^*(\ttT_d,o)$ with any boundary condition stochastically dominates $\varphi_{\beta,h,r}^{\free}$, and in turn stochastically dominated by $\varphi_{\beta,h,r}^{\wired}$. In particular, since the event $\{o\leftrightarrow v^*\}$ is increasing, we conclude that for any boundary condition $\mathscr C$ at level $r$, 
\begin{equation}\label{eq-domnnation-relation}
\varphi_{r}^{\free}[o\leftrightarrow v^*]\le \varphi_{r}^{\mathscr C}[o\leftrightarrow v^*]\le \varphi_{r}^{\wired}[o\leftrightarrow v^*]\,.
\end{equation}
%Furthermore, due to the strict stochastic domination relation, the equalities in \eqref{eq-domnnation-relation} can hold only when $\mathscr C=\mathscr{C}_{\operatorname{f}}$ or $\mathscr C=\mathscr{C}_{\operatorname{w}}$. 

Additionally, for $\dagger\in \{\free,\wired\}$ and any $r\le r'$, $\varphi_{r'}^{\dagger}$ restricting on $\ttE_r^*$ is a mixture of random cluster measures on $\ttE_r^*$ with various boundary conditions at level $r$, by the domain Markov property, it holds that 
\[
\varphi_{r'}^{\free}\!\mid_{\ttE_r^*}\succeq_{\operatorname{st}}\varphi_r^{\free}\,,\quad\varphi_{r'}^{\wired}\!\mid_{\ttE_r^*}\preceq_{\operatorname{st}}\varphi_r^{\wired}\,,\quad\forall r\le r'\,.
\]
This monotone relation implies that for $\dagger\in \{\free,\wired\}$, as $r\to\infty$, $\varphi_r^\dagger$ converges weakly to a probability measure $\varphi^\dagger=\varphi^{\dagger,w,B}$ on $\ttE^*$. This defines the candidates of local weak limit of random cluster measure on locally $\ttT_d$-like trees, as promised in Section~\ref{subsec-bg}.

\subsection{Small components in the random cluster model}
We end this section with an easy but useful lemma on the tail of a non-giant component in $\eta\sim \varphi_{G^*}$. 

\begin{lemma}\label{lem-cluster-size}
For any $B>0$, there exists $c<1$ such that for any $G=(V,E)$, $v\in V$ and $k\in \mathbb N$, 
\[
\varphi_{G^*}[v\not\leftrightarrow v^*\,,|C(v)|\ge k]\le c^k\,,
\]
where $C(v)$ is the component containing $v$. In particular, the probability that there exists a component not containing $v^*$ that is of size at least $(\log|V|)^2$ vanishes as $|V|\to\infty$. 
\end{lemma}

\begin{proof}
Fix $v\in V$ and $k\in\mathbb{N}$.  Condition on a realization of $\hat\eta=\eta\mid_E$ drawn from the marginal of $\varphi_{G^*}^{w,B}$ on $\{0,1\}^{E}$, and let $\hat C(v)$ be the component of $\hat\eta$ containing $v$.  By the finite‐energy property of the random cluster model, for any such $\hat\eta$, the conditional law of $\check\eta\in\{0,1\}^{E^*\setminus E}$ stochastically dominates the Bernoulli percolation with parameter $p_0=q^{-1}(1-e^{-B})>0$.  Hence, if $|\hat C(v)|\ge k$, the conditional probability that $v$ is not connected to $v^*$ in the full configuration $\eta$ is at most $(1-p_0)^k=:c^k$.  The first claim then follows by the iterated expectation theorem, and the second follows by applying this tail bound together with the union bound.
\end{proof}

\section{Local weak limits of Potts and random cluster measures}\label{sec-local-weak-limit}
In this section, we prove Theorems~\ref{thm-main-main} and~\ref{thm-main-RCM}.  Throughout, we fix an integer $d\ge3$ and a real $q>2$, and when discussing the Potts model we further assume that $q\in\mathbb{N}$.  We also fix parameters $\beta,w,B>0$ satisfying $(w,B)\in\mathsf R'_c$ and $\beta=\log(1+w)$.  Finally, we let $G_n=(V_n,E_n)$ be a sequence of uniform edge–expander graphs with $G_n\stackrel{\mathrm{loc}}{\longrightarrow}\ttT_d$, and for brevity we write $\mu_n=\mu_{G_n}^{\beta,B}$ and $\varphi_n=\varphi_{G_n^*}^{w,B}$. We also abbreviate $\mu^\dagger=\mu^{\dagger,\beta,B}$ and $\varphi^\dagger=\varphi^{\dagger,w,B}$ for $\dagger\in \{\free,\wired\}$. 

Since it is known from \cite[Lemma 2.9-(a)]{BDS23} that any subsequence of $\{\mu_n\}$ or $\{\varphi_n\}$ contains a locally weakly convergent subsequence, we may assume without loss of generality that both $\{\mu_n\}$ and $\{\varphi_n\}$ converge locally weakly. It remains for us to show that $\{\mu_n\}$ and $\{\varphi_n\}$ in fact \emph{converge locally weakly in probability} to measures in $\mathcal M$ and $\mathcal N$, respectively.

For any $n\in \mathbb N$, $\dagger\in \{\free,\wired\}$ and $\varepsilon>0$, recall the measures $\check{\nu}^{\free},\check{\nu}^{\wired}\in \Delta^\star$. We define the events 
\begin{equation}\label{eq-def-V-n-eps}
\mathcal V_{n,\varepsilon}^{\dagger}:=\Big\{\sigma\in [q]^{V_n}:\big||V_n|^{-1}\#\{v\in V_n:\sigma(v)=i\}-\check{\nu}^{\dagger}(i)\big|<\varepsilon\,, \forall i \in [q]\Big\}\,,\quad \dagger\in \{\free,\wired\}\,.
\end{equation}
Moreover, recall that we denote $\psi^{\dagger}=\varphi^{\dagger}[o\leftrightarrow v^*],\dagger\in \{\free,\wired\}$. We also define
the events
\begin{equation}\label{eq-def-E-n-eps}
           \mathcal E_{n,\varepsilon}^\dagger:=\Big\{\eta\in \{0,1\}^{E^*_n}:\big||V_n|^{-1}\#\{v\in V_n:v\leftrightarrow v^*\}- \psi^{\dagger}\big|\le\varepsilon\Big\}\,,\quad\dagger\in \{\free,\wired\}\,.
\end{equation}
We have the following easy relation from the ES coupling. 
\begin{lemma}\label{lem-relation-ES}
    For any $\varepsilon>0$, it holds that as $n\to\infty$,
    \[
    \mu_n[\mathcal V_{n,\varepsilon}^\dagger]\ge \varphi_n[\mathcal E_{n,\varepsilon}^\dagger]-o(1)\,,\quad\forall \dagger\in \{\free,\wired\}\,.
    \]
\end{lemma}
\begin{proof}
Fix $\epsilon>0$ and $\dagger\in \{\free,\wired\}$. %If $\liminf_{n\to\infty}\varphi_n[\mathcal E_{n,\varepsilon}^\dagger]=0$, then there is nothing to prove. Thus, we can assume without loss of generality that $\liminf_{n\to\infty}\varphi_n[\mathcal E_{n,\varepsilon}^\dagger]>0$ and work with $n$ large enough such that this probability is strictly positive. 
Let $\varpi_n$ be the ES measure on $[q]^{V_n}\times \{0,1\}^{E_n^*}$ that couples $\mu_n$ and $\varphi_n$. %Note that sine $ \mu_n[\mathcal V_{n,\varepsilon}^\dagger] = \varpi_n(V_{n,\varepsilon}^\dagger | \mathcal E_{n,\varepsilon}^\dagger)\varphi_n[\mathcal E_{n,\varepsilon}^\dagger] $, 
It suffices to show that as $n\to\infty$, $\varpi_n[\sigma\notin \calV_{n,\varepsilon}^\dagger, \eta\in \mathcal E_{n,\varepsilon}^\dagger] = o(1)$.

For $\eta\in \mathcal E_{n,\varepsilon}^\dagger$, let $C_*$ be the component of $\eta$ containing $v^*$ and let $C_1, \ldots,C_\ell$ be the remaining components. It is straightforward to check that for $\sigma\notin \mathcal V_{n,\varepsilon}^\dagger$, there exists $i\in [q]$ such that 
\begin{equation}\label{eq-color-discrepancy}
\big|\#\{v\in C_1\cup\cdots\cup C_\ell:\sigma(v)=i\}-\tfrac{1}{q}|V_n\setminus C_*|\big|\ge \tfrac{\varepsilon}{q}|V_n|\,.
\end{equation}

Conditioning on any realization of $\eta$, by Proposition~\ref{prop-ES-independent-smapling} and Azuma's inequality together with the union bound, we see the conditional probability that \eqref{eq-color-discrepancy} happens for some $i\in [q]$ is at most
\[
2q\cdot \exp\left(-\frac{\varepsilon^2|V_n|^2/q^2}{2\sum_{i=1}^\ell |C_i|^2}\right)\,.
\]
Additionally, Lemma~\ref{lem-cluster-size} states that with probability $1-o(1)$ as $n\to\infty$, $\max_{1\le i\le \ell}|C_i|\le (\log |V_n|)^2$, which implies $\sum_{i=1}^{\ell}|C_i|^2\le |V_n|(\log|V_n|)^2$. Altogether it yields that as $n\to\infty$,
\[
\varpi_n[\sigma\notin \mathcal V_{n,\varepsilon}^\dagger,\eta\in \mathcal E_{n,\varepsilon}^\dagger]\le o(1)+2q\cdot \exp\left(-\frac{\varepsilon^2|V_n|}{2q^2(\log|V_n|)^2}\right)=o(1)\,,
\]
as desired. This completes the proof.
\end{proof}

The main technical input of this section is the following proposition.

\begin{proposition}\label{prop-admissibility-RCM}
For any $\varepsilon>0$, it holds that $\varphi_n[\mathcal{E}_{n,\varepsilon}^{\free}\cup \mathcal{E}_{n,\varepsilon}^{\wired}]=1-o(1)$ as $n\to\infty$.
\end{proposition}

The proof of Proposition~\ref{prop-admissibility-RCM} is deferred to Sections~\ref{subsec-rank-2-approximation} and \ref{subsec-large-deviation}. Combining Proposition~\ref{prop-admissibility-RCM} with Lemma~\ref{lem-relation-ES}, we can pick a sequence $\{\varepsilon_n\}$ such that the following properties hold as $n\to\infty$:
\begin{align}
   \label{eq-choice-eps_n-1} \varepsilon_n\to 0&\,,\quad \varepsilon_n|V_n|\to\infty\,,\\
   \label{eq-choice-eps_n-2} \varphi_n[\mathcal E_{n,\varepsilon_n}^{\free}\cup& \mathcal E_{n,\varepsilon_n}^{{\wired}}]=1-o(1)\,,\\
   \label{eq-choice-eps_n-3} \mu_n[\mathcal V_{n,\varepsilon_n}^\dagger]=\varphi_n[&\mathcal E_{n,\varepsilon_n}^\dagger]+o(1),\quad\forall \dagger\in \{\free,\wired\}\,.
\end{align}
Moreover, we define 
\[
\mu_n^\dagger[\cdot]:=\mu_n[\cdot\mid \mathcal V_{n,\varepsilon_n}^\dagger]\,,\quad \varphi_n^{\dagger}[\cdot]:=\varphi_n[\cdot\mid \mathcal E_{n,\varepsilon_n}^\dagger]\,,\quad\forall \dagger\in \{\free,\wired\}\,.
\]
We claim that it suffices to show the following local weak convergence in probability result for $\mu_n^\dagger$ and $\varphi_n^\dagger$. 
\begin{proposition}\label{prop-lwcp}
    For $\dagger\in \{\free,\wired\}$, assume that $\liminf_{n\to\infty}\varphi_n[\mathcal E_{n,\varepsilon_n}^\dagger]>0$, then we have
    $$\mu_n^\dagger\stackrel{\operatorname{lwcp}}{\longrightarrow}\mu^\dagger\,,\quad \varphi_n^\dagger\stackrel{\operatorname{lwcp}}{\longrightarrow}\varphi^{\dagger}\,.$$
\end{proposition}

The proof of Proposition~\ref{prop-lwcp} is given in Section~\ref{subsec-lwcp}. We now prove Theorems~\ref{thm-main-main} and \ref{thm-main-RCM}. 

\begin{proof}[Proof of Theorems~\ref{thm-main-main} and \ref{thm-main-RCM} assuming Propositions~\ref{prop-lwcp}]
Since we have assumed that $\varphi_n$ converges locally weakly, it suffices to show a subsequence of $\varphi_n$ converges locally weakly in probability to a measure in $\mathcal N$. We pick a subsequence $\{\varphi_{n_k}\}$ such that $\alpha=\lim_{k\to\infty}\varphi_{n_k}[\mathcal V_{{n_k},\varepsilon_{n_k}}^{\free}]$ and $\alpha'=\lim_{k\to\infty}\varphi_{n_k}[\mathcal V_{{n_k},\varepsilon_{n_k}}^{\wired}]$ both exist. 
    It follows from \eqref{eq-choice-eps_n-2} that $\alpha'=1-\alpha$ and thus as $k\to\infty$, 
    $
    \operatorname{TV}(\varphi_{n_k},\alpha\varphi_{n_k}^{\free}+(1-\alpha)\varphi_{n_k}^{\wired})=o(1)
    $.
    
    If $\alpha=1$ or $0$, by proposition~\ref{prop-lwcp}, we have that $\varphi_n$ converges locally weakly in probability to $\varphi^{\free}$ or $\varphi^{\wired}$, and the result follows. 
    If $0<\alpha<1$, 
    By Proposition~\ref{prop-lwcp}, we have $\varphi_n^\dagger\stackrel{\operatorname{lwcp}}{\longrightarrow}\varphi^\dagger,\dagger\in \{\free,\wired\}$. From the equivalent condition of local weak convergence in probability \eqref{eq-characterization-lwcp}, it is easy to deduce that 
    $\alpha\varphi_n^{\free}+(1-\alpha)\varphi_n^{\wired}\stackrel{\operatorname{lwcp}}{\longrightarrow}\alpha\varphi^{\free}+(1-\alpha)\varphi^{\wired}$. This implies that $\varphi_n\stackrel{\operatorname{lwcp}}{\longrightarrow}\alpha\varphi^{\free}+(1-\alpha)\varphi^{\wired}$, and thus concludes Theorem~\ref{thm-main-RCM}. Theorem~\ref{thm-main-main} can be proved analogously. 
\end{proof}

\subsection{The rank-2 approximation and reduction to the Ising model}\label{subsec-rank-2-approximation}
The current and the next subsections are devoted to prove Proposition~\ref{prop-admissibility-RCM}. Fix $\varepsilon>0$, our goal is to show that as $n\to\infty$, 
\begin{equation}\label{eq-partition-fucntion-RCM-ratio}
\varphi_n\big[(\mathcal E_{n,\varepsilon}^{\free}\cup \mathcal E_{n,\varepsilon}^{\wired})^c\big]=\frac{\sum_{\eta\in (\mathcal E_{n,\varepsilon}^{\free}\cup \mathcal E_{n,\varepsilon}^{\wired})^c}\prod_{e\in E_n^*}w_e^{\eta_e}q^{|\operatorname{C}(\eta)|-1}}{\mathcal Z_{G^*_n}^{\RC,w,B}}=o(1)\,.
\end{equation}
We will relate the expression in \eqref{eq-partition-fucntion-RCM-ratio} to a certain deviation probability in a suitable Ising model. 

To this end, The rank-2 approximation of the random cluster partition function established in \cite{BBC23, CH25} plays an essential role.
We define
\begin{equation}\label{eq-rank-2-approximation-def}
\widetilde{\mathcal Z}_{G_n^*}^{\RC,w,B}:=\sum_{S\subset V_n}e^{B|S|}(1+w)^{|E(S)|}(q-1)^{|V_n\setminus S|}\left(1+\frac{w}{q-1}\right)^{E(V_n\setminus S)}\,.
\end{equation}
Moreover, for any $\varepsilon>0$, denoting $$I_\varepsilon=\mathbb{R}\setminus \Big(\big(\tfrac{q-1}{q}\psi^{\free} + \tfrac{1}{q}-{\varepsilon},\tfrac{q-1}{q}\psi^{\free} + \tfrac{1}{q}+{\varepsilon}\big)\cup\big(\tfrac{q-1}{q}\psi^{\wired} + \tfrac{1}{q}-\varepsilon,\tfrac{q-1}{q}\psi^{\wired} + \tfrac{1}{q}+{\varepsilon}\big)\Big)\,,$$ we further define
\begin{equation}\label{eq-rank-2-approximation-local-def}
\widetilde{\mathcal Z}_{G_n^*}^{\RC,w,B}(\varepsilon):=\sum_{\substack{S\subset V_n\\|S|/|V_n|\in I_\varepsilon}}e^{B|S|}(1+w)^{|E(S)|}(q-1)^{|V_n\setminus S|}\left(1+\frac{w}{q-1}\right)^{E(V_n\setminus S)}\,.
\end{equation}
We have the following key proposition.

\begin{proposition}\label{prop-rank-2-approximation}
    For any $n\in \mathbb{N}$, if holds
    \begin{equation}\label{eq-rank-2-approximation-global}
    {\mathcal Z}_{G_n^*}^{\RC,w,B}\ge \widetilde{\mathcal Z}_{G_n^*}^{\RC,w,B}\,.
    \end{equation}
    Moreover, for any $\varepsilon> 0$, it holds that as $n\to\infty$,
    \begin{equation}\label{eq-rank-2-approximation-local}
        \sum_{\eta\in (\mathcal E_{n,\varepsilon}^{\free}\cup \mathcal E_{n,\varepsilon}^{\wired})^c}\prod_{e\in E_n^*}w_e^{\eta_e}q^{|\operatorname{C}(\eta)|-1}
        \le e^{o(|V_n|)}\cdot\widetilde{\mathcal Z}_{G_n^*}^{\RC,w,B}(\varepsilon)+o(\mathcal Z_{G_n^*}^{\RC,w,B})\,.
    \end{equation}
\end{proposition}

Notice that \eqref{eq-rank-2-approximation-global} is simply a restatement of \cite[Lemma 2.3]{CH25}, so in what follows we focus on the proof of \eqref{eq-rank-2-approximation-local}. While the proof largely draws inspiration from \cite[Theorem 2.4]{BBC23}, we need to address several additional technical points here. %and our argument relies on $B>0$ (more specifically, Lemma~\ref{lem-cluster-size}). 
For an edge configuration $\eta\in \{0,1\}^{E_n^*}$, we write
\[
\operatorname{C}(\eta)=\{C_*\}\cup \operatorname{TC}(\eta)\cup \operatorname{NTC}(\eta)\,,
\]
where $C_*=C_*(\eta)$ denotes the component in the graph induced by $\eta$ that contains $v^*$, and $\operatorname{TC}(\eta)$ (resp. $\operatorname{NTC}(\eta)$) represents the set of tree components (resp. non-tree components) that do not contain $v^*$.
We start with the following easy lemma.

\begin{lemma}\label{lem-mathcal C_n}
    For $\delta>0$, let $\mathcal C_n(\delta)$ be the event that $\sum_{C\in \operatorname{NTC}(\eta)}|C|\le \delta n$.
    Then, for any $\delta>0$, $\varphi_n[\mathcal C_n(\delta)]=1-o(1)$ as $n\to\infty$. 
\end{lemma}
\begin{proof}
    Fix any $\varepsilon,\delta>0$. Recall $c<1$ defined as in Lemma~\ref{lem-cluster-size}, we pick $L\in \mathbb{N}$ such that $c^L\le \varepsilon$. Fix $\eta\in \{0,1\}^{E_n^*}$. Note that for any $v\in V_n$ such that $C(v)\in \operatorname{NTC}(\eta)$, either $|C(v)|\ge L$ or $C(v)$ contains a cycle with length no more than $L$. For each $3\le k\le L$, let $V_{n,k}\subset V_n$ be the set of vertices that lie on a length $k$ cycle in $G_n$. Since $G_n\stackrel{\operatorname{loc}}{\longrightarrow}\ttT_d$, we have $|V_{n,k}|=o(|V_n|)$ as $n\to\infty$. Utilizing Lemma~\ref{lem-cluster-size}, it follows that as $n\to\infty$, 
    \begin{align*}
        \mathbb{E}_{\eta\sim \varphi_{n}}\Big[\sum_{C\in \operatorname{NTC}(\eta)}|C|\Big]\le&\  \sum_{v\in V_n}\varphi_n\big[v\not\leftrightarrow v^*,|C(v)|\ge L\big]+\sum_{k=3}^L\sum_{v\in V_{n,k}}\mathbb{E}_{\eta\sim \varphi_n}\big[|C(v)|\mathbf{1}\{v\not\leftrightarrow v^*\}\big]\\
        \le&\ c^L|V_n|+\sum_{k=3}^L|V_{n,k}|\cdot \max_{v\in V_n}\sum_{t=1}^\infty \varphi_n\big[v\not\leftrightarrow v^*,|C(v)|\ge t\big]\\
        \le&\ \varepsilon|V_n|+o(|V_n|)\cdot (1-c)^{-1}\le 2\varepsilon|V_n|\,.
    \end{align*}
    Therefore, Markov's inequality yields that for all large enough $n\in \mathbb N$, $\varphi_n[\mathcal C_n(\delta)]\le 2\varepsilon/\delta$. Since $\varepsilon>0$ is arbitrary, the result follows.
\end{proof}
\begin{comment}
\begin{lemma}\label{lem-cluster-size-square}
    As $n\to\infty$, it holds that
    \[
    \mathbb{E}_{\eta\sim \varphi_n}\Big[\sum_{C\in \operatorname{C}(\eta)\setminus\{C_*\}}|C|^2\Big]=o(|V_n|^2)\,.
    \]
\end{lemma}
\begin{proof}
    \HD{Write down the proof.}
\end{proof}
\end{comment}
We now complete the proof of Proposition~\ref{prop-rank-2-approximation}.
\begin{proof}[Proof of \eqref{eq-rank-2-approximation-local}]
    For an edge configuration $\eta\in \{0,1\}^{E_n^*}$, we say a set $R\subset V_n$ is compatible with $\eta$, denoted by $\eta\sim R$, if $R$ is a union of some components (possibly nothing) in $\operatorname{TC}(\eta)$. Moreover, for $\eta\sim R$, we let $k(\eta,R)$ be the number of components in $\operatorname{TC}(\eta)$ that lie in $R$. 
    
    Denote for simplicity $\mathcal E_n:=(\mathcal E_{n,\varepsilon}^{\free}\cup \mathcal E_{n,\varepsilon}^{\wired})^c\cap \mathcal C_n(\varepsilon/q)$. We start by noticing that
    \begin{align*}
                \sum_{\eta\in (\mathcal E_{n,\varepsilon}^{\free}\cup \mathcal E_{n,\varepsilon}^{\wired})^c}\prod_{e\in E_n^*}w_e^{\eta_e}q^{|\operatorname{C}(\eta)|-1}=&\         \sum_{\eta\in \mathcal E_n}\prod_{e\in E_n^*}w_e^{\eta_e}q^{|\operatorname{NTC}(\eta)|+|\operatorname{TC}(\eta)|}+o(\mathcal Z_{G_n^*}^{\RC,w,B})\\
                =&\ \sum_{\eta\in \mathcal E_n}\prod_{e\in E_n^*}w_e^{\eta_e}q^{|\operatorname{NTC}(\eta)|}\sum_{R:\eta\sim R}(q-1)^{k(\eta,R)}+o(\mathcal Z_{G_n^*}^{\RC,w,B})\,,
    \end{align*}
    where the first equality follows from Lemma~\ref{lem-mathcal C_n} and the second one follows from the binomial identity.
    Define the set 
    \[
    \Omega_n(\varepsilon)=\Big\{(\eta,R):\eta\sim R, \Big||R|-\tfrac {q-1} q\sum_{C\in \operatorname{TC}(\eta)}|C|\Big|\le \tfrac \varepsilon{q^2}|V_n|\Big\}\,.
    \]
    Since under $\mathcal C_n(\varepsilon/q)$, we have
    \[
    |V_n|^{-1}\sum_{C\in \operatorname{TC}(\eta)}|C|\in \big(1-\tfrac{|C_*|}{|V_n|}-\tfrac{\varepsilon}{q},1-\tfrac{|C_*|}{|V_n|}+\tfrac{\varepsilon}{q}\big)\,,
    \]
    it follows that for any $\eta\in \mathcal E_n$, any set $R\subset V_n$ such that $(\eta,R)\in \Omega_n(\varepsilon)$ must satisfy that $1-|R|/|V_n|\in I_\varepsilon$. Therefore, we have
    \begin{align}
        &\ \sum_{\eta\in \mathcal E_n}\prod_{e\in E_n^*}w_e^{\eta_e}q^{|\operatorname{NTC}(\eta)|}\sum_{R:\eta\sim R}(q-1)^{k(\eta,R)}\nonumber
        \\\le &\ \sum_{\substack{R\subset V_n\\1-|R|/|V_n|\in I_\varepsilon}}\sum_{\eta:\eta\sim R}\prod_{e\in E_n^*}w_e^{\eta_e}q^{|\operatorname{NTC}(\eta)|}(q-1)^{k(\eta,R)}\label{eq-first-term}
        \\+&\ \sum_{\eta\in \mathcal E_n}\prod_{e\in E_n^*}w_e^{\eta_e}q^{|\operatorname{NTC}(\eta)|}\sum_{R:(\eta,R)\notin \Omega_n(\varepsilon)}(q-1)^{k(\eta,R)}\,,\label{eq-second-term}
    \end{align}
    
    We first prove that \eqref{eq-second-term} is $o(\mathcal Z_{G_n^*}^{\RC,w,B})$. The intuition is that for each fixed $\eta$, we may think of $R$ as a random union of components in $\operatorname{TC}(\eta)$ where each component is included independently with probability $\tfrac{q-1}{q}$, and thus typically $|R|$ is close to its expectation $\tfrac{q-1}{q}\sum_{C\in \operatorname{TC}(\eta)}|C|$. We prove this via the Chebyshev inequality: for any $\eta$, 
    \begin{align*}
    \sum_{R:(\eta,R)\notin \Omega_n(\varepsilon)}(q-1)^{k(\eta,R)}\le&\ \frac{q^4}{\varepsilon^2|V_n|^2}\sum_{R:\eta\sim R}(q-1)^{k(\eta,R)}\Big(|R|-\frac {q-1} q\sum_{C\in \operatorname{TC}(\eta)}|C|\Big)^2\\
    =&\ \frac{q^4}{\varepsilon^2|V_n|^2}\cdot q^{\operatorname{TC}(\eta)}\cdot
    q^{-1}(1-q^{-1})\sum_{C\in \operatorname{TC}(\eta)}|C|^2\,.
    \end{align*}
    Therefore, by dropping the term $q^{-1}(1-q^{-1})<1$, we have
    \begin{align*}
    \eqref{eq-second-term}\le&\  \sum_{\eta\in \mathcal E_n^*}\prod_{e\in E_n^*}w_e^{\eta_e}q^{|\operatorname{NTC}(\eta)|+|\operatorname{TC}(\eta)|}\cdot \frac{q^4}{\varepsilon^2|V_n|^2}\sum_{C\in \operatorname{TC}(\eta)}|C|^2\\
    \le &\ \mathcal Z_{G_n^*}^{\RC,w,B}\cdot \frac{q^4}{\varepsilon^2|V_n|^2}\mathbb{E}_{\eta\sim \varphi_n}\Big[\sum_{C\in \operatorname{C}(\eta)\setminus \{C_*\}}|C|^2\Big]=o(\mathcal Z_{G_n^*}^{\RC,w,B})\,,
    \end{align*} 
where the last inequality follows from the fact that as $n\to\infty$ (recall Lemma~\ref{lem-cluster-size}),
\[
\mathbb{E}_{\eta\sim \varphi_n}\Big[\sum_{C\in \operatorname{C}(\eta)\setminus \{C_*\}}|C|^2\Big]\le \varphi_n\big[\max_{C\neq C_*}|C|\ge (\log |V_n|)^2\big]|V_n|^2+|V_n|(\log |V_n|)^2=o(|V_n|^2)\,.
\]

We now address the sum in \eqref{eq-first-term}. Note that as $n\to\infty$, we have deterministically,
\[
\sup_{\eta\in \{0,1\}^{E_n^*}}|\operatorname{NTC}(\eta)|=o(|V_n|)\,.
\]
To show this, we note that for any $L\in \mathbb{N}$, it holds
\[
\sup_{\eta\in \{0,1\}^{E_n^*}}|\operatorname{NTC}(\eta)|\le X_{n,3}+X_{n,4}+\cdots+X_{n,L-1}+\frac{n}{L}\,,
\]
where $X_{n,k}$ is the number of $k$-cycles in $G_n$. This is because for any $\eta$ and $C\in \operatorname{NTC}(\eta)$, either $C$ contains a cycle with length less than $L$, or $C$ has size at least $L$. Since $G_n$ is locally tree-like, we have for any fixed $k$, $X_{n,k}=o(|V_n|)$ as $n\to\infty$, and thus the claim follows by taking $n\to\infty$ and subsequently $L\to\infty$.

Fix $R\subset V_n$, we have as $n\to\infty$,
\begin{align}
\nonumber&\ \sum_{\eta:\eta\sim R}\prod_{e\in E_n^*}w_e^{\eta_e}q^{|\operatorname{NTC}(\eta)|}(q-1)^{k(\eta,R)}\\
\nonumber\le&\ q^{o(|V_n|)}\sum_{\eta:\eta\sim R}\prod_{e\in E_n^*}w_e^{\eta_e}(q-1)^{k(\eta,R)}\\
\label{eq-forest}=&\ e^{o(|V_n|)}\cdot e^{B|V_n\setminus R|}(1+w)^{E_n(V_n\setminus R)}\sum_{F\subset E_n(R)\text{ forest}}w^{|F|}(q-1)^{|\operatorname{C}(F)|}\,.
\end{align}
Using the fact that for a forest $F\subset E_n(R)$, $|\operatorname{C}(F)|=|R|-|F|$, we see that \eqref{eq-forest} is further upper-bounded by
\[
\sum_{F\subset E_n(R)}w^{|F|}(q-1)^{|R|-|F|}=(q-1)^{|R|}\left(1+\frac{w}{q-1}\right)^{E(R)}\,.
\]
Combining these bounds and rearranging, we obtain that
\[
\eqref{eq-first-term}\le e^{o(|V_n|)}\sum_{\substack{R\subset V_n\\1-|R|/|V_n|\in I_\varepsilon}}e^{B}(1+w)^{E(V_n\setminus R)}(q-1)^{|R|}\left(1+\frac{w}{q-1}\right)^{|E(R)|}\,.
\]
Substituting $R$ with $V_n\setminus S$ and combining this with the upper-bound of \eqref{eq-second-term}, we obtain the inequality \eqref{eq-rank-2-approximation-local}. This concludes the proof. 
\end{proof}

Given Proposition~\ref{prop-rank-2-approximation}, we have
\begin{equation}\label{eq-relaxation-1}
\frac{\sum_{\eta\in (\mathcal E_{n,\varepsilon}^{\free}\cup \mathcal E_{n,\varepsilon}^{\wired})^c}\prod_{e\in E_n^*}w_e^{\eta_e}q^{|\operatorname{C}(\eta)|-1}}{\mathcal Z_{G^*_n}^{\RC,w,B}}\le e^{o(|V_n|)}\cdot\frac{\widetilde{\mathcal Z}_{G_n^*}^{\RC,w,B}(\varepsilon)}{\widetilde{\mathcal Z}_{G_n^*}^{\RC,w,B}}+o(1)\,.
\end{equation}
We define three new parameters $\beta^*,k,h$ as follows:
\begin{align}
\label{eq-beta^*}\beta^*&=\frac{1}{4}\log((1+w)(1+w/(q-1))\,,\\
k&=\frac{1}{4}\log\left(\frac{1+w}{1+w/(q-1)}\right)\,,\\ h&=\frac{1}{2}\log(e^B/(q-1))\,.
\end{align}
For $v\in V_n$, let $d_v$ be the degree of $v$ in $G$. Recall \eqref{eq-rank-2-approximation-def} for the definition of $\widetilde{\mathcal Z}_{G_n}^{\RC,w,B}$. By considering the mapping from $S\subset V_n$ to $\sigma\in \{\pm 1\}^{V_n}$ defined by $S\mapsto (2\mathbf{1}\{i\in S\}-1)_{i\in V_n}$, it is straightforward to check that (see the discussions below \cite[Theorem 1.6]{CH25} for mode details)
\[
\widetilde{\mathcal Z}_{G_n^*}^{\RC,w,B}=(q-1)^{|V_n|/2}e^{\beta^*|E_n|}\sum_{\sigma\in \{\pm 1\}^{V_n}}\exp\Bigg(\beta^*\sum_{(u,v)\in E_n}\sigma(u)\sigma(v)+\sum_{v\in V}(kd_v+h)\sigma(v)\Bigg)\,.
\]
For $(w,B)\in \mathsf R'_c$ as defined in \eqref{eq-R_c}, one directly checkes that $kd+h=0$. Since $G_n\stackrel{\loc}{\longrightarrow}\ttT_d$ (which indicates that $G_n$ is uniformly sparse and converges locally to $\ttT_d$), we have that as $n\to\infty$,
\[
\widetilde{\mathcal Z}_{G_n^*}^{\RC,w,B}=(q-1)^{|V_n|/2}e^{\beta^*|E_n|}\cdot e^{o(|V_n|)}\sum_{\sigma\in \{\pm 1\}^{V_n}}\exp\Bigg(\beta^*\sum_{(u,v)\in E_n}\sigma(u)\sigma(v)\Bigg)\,.
\]
Similarly, recalling \eqref{eq-rank-2-approximation-local-def}, we have
\[
\widetilde{\mathcal Z}_{G_n^*}^{\RC,w,B}(\varepsilon)=(q-1)^{|V_n|/2}e^{\beta^*|E_n|}\cdot e^{o(|V_n|)}\sum_{\substack{\sigma\in \{\pm 1\}^{V_n}\\\langle \sigma,1\rangle/|V_n|\in I'_\varepsilon}}\exp\Bigg(\beta^*\sum_{(u,v)\in E_n}\sigma(u)\sigma(v)\Bigg)\,,
\]
where $\langle \sigma,1\rangle:=\sum_{v\in V_n}\sigma(v)$ and $I_\varepsilon':=\{2x - 1,x\in I_\varepsilon\}$. 

For any $\boldsymbol{\beta}
>0$, denote by $\tau_n^{\boldsymbol{\beta}
}$ the Ising measure on $G_n$ with inverse temperature $\boldsymbol{\beta}
$ without external field.\footnote{Here we use bold font to emphasize that $\boldsymbol{\beta}
$ is a variable and to distinguish it with the parameter $\beta$ defined as before.} we conclude from the above identities that  as $n\to\infty$,
\begin{equation}\label{eq-relaxation-2}
\frac{\widetilde{\mathcal Z}_{G_n^*}^{\RC,w,B}(\varepsilon)}{\widetilde{\mathcal Z}_{G_n^*}^{\RC,w,B}}=e^{o(|V_n|)}\cdot \tau_n^{\beta^*}[\langle \sigma,1\rangle/|V_n|\in I_\varepsilon']\,.
\end{equation}
For $\bbeta>0$, denote by $x(\bbeta)$ the largest solution in $(0,1)$ of the equation
\[
\frac{x}{1-x}=\left(\frac{xe^\bbeta+(1-x)e^{-\bbeta}}{xe^{-\bbeta}+(1-x)e^{\bbeta}}\right)^{d-1}\,.
\]
Let $\beta^{\operatorname{Uni}}_d:=\inf\{\bbeta>0:x(\bbeta)>1/2\}$. 
Moreover, let $x^*(\bbeta)\in (0,1)$ be such that
\begin{equation}\label{eq-x^*(bbeta)}
\frac{x^*(\bbeta)}{1-x^*(\bbeta)}=\left(\frac{x(\bbeta)e^\bbeta+(1-x(\bbeta))e^{-\bbeta}}{x(\bbeta)e^{-\bbeta}+(1-x(\bbeta))e^{\bbeta}}\right)^{d}\,,
\end{equation}
and we denote $m(\boldsymbol{\beta}
)=2x^*(\bbeta)-1$. It is straightforward to check that $m(\bbeta)>0$ for $\bbeta>\beta^{\operatorname{Uni}}_d$. We record the following numerical lemma whose proof is deferred to Appendix~\ref{appendix-lem-numerical}.
\begin{lemma}\label{lem-numerical}
    For $(\beta,B)\in \mathsf R_c$, it holds that $\beta^*>\beta^{\operatorname{Uni}}_d$, and 
    \begin{equation}\label{eq-rel-psi-m}
2(\tfrac{q-1}{q}\psi^{\free}+\tfrac{1}{q})-1=-m(\beta^*)\,,\quad 2(\tfrac{q-1}{q}\psi^{\wired}+\tfrac{1}{q})-1=m(\beta^*)\,.
    \end{equation}
\end{lemma}
It follows from Lemma~\ref{lem-numerical} that $I'_\varepsilon=\{x\in \mathbb{R}:||x|-m|\ge 2\varepsilon\}$. Therefore, in light of \eqref{eq-relaxation-1} and \eqref{eq-relaxation-2}, to prove \eqref{eq-partition-fucntion-RCM-ratio} it suffices to show the following proposition.

\begin{proposition}\label{prop-Ising-large-deviation}
    For any $\boldsymbol{\beta}
>\beta^{\operatorname{Uni}}(d)$ and $\delta>0$, there exists $a>0$ (depending on $\bbeta,\delta$ and the graph sequence $\{G_n\}$) such that for all large enough $n$, 
    \begin{equation}\label{eq-Ising-large-deviation}
    \tau_n^{\boldsymbol{\beta}}
\Big[\big|\tfrac{|\langle \sigma,1\rangle|}{|V_n|}-m(\boldsymbol{\beta}
)\big|\ge \delta\Big]\le \exp(-a|V_n|)\,.
    \end{equation}
\end{proposition}
 Proposition~\ref{prop-Ising-large-deviation} is itself a new result of the Ising model on locally $\ttT_d$-like expander graphs. Note that an $o(1)$ upper-bound of the probability in \eqref{eq-Ising-large-deviation} has been obtained in \cite[Theorem 2.5]{MMS12}, but the argument therein does not yield an exponential bound. We prove Proposition~\ref{prop-Ising-large-deviation} in Section~\ref{subsec-large-deviation}.

 \subsection{Deviation estimates in the Ising and FK-Ising models}\label{subsec-large-deviation}
 This section is devoted to prove Proposition~\ref{prop-Ising-large-deviation}.
 We will again use the ES coupling to relate the Ising model with the FK-Ising model (i.e., the random cluster model with parameter $q=2$), and use properties of the latter to prove the result of the Ising model. We start by introducing the FK-Ising measure.

 For any $\bw>0$, define the FK-Ising measure on $G_n$ by
 \[
 \psi_n^{\bw}[\eta]:=\frac{\bw^{|E(\eta)|}2^{|C(\eta)|}}{\mathcal Z_{G_n}^{\FKIs,\bw}}\,,\quad\forall \eta\in \{0,1\}^E\,,
 \]
 where $E(\eta)=\{e\in E:\eta_e=1\}$ and $C(\eta)$ is the set of components of the graph induced by $\eta$, and $\mathcal Z_{G_n}^{\FKIs,\bw}$ is the normalizing constant such that $\psi_n^\bw$ is a probability measure. In what follow, we always assume that $\bw$ and $\bbeta$ satisfy that $\boldsymbol{w}=e^{2\bbeta}-1$.

For any $n\in \mathbb N$ and $\eta\in \{0,1\}^{E_n^*}$, we denote $\operatorname{C}(\eta)=\{C_1,C_2,\dots,C_\ell\}$, where $|C_1|\ge |C_2|\ge\cdots |C_\ell|$. For $\delta,R>0$, we say that an edge configuration $\eta\in \{0,1\}^{E_n}$ is $(\bbeta,\delta,R)$-good, if the following two properties hold:
 \begin{enumerate}
     \item [(i)] The largest component has size $|C_1|\in [(m(\bbeta)-\delta/4)|V_n|,(m(\bbeta)+\delta/4)|V_n|]$. 
     \item [(ii)] It holds that $\sum_{k\ge 2}|C_k|\mathbf{1}\{|C_k||\ge R\}\le \delta |V_n|/2$, 
 \end{enumerate}
 The main technical input is the following Proposition about good realizations. In what follows we use the convention  ``with probability $1-O(\exp(-a|V_n|)$'' to mean that some event with respect to $G_n$ happens with probability at least $1-\exp(-a|V_n|)$ for some constant $a>0$ for all large $n\in \mathbb N$, where the constant $a$ is independent of $n$ but may vary from line to line. 
 \begin{proposition}\label{prop-good-configuration}
     For any $\bbeta>\beta^{\operatorname{Uni}}(d)$ and $\delta>0$, there exists $R=R(\bbeta,\delta,\{G_n\})>0$ such that with probability $1-O(\exp(-a|V_n|)$ over $\eta\sim \psi_n^\bw$, $\eta$ is $(\bbeta,\delta,R)$-good.
 \end{proposition}

 Proposition~\ref{prop-Ising-large-deviation} follows readily from Proposition~\ref{prop-good-configuration} in view of the ES coupling.

 \begin{proof}[Proof of Proposition~\ref{prop-Ising-large-deviation} assuming Proposition~\ref{prop-good-configuration}]
     Fix $\bbeta>\beta^{\operatorname{Uni}}(d)$ and $\delta>0$, and let $R=R(\bbeta,\delta,\{G_n\})>0$ be chosen to satisfy Proposition~\ref{prop-good-configuration}. From Proposition~\ref{prop-ES-independent-smapling}, we can sample $\sigma\sim \tau_n^\bbeta$ as follows: first sample $\eta\sim \psi_n^\bw$ with $C(\eta)=\{C_1,\dots,C_{\ell}\}$, then for each $1\le k\le \ell$, chose $c_k\in \{\pm 1\}$ uniformly and independently at random, and set $\sigma(i)=c_k,\forall i\in C_k$. Note that if $\eta$ is $(\bbeta,\eta,R)$-good and 
     \[
     \Big|\sum_{k\ge 2}c_k|C_k|\mathbf{1}\{|C_k|<R\}\Big|\le {\delta |V_n|}/{4},
     \]
     then 
     \begin{align*}
    {\langle \sigma,1\rangle}=&\ c_1|C_1|+\sum_{k\ge 2}c_k|C_k|
    \le c_1|C_1|+\sum_{k\ge 2}|C_i|\mathbf{1}\{|C_k|\ge R\}+\Big|\sum_{k\ge 2}c_k|C_k|\mathbf{1}\{|C_k|<R\}\Big|\\
    \le&\ c_1m(\bbeta)|V_n|+\delta|V_n|/4+\delta|V_n|/2+\delta|V_n|/4
    \le c_1m(\bbeta)|V_n|+\delta |V_n|\,,
     \end{align*}
     and similarly, $\langle \sigma,1\rangle\ge c_1m(\bbeta)|V_n|-\delta|V_n|$. Therefore,
     denoting by $\mathbb{P}_n^{\bbeta,\bw}$ the joint probability measure of the above procedure, we have
     \begin{align*}
         &\ \tau_n^\bbeta\Big[\big|\tfrac{|\langle\sigma,1\rangle|}{|V_n|}-m(\bbeta)\big|\ge \delta\Big]\\
         \le& \psi_n^\bw[\eta\text{ is not $(\delta,R)$-good}]+\mathbb{P}_n^{\bbeta,\bw}\Big[    \Big|\sum_{k\ge 2}c_k|C_k|\mathbf{1}\{|C_k|<R\}\Big|\ge \delta |V_n|/4\Big]\,.
     \end{align*}
     The first probability is upper-bounded by $O(\exp(-a|V_n|))$ by Proposition~\ref{prop-good-configuration}. For the second one, since for any $\eta$ it holds deterministically that
     \[
     \sum_{k\ge 2}|C_k|^2\mathbf{1}\{|C_k|<R\}\le R\sum_{k\ge 2}|C_k|\le R|V_n|\,,
     \]
     it follows from Azuma's inequality that
     \begin{equation}\label{eq-Azuma-<=R}
     \mathbb{P}_n^{\bbeta,\bw}\Big[    \Big|\sum_{k\ge 2}c_k|C_k|\mathbf{1}\{|C_k|<R\}\Big|\ge \tfrac{\delta |V_n|}4\Big]\le 2\exp\left(-\frac{(\delta |V_n|/4)^2}{2R|V_n|}\right)=O(\exp(-a|V_n|)\,.
     \end{equation}
     Combining the two estimates together concludes the proof.
 \end{proof}

 We now turn to the proof of Proposition~\ref{prop-good-configuration}, following a similar strategy that was developed in \cite{KLS20} for the Bernoulli percolation. For $n\in \mathbb N$ and $\xi\in (0,1)$, we denote $\operatorname{Ber}_n^\xi$ as the Bernoulli edge percolation with parameter $\xi$ on $G_n$. We begin with a lemma that enables us to employ a sprinkling argument as in \cite{KLS20}.

 \begin{lemma}\label{lem-RCM-Ber-domination}
  For any graph $G$ and $\bw',\bw>0$ such that $\xi=(\bw-\bw')/4\in (0,1/2)$, it holds that 
     \[
     \psi_G^\bw\succeq_{\operatorname{st}}\psi_{G}^{\bw'}\oplus \operatorname{Ber}_G^{\xi}\,,
     \]
     where $\psi\oplus\operatorname{Ber}$ denotes the distribution of the superposition of two independent edge configurations sampled from $\psi,\operatorname{Ber}$, respectively. 
 \end{lemma}
 We give the proof of Lemma~\ref{lem-RCM-Ber-domination} in Appendix~\ref{appendix-stochastic-domination}. We will also need the following lemma appeared in \cite{KLS20}.
 \begin{lemma}\label{lem-sprinkling}
     For any $\xi>0$, there exists $R=R(\xi,\{G_n\})$ such that the following holds: For two independent edge configurations $\eta_1\sim \psi_n^{\bw'}$ and $\eta_2\sim \operatorname{Ber}_{n}^{\xi}$, with probability at least $1-O((\exp(-a|V_n|))$, there are no two disjoint subsets $A,B\subset V_n$ that are both unions of components in $\eta_1$ of size at least $R$, such that $|A|,|B|\ge \xi |V_n|$ and $A,B$ is not connected by any path in $\eta_2$. 
 \end{lemma}
 \begin{proof}
     The proof is almost identical to the proof of \cite[Claim 2.3]{KLS20}. By our assumption that $G_n$ are uniform edge-expander graphs, we can pick $s>0$ such that for all large $n\in\mathbb{N}$,
     \begin{equation}\label{eq-edge-expansion-ratio}
     \inf_{\substack{S\subset V_n\\\xi|V_n|\le|S|\le |V_n|/2}}\frac{|E_n(S,S^c)|}{|V_n|}\ge s\,.
     \end{equation}
Provided that \eqref{eq-edge-expansion-ratio} holds, by Menger's theorem, we have that for all large $n$ and any two 
disjoint subsets $A,B\subset V_n$ of size at least $\xi |V_n|$, there are at least $s\xi |V_n|$ disjoint paths in $G_n$ connecting $A,B$. Since $|E_n|=(\tfrac d 2+o(1))|V_n|\le d|V_n|$ for large $n$, there are at least $\tfrac {s\xi}{2} |V_n|$ many disjoint paths from $A$ to $B$ in $G$ of length no more than $\tfrac{2d}{s\xi}$. It follows that for all large $n\in \mathbb N$ and disjoint $A,B\subset V_n$ with $|A|,|B|\ge \xi n$, 
\[
\operatorname{Ber}_n^\xi[A,B\text{ is not connected by a path in }\eta]\le \big(1-\xi^{\tfrac{2d}{s\xi}}\big)^{\tfrac{s\xi}{2}|V_n|}\le \exp\Big(-\tfrac{s\xi}{2}\xi^{\tfrac{2d}{s\xi}}|V_n|\Big)\,.
\]
On the other hand, notice that for any $\eta_1\in\{0,1\}^{E_n}$, there are at most $2^{\tfrac{2|V_n|}{R}}$ many pairs of disjoint sets $A,B\subset V_n$ that are unions of components in $\eta_1$ of size at least $R$. Thus, if we pick $R$ such that
\[
\tfrac{2\log 2}{R}<\tfrac{s\xi}{2}\xi^{\tfrac{2d}{s\xi}}\,,
\]
then the desired result follows via taking the union bound.
 \end{proof}

For any $n\in \mathbb{N}$ and $R>0$, let $X_n(R):=\sum_{k\ge 1}|C_k|\mathbf{1}\{|C_k|\ge R\}$ where $C_1,C_2,\dots$ are the components of $\eta$ sampled from the FK-Ising measure on $G_n$. The next lemma collects some nice properties of $X_n(R)$.

\begin{lemma}\label{lem-X(R)}
    For any $\bbeta>\beta^{\operatorname{Uni}}(d)$ and $\delta>0$, the following hold:\\
    \noindent(i) For any $n\in \mathbb{N}$ and $t>0$, 
     \[
 \psi_n^\bw\Big[\big|X_n(R)-\mathbb{E}_{\psi^\bw_n}[X_n(R)]\big|\ge t\Big]\le 2\exp\left(-\frac{t^2}{4R|E_n|}\right)\,.
    \]
    \noindent(ii) For any $R>0$, $\mathbb{E}_{\psi^\bw_n}[X_n(R)]\ge (m(\bbeta)-\delta/10)|V_n|$ holds for all large $n\in \mathbb N$.\\
    \noindent(iii) There exists $R>0$ such that $\mathbb{E}_{\psi^\bw_n}[X_n(R)]\le (m(\bbeta)+\delta/10)|V_n|$ holds for all large $n\in \mathbb{N}$.
\end{lemma}

\begin{proof}
     Notice that adding or deleting one edge in the graph induced by $\eta\in \{0,1\}^{E_n}$ only influences $X_n(R)$ by at most $2R$, Item-(i) follows directly from Azuma's inequality. We next prove Items (ii) and (iii) by contrary.

If Item-(ii) is not true, by picking a subsequence if necessary, we may assume that there exists $R_0>0$ such that  $\mathbb{E}[X_n(R_0)]\le (m(\bbeta)-\delta/10)|V_n|$ for all $n$. By Item-(i), we get that with probability $1-O(\exp(-a|V_n|)$ for $\eta\sim \psi_n^\bw$, $X_n(R)\le (m(\bbeta)-\delta/20)|V_n|$. On the other hand, similarly as in \eqref{eq-Azuma-<=R} we have 
 \[
 \mathbb{P}_n^{\bbeta,\bw}\Big[\Big|\sum_{k\ge 2}c_k|C_k|\mathbf{1}\{|C_k|\le R_0\}\Big|\le \delta|V_n|/40\Big]=O(\exp(-a|V_n|)\,.
 \]
 Thus, we conclude that with probability $1-O(\exp(-a|V_n|))$ under $\mathbb P_n^{\bbeta,\bw}$,
 \[
 |\langle \sigma,1\rangle| \le X_n(R)+\delta|V_n|/40\le (m(\delta)-\delta/40)|V_n|\,.
 \]
 However, this contradicts with \cite[Theorem 2.5]{MMS12} which implies that as $n\to\infty$, with probability $1-o(1)$ under $\tau_n^\bbeta$ it holds that $$|\langle\sigma,1\rangle/|V_n||\in (m(\bbeta)-\delta/40,m(\bbeta)+\delta/40)\,.$$ 
 Thus, Item-(ii) is true. 

 Now we turn to prove Item-(iii). Since $\bbeta\mapsto m(\bbeta)$ is continuous, we can pick $\bbeta'>\bbeta$ such that $m(\bbeta')<m(\bbeta)+\delta/100$. Consider $\xi=\min\{(\bw'-\bw)/4,\delta/40\}$, we may assume that $\xi\in (0,1/2)$. Moreover, we pick $R=R(\xi,\{G_n\})$ as in Lemma~\ref{lem-sprinkling}. If Item-(ii) is not true, we may without loss of generality assume that $\mathbb{E}[X_n(R)]\ge (m(\bbeta)+\delta/10)|V_n|$ for all large $n\in \mathbb N$. By Item-(i), we have for $\eta_1\sim \psi_{n}^\bw$, it holds with probability $1-O(\exp(-a|V_n|))$ that $X_n(R)\ge (m(\bbeta)+\delta/20)|V_n|$. From Lemma~\ref{lem-RCM-Ber-domination} and Lemma~\ref{lem-sprinkling}, we conclude that 
 \[ 
 \psi_n^{\bw'}[|C_1|\ge (m(\bbeta)+\delta/40)|V_n|]\ge \psi_n^{\bw}\oplus \operatorname{Ber}_n^\xi[|C_1|\ge (m(\bbeta)+\delta/40)|V_n|]\ge 1-O(\exp(-a|V_n|))\,.
 \]
 Conditioning on any realization of $\eta$, we have with probability at least $1/2$ under $\mathbb{P}_n^{\bbeta',\bw'}$ it holds that $c_1\cdot \sum_{k\ge 2}c_k|C_k|\ge 0$. It follows that with probability at least $1/2-o(1)$ under $\mathbb{P}_n^{\bbeta',\bw'}$, 
 \[
 |\langle \sigma,1\rangle|=\Big|\sum_{k\ge 1}c_k|C_k|\Big|\ge |C_1|\ge (m(\bbeta)+\delta/40)|V_n|\,.
 \]
 This again contradicts with \cite[Theorem 2.5]{MMS12} applying to $\tau_n^{\bbeta'}$. Thus, Item-(iii) is also true. This completes the proof.
\end{proof}

We now give the proof of Proposition~\ref{prop-good-configuration}.

\begin{proof}[Proof of Proposition~\ref{prop-good-configuration}]
Fix $\bbeta>\beta^{\operatorname{Uni}}(d)$ and $\delta>0$. We pick $\bbeta'<\bbeta$ such that $m(\bbeta')>m(\bbeta)-\delta/10$, and let $\xi=\min\{(\bw-\bw')/4,\delta/20\}$. Moreover, we pick $R=R(\xi,\{G_n\})$ as in Lemma~\ref{lem-sprinkling}. 
We now verify that a configuration $\eta\sim \psi_n^{\bw}$ is $(\bbeta,\delta,R)$-good with probability $1-O(\exp(-a|V_n|)$. 

First, we show that with probability $1-O(\exp(-a|V_n|)$ over $\eta\sim \psi_n^\bw$, it holds that $|C_1|\ge (m(\bbeta)-\delta/4)|V_n|$. Apply Items-(i) and (ii) to $X_n(R)$ under $\psi_n^{\bw'}$ in Lemma~\ref{lem-X(R)}, we have 
\[
\psi_n^{\bw'}[|X_n(R)|\ge m(\bbeta)-\delta/5)|V_n|]=1-O(\exp(-a|V_n|)\,.
\]
Combining this with Lemmas~\ref{lem-RCM-Ber-domination}, \ref{lem-sprinkling}, we see that 
\[
\psi_n^{\bw}[|C_1|\ge (m(\bbeta)-\delta/4)|V_n|]\ge \psi_n^{\bw'}\oplus \operatorname{Ber}_n^\xi[|C_1|\ge (m(\bbeta)-\delta/4)|V_n|]\ge 1-O(\exp(-a|V_n|))\,,
\]
as desired. 

Now, applying Item-(iii) in Lemma~\ref{lem-X(R)} to $X_n(R)$ under $\psi_n^{\bw}$, we see there exists $R$ such that $\mathbb{E}_{\psi_n^\bw}[X_n(R)]\le (m(\bbeta)+\delta/10)|V_n|$, and thus Item-(i) in Lemma~\ref{lem-X(R)},
\[
\psi_n^{\bw}[X_n(R)\le (m(\bbeta)+\delta/4)|V_n|]=1-O(\exp(-a|V_n|))\,.
\]
Note that when $n$ is large enough such that $(m(\bbeta)-\delta/4)|V_n|>R$, given that $\eta$ satisfies both $|C_1|\ge (m(\bbeta)-\delta/4)|V_n|$ and $X_n(R)\le (m(\bbeta)+\delta/4)|V_n|$, we have $|C_1|\le (m(\bbeta)+\delta/4)|V_n|$ and $\sum_{k\ge 2}|C_k|\mathbf{1}\{|C_k|\ge R\}\le \delta|V_n|/2$, meaning that $\eta$ is a $(\bbeta,\delta,R)$ good configuration. The proof is now completed.
\end{proof}

\subsection{Local weak convergence of the conditional measures}\label{subsec-lwcp}
In this subsection, we prove Proposition~\ref{prop-lwcp}. Throughout, we fix $\dagger\in \{\free,\wired\}$ and assume that 
\begin{equation}\label{eq-assumption-liminf>0}
\liminf_{n\to\infty}\varphi_n[\mathcal E_{n,\varepsilon_n}^\dagger]>0
\end{equation}

We first show that $\varphi_n^\dagger\stackrel{\operatorname{lwcp}}{\longrightarrow} \varphi^\dagger$. To this end, we need to understand the set of possible local weak limit points of $\varphi_n^\dagger$. We start with an easy lemma which implies that any subsequential local weak limit of $\varphi_n^\dagger$ is still a random cluster measure on $\ttT_d$ (whose meaning will become clear later).
\begin{lemma}\label{lem-local-CLT}
    It holds that as $n\to\infty$,
    \begin{equation}\label{eq-CLT-RCM}
    \sup_{t\in \mathbb N}\varphi_n\big[\#\{v\in V_n:v\leftrightarrow v^*\}=t\big]=o(1)\,,
    \end{equation}
    \begin{equation}\label{eq-CLT-Potts}
    \sup_{k\in [q]}\sup_{t\in \mathbb N}\mu_n\big[\#\{v\in V_n:\sigma(v)=k\}=t\big]=o(1)\,.
    \end{equation}
\end{lemma}
The proof of Lemma~\ref{lem-local-CLT} is in spirit similar to the proof of \cite[Lemma 4.1]{MMS12} and we leave it into Appendix~\ref{appendix-lem-local-CLT}. We now proceed to introduce the set of possible subsequential local weak limit point of $\varphi_n^\dagger$.

Recall that for any $r\in \mathbb{N}$ and $\mathscr C\in \mathfrak C(r)$, we define the random cluster measure $\varphi_r^\mathscr{C}=\varphi_{r}^{\mathscr C,w,B}$ on $\operatorname{N}^*_r(\ttT_d,o)=(\ttV_r^*,\ttE_r^*)$ with boundary condition $\mathscr C$ at level $r$. Let $\mathcal R_r,\widetilde{\mathcal R}_r\subset \mathcal P(\{0,1\}^{\ttE_r^*})$ be defined as follows:
\begin{align*}
\mathcal R_r&:=\Big\{\sum_{\mathscr C\in \mathfrak C(r)}\lambda(\mathscr C)\varphi^{\mathscr C}_{r}:\lambda\ge 0\,,\sum_{\mathscr C\in \mathfrak C(r)}\lambda(\mathscr C)=1\Big\}\,,\\
\widetilde{\mathcal R}_r&:=\{\varphi_r\in \mathcal R_r:\varphi^{\free}\!\mid_{\ttE_r^*}\preceq_{\operatorname{st}}\varphi_r\preceq_{\operatorname{st}}\varphi^{\wired}\!\mid_{\ttE^*_{r}}\}\,.
\end{align*}
It is straightforward to check that for any $r\in \mathbb N$, both $\mathcal R_r$ and $\widetilde{\mathcal R}_r$ are closed under the weak topology.
Moreover, we define $\widetilde{\mathcal{R}}$ to be the set of probability measures $\varphi$ on $\{0,1\}^{\ttE^*}$ such that for any $r\in \mathbb N$, it holds that $\varphi_r$, the marginal of $\varphi$ on $\{0,1\}^{\ttE_r^*}$, lies in $\widetilde{\mathcal R}_r$.

\begin{lemma}\label{lem-subsequential-limit}
   Under our assumption, any subsequential local weak limit point of $\{\varphi_n^\dagger\}$ os supported on the set $\widetilde{\mathcal R}$. 
\end{lemma}
\begin{proof}
For any $\varphi\in \mathcal P(\{0,1\}^{\ttE^*})$ that lies in the support of a subsequential limit of $\varphi_n^\dagger$, it suffices to show that $\varphi_r\in \widetilde{\mathcal R}_r$ for any $r\in \mathbb{N}$. For any $v\in V_n$ such that $\operatorname{N}_r(G_n,v)\cong \NB_r(\ttT_d,o)$, upon fixing a realization of $\check{\eta}=\eta\!\mid_{E_n^*\setminus E_{n,r}^*}$, the remaining configuration $\eta\!\mid_{E_{n,r}^*}$ of $\eta\sim \varphi_n^\dagger$ has the distribution of $\varphi_n[\cdot\mid \mathcal V_{n,\varepsilon_n}^\dagger,\check{\eta}]$. Note that if $\check{\eta}$ satisfies that
\begin{equation}\label{eq-check-eta-good}
(\psi^\dagger-\varepsilon_n)|V_n|\le \#\{v\in V_n\setminus V_{n,r}:v\stackrel{\check{\eta}}{\leftrightarrow} v^*\}\le (\psi^\dagger+\varepsilon_n)|V_n|-3d^r\,,
\end{equation}
then since $|V_{n,r}^*|\le 3d^r$, it follows that 
\[
\varphi_n[\cdot\mid \mathcal E_{n,\varepsilon_n}^\dagger,\check{\eta}]=\varphi_n[\cdot\mid\check{\eta}]\,,
\]
which, by the domain Markov property of $\varphi_n$, is the random cluster measure on $\{0,1\}^{E_{n,r}^*}$ with the boundary condition induced by $\check{\eta}$. On the other hand, since any proper realization of $\check{\eta}=\eta\!\mid_{E_n^*\setminus E_{n,r}^*}$ with $\eta\sim \varphi_n^\dagger$ necessarily satisfies
\[
(\psi^\dagger-\varepsilon_n)|V_n|-3d^r\le \#\{v\in V_n\setminus V_{n,r}:v\stackrel{\check{\eta}}{\leftrightarrow} v^*\}\le (\psi^\dagger+\varepsilon_n)|V_n|\,,
\]
it follows that 
\begin{align*}
    \varphi_n^\dagger[\eqref{eq-check-eta-good}\text{ fails}]\le &\ \varphi_n^\dagger\Big[\min_{a \in \{\pm1\}}\big\{\big|\#\{v\in V_n:v\leftrightarrow v^*\}-(\psi^\dagger+a \varepsilon_n)|V_n|\big|\big\}\le 3d^r\Big]\\
    \le&\ \frac{12d^r}{\varphi_n[\mathcal E_{n,\varepsilon_n}^\dagger]}\times \max_{t\in \mathbb N}\varphi^\dagger_n\big[\#\{v\in V_n:v\leftrightarrow v^*\}=t\big]\,,
\end{align*}
which tends to $0$ as $n\to\infty$ by Lemma~\ref{lem-local-CLT} and our assumption \eqref{eq-assumption-liminf>0}. This proves that as $n\to\infty$, for any $v\in V_n$ such that $\operatorname{N}_r(G_n,v)\cong \NB_r(\ttT_d,o)$, the marginal of $\varphi_n^\dagger$ on $E_{n,r}^*$ is $o(1)$ TV-close to an element in $\mathcal R_r$. Since $\mathcal R_r$ is closed, it follows that $\varphi_r\in \mathcal R_r$.

Similarly, for any $R>r$, arguing as above we get that as $n\to\infty$, for $v\in V_n$ such that $\operatorname{N}_R(G_n,v)\cong \NB_R(\ttT_d,o)$, the marginal of $\varphi_n^\dagger$ on $\{0,1\}^{E_{n,R}^*}$ is $o(1)$ TV-close to an element in $\mathcal R_R$. It then follows from Proposition~\ref{prop-stochastic-domination} that the marginal of $\varphi_n^\dagger$ on $\{0,1\}^{E_{n,r}^*}$ is $o(1)$-close to an element that dominates 
$(\varphi_{R}^{\free})_r$ and in turn be dominated by $(\varphi_{R}^{\wired})_r$. Since for any $R>r$, the set of measures in $\mathcal P(\{0,1\}^{\ttE_{r}^*})$ satisfying the above property is also closed, we obtain that 
$(\varphi^{\free}_R)_r\preceq_{\operatorname{st}}\varphi_r\preceq_{\operatorname{st}}(\varphi^{\wired}_R)_r$. 
Sending $R\to\infty$, we conclude that $\varphi_r\in \widetilde{R}_r$. This completes the proof.
\end{proof}

Knowing that any subsequential local weak limit of $\varphi_n^\dagger$ has support in $\widetilde{\mathcal R}$, to show the support contains only $\varphi^\dagger$, we need to find a way of distinguishing $\varphi^{\dagger}$ from other elements in $\widetilde{\mathcal R}$. This is incorporated in the following lemma.

\begin{lemma}\label{prop-characterization}
     For any $\varphi\in \widetilde{\mathcal R}\setminus \{\varphi^{\free},\varphi^{\wired}\}$, it holds that $\psi^{\free}<\varphi[o\leftrightarrow v^*]<\psi^{\wired}$. 
\end{lemma}
\begin{proof}
    Since for any $\varphi\in \widetilde{\mathcal R}$, it holds $\varphi^{\free}\preceq \varphi\preceq \varphi^{\wired}$ and hence $\psi^{\free}\le \varphi[o\leftrightarrow v^*]\le \psi^{\wired}$, it suffices to show that for $\ddagger\in \{\free,\wired\}$, $\varphi[o\leftrightarrow v^*]=\psi^{\ddagger}$ implies $\varphi=\varphi^{\ddagger}$.
    
    The proof is essentially an adaption of the proof of \cite[Lemma~5.4]{BDS23}. We first introduce the RCM pre-messages as defined in \cite[Definition~5.3]{BDS23}. For any $r\ge 0$, let $\mathscr T_{r}$ be the $\sigma$-algebra generated by the indicators $\eta_{e},e\in \ttE^*\setminus\ttE_r^*$ and denote $\mathscr T=\bigcap_{r\ge 0}\mathscr T_r$. For $(u,v)\in \ttE$, we define the $\mathscr T$-measurable random variable
    \[
    s_{u\to v}(\varphi)=\varphi[u\leftrightarrow v^*\mid \eta_{(u,v)}=0,\mathscr T]\,.
    \]
    $s_{u\to v}(\varphi),(u,v)\in \ttE$ are called the RCM pre-messages and they can be viewed as certain marginals of $\varphi$ with a specified boundary condition at infinity. 
    
    We denote $\gamma=(e^\beta-1)/(e^\beta+q-1)$, and define the function $\widehat{\operatorname{BP}}_k:[0,1]^k\to [0,1]$ as
    \begin{equation}\label{eq-BP-RCM}
    \widehat{\operatorname{BP}}_k(s_1,\dots,s_k):=\frac{e^B\prod_{i=1}^{k}(1+(q-1)\gamma s_i)-\prod_{i=1}^{k}(1-\gamma s_i)}{e^B\prod_{i=1}^{k}(1+(q-1)\gamma s_i)+(q-1)\prod_{i=1}^{k}(1-\gamma s_i)}\,.
    \end{equation}
    It can be checked that $\widehat{\operatorname{BP}}_k$ is coordinate-wise strictly increasing for any $k\in \mathbb N$. 
    Moreover, we write $\widehat{\operatorname{BP}}_{d-1}$ as $\widehat{\operatorname{BP}}$ for simplicity. 
    Define $b^{\free}$ (resp. $b^{\wired}$) as the $n\to\infty$ limit of $b_{n+1}=\widehat{\operatorname{BP}}(b_n,\dots,b_n)$ with initialization $b_0=0$ (resp. $b_0=1$). We have the fact that $b^{\free},b^{\wired}$ are the smallest and largest fixed points of $\widehat{\operatorname{BP}}(x,\dots,x)$, and $\widehat{\operatorname{BP}}_d(b^{\ddagger},\dots,b^{\ddagger})=\psi^{\ddagger},\forall \ddagger\in \{\free,\wired\}$. Furthermore, for $\ddagger\in \{\free,\wired\}$,  $\varphi^{\ddagger}$ is characterized by the property that $\varphi^\ddagger$-a.s. $s_{u\to v}(\varphi^\ddagger)=b^{\ddagger},\forall (u,v)\in \ttE$ (see, e.g. \cite[Proposition 4.6]{Gri06}). We will argue that $\varphi[o\leftrightarrow v^*]=\psi^{\ddagger}$ also implies this property, thereby deducing $\varphi=\varphi^{\ddagger}$.  
    
    Towards this end, we will make use the following $\varphi$-a.s. properties of the RCM pre-messages:
    \begin{itemize}
        \item $b^{\free}\le s_{u\to v}(\varphi)\le b^{\wired}$ for all $(u,v)\in \ttE$. 
        \item $s_{u\to v}=\widehat{\operatorname{BP}}(\mathbf{s}_{\partial u\setminus \{v\}\to u})$, where $\mathbf{s}_{\partial u\setminus \{v\}\to u}=(s_{w\to u})_{w\sim u,w\neq v}$. 
        \item $\varphi[o\leftrightarrow v^*\mid \mathscr T]=\widehat{\operatorname{BP}}_d(\mathbf{s}_{\partial o\to o})$, where $\mathbf{s}_{\partial o\to o}=(s_{v\to o})_{v\sim o}$. 
    \end{itemize}
    The first two properties are restatements of \cite[Lemma~5.4]{BDS23}, and the third one follows from a straightforward belief-propagation calculation. From these properties and the coordinate-wise monotonicity of $\widehat{\operatorname{BP}}_d$, we conclude that $\varphi$-a.s. it holds
    \[
    \psi^{\free}=\widehat{\operatorname{BP}}(b^{\free},\dots,b^{\free})\le\varphi[o\leftrightarrow v^*\mid \mathscr T]\le \widehat{\operatorname{BP}}(b^{\wired},\dots,b^{\wired})=\psi^{\wired}\,.
    \]
    Therefore, for $\ddagger\in \{\free,\wired\}$, the assumption 
    \[
    \psi^{\ddagger}=\varphi[o\leftrightarrow v^*]=\mathbb{E}\big[\varphi[o\leftrightarrow v^*\mid \mathscr T]\big]
    \]
    implies that $\varphi$-a.s. $\varphi[o\leftrightarrow v^*\mid \mathscr T]=\psi^{\ddagger}$. Hence, by the strict coordinate-wise increasing property of $\widehat{\operatorname{BP}}_{d}$, we have $\varphi$-a.s. $s_{v\to o}(\varphi)=b^{\ddagger},\forall v\sim o$. Then, since $\widehat{\operatorname{BP}}$ is also coordinate-wise strictly increasing and $b^{\ddagger}$ is the fixed point of $\widehat{\operatorname{BP}}(x,\dots,x)$, we can show inductively that $\varphi$-a.s. $s_{u\to v}(\varphi)=b^{\ddagger},\forall (u,v)\in \ttE$. This proves that $\varphi=\varphi^\ddagger$, as desired.
\end{proof}

We can now prove the local weak convergence in probability of $\varphi_n^\dagger$.

\begin{lemma}\label{lem-lwcp-RCM}
    Under our assumption, $\varphi_n^\dagger\stackrel{\operatorname{lwcp}}{\longrightarrow}\varphi^\dagger$. 
\end{lemma}
\begin{proof}
For any subsequential local weak limit point $\mathtt n$ of $\varphi_n$, there is a subsequence $\{\varphi_{n_k}\}$ of $\{\varphi_n\}$ such that 
\[
|V_{n_k}|^{-1}\mathbb{E}_{\varphi_{n_k}^\dagger}\big[\#\{v\in V_{n_k}:v\leftrightarrow v^*\}\big]\to \int\varphi[o\leftrightarrow v^*]\operatorname{d}\!\mathtt n
\]
    On the other hand, by the definition of $\mathcal E_{n,\varepsilon_n}^\dagger$, we have 
    \[
    \psi^\dagger-\varepsilon_n\le |V_n|^{-1}\mathbb{E}_{\varphi_n^\dagger}\big[\#\{v\in V_n:v\leftrightarrow v^*\}\big]\le \psi^\dagger+\varepsilon_n\,.
    \]
    Since $\varepsilon_n\to 0$, combining Lemmas~\ref{lem-subsequential-limit} and \ref{prop-characterization}, we conclude that $\mathtt n$ can only be supported on the single measure $\varphi^\dagger$. Thus by definition, $\varphi_n\stackrel{\operatorname{lwcp}}{\longrightarrow}\varphi^\dagger$, as desired.
\end{proof}

%Theorem~\ref{thm-main-RCM} assuming Proposition~\ref{prop-admissibility-RCM}. 
%\begin{proof}[Proof of Theorem~\ref{thm-main-RCM}]
%Let $\mathtt n$ be a subsequential local weak limit of $\mu_n$. By the definition of local weak convergence (recall Definition~\ref{def-lwc-unified} with $\mathcal X=\{\textbf{o}\}$ and $\mathcal Y=\{0,1\}$), it must hold that $\mathbf P_n^0(v_n)$, the joint distribution of $(\NB_r(G_n,v_n),\mathbf{1}\{v_n\leftrightarrow v^*\})$ weakly converges to $\delta_{\ttT_d}\otimes \mathtt n^0$, where $\mathtt n^0$ is the marginal of $\mathtt n$ on the indicator $\mathbf{1}\{o\leftrightarrow v^*\}$. By Proposition~\ref{prop-admissibility-RCM}, we see that $\mathtt n^0$ supports on $\{\operatorname{Ber}(\psi^{\free}),\operatorname{Ber}(\psi^{\wired})\}$. Combining with Lemma~\ref{prop-characterization}, this implies $\mathtt n$ supports on $\{\varphi^{\free},\varphi^{\wired}\}$, and thus $\mathtt n$ must have the form $\alpha\delta_{\varphi^{\free}}+(1-\alpha)\delta_{\varphi^{\wired}}$.
%\end{proof}

Turning to prove the local weak convergence in probability statement for $\mu_n^\dagger$, we first given an approximate coupling of $\mu_n^\dagger$ and $\varphi_n^\dagger$ in view of the ES coupling. Let $\varpi_n$ be the ES measure that couples $\mu_n$ and $\varphi_n$ as defined in Definition~\ref{def:ES_measure}. Moreover, we denote by $\varpi_n^\dagger[\cdot]:=\varpi_n[\cdot\mid \eta\in \mathcal E_{n,\varepsilon_n}^\dagger]$ and $\varpi_{n,s}^\dagger$ the marginal of $\varpi_n^\dagger$ on $\sigma\in [q]^{V_n}$.

\begin{lemma}\label{lem-TV-ES}
    Under our assumption, it holds that as $n\to\infty$,
    \begin{equation}\label{eq-TV-ES-1}
    \operatorname{TV}(\mu_n^\dagger,\varpi^\dagger_{n,s})=o(1)\,.
    \end{equation}
\end{lemma}

\begin{proof}
Let $\varpi_{n,s}$ and $\varpi_{n,b}$ be the marginals on $\sigma\in [q]^{V_n}$ and $\eta\in \{0,1\}^{E_n^*}$ of $\varpi_n$, respectively. We have $\mu_n\stackrel{\operatorname{d}}{=}\varpi_{n,s}$ and $\varphi_n\stackrel{\operatorname{d}}{=}\varpi_{n,b}$ by the ES coupling. It follows from \eqref{eq-choice-eps_n-3} that as $n\to\infty$,
\[
\varpi_{n,s}[\sigma\in \mathcal V_{n,\varepsilon_n}^{\dagger}]=\mu_n[\mathcal V_{n,\varepsilon_n}^\dagger]=\varphi_n[\mathcal E_{n,\varepsilon_n}^\dagger]+o(1)=\varpi_{n,b}[\eta\in \mathcal E_{n,\varepsilon_n}^{\dagger}]+o(1)\,,
\]
which is bounded away from $0$ by our assumption \eqref{eq-assumption-liminf>0}. Upon inspecting the proof of Lemma~\ref{lem-relation-ES}, we have indeed shown that $\varpi_n[\{\eta\in \mathcal E_{n,\varepsilon_n}^\dagger\}\setminus \{\sigma\in \mathcal V_{n,\varepsilon_n}^\dagger\}]=o(1)$. The above yields that $\varpi_n[\{\eta\in \mathcal E_{n,\varepsilon_n}^\dagger\}\Delta \{\sigma\in \mathcal V_{n,\varepsilon_n}^\dagger\}]=o(1)$. Thus, as $n\to\infty$,
\[
\operatorname{TV}\big({\mu_n^\dagger,\varpi_{n,s}^\dagger}\big)=\operatorname{TV}\big(\varpi_{n,s}[\cdot\mid \sigma\in \mathcal V_{n,\varepsilon_n}^\dagger],\varpi_{n,s}[\cdot\mid \eta\in \mathcal E_{n,\varepsilon_n}^\dagger]\big)=o(1)\,.\qedhere 
\]
\begin{comment}
%This verifies \eqref{eq-TV-ES-1}.
Similarly, since $\widehat{G}_n$ are also uniform edge-expander graphs satisfying that $\widehat{G}_n\stackrel{\operatorname{loc}}{\longrightarrow}\ttT_d$, we have
\begin{equation}\label{eq-limit-2}
\lim_{n\to\infty}\hat{\mu}_n[\mathcal V_{n,\varepsilon_0}^{\dagger}]=\lim_{n\to\infty}\hat{\mu}_n[\mathcal E_{n,\varepsilon_0}^\dagger],\quad\forall\dagger\in \{\free,\wired\}\,.
\end{equation}
Moreover, by explicitly writing down the expressions, it is straightforward to check that for $\dagger\in \{\free,\wired\}$, the ratio of
$
{\mu_n[\mathcal V_{n,\varepsilon_0}^{\dagger}]}$ to ${\hat{\mu}_n[\mathcal V_{n,\varepsilon_0}^\dagger]}
$ is uniformly bounded both from above and from below. Therefore, we conclude from \eqref{eq-limit} that the limits in \eqref{eq-limit-2} is positive. Then \eqref{eq-TV-ES-2} can be verified similarly as above. This concludes the proof.
\end{comment}
\end{proof}

We can define the local weak convergence of $\varpi_n$ by taking $\mathcal X=[q],\mathcal Y=\{0,1\}$ in Definition~\ref{def-lwc-unified}, and below we discuss the properties of local weak limit points of $\varpi_n^\dagger$. For any $r\in \mathbb{N}$, let $\eta$ be a bond configuration on $\ttE^*_r$ with boundary condition $\mathscr C=\{C_*,C_1,\dots,C_\ell\}$.\footnote{This means that vertices in the same sets of $\{v^*\}\cup C_*,C_1,\dots,C_\ell$ are viewed as connected.} We denote by $C_*,C_1,\dots,C_\ell$ the sets of connected components of $\ttV_r^*$ induced by $\eta$, and define $\mathtt{u}_r(\eta)$ as the distribution of spin configurations $\sigma\in [q]^{\mathbb{V}_r^*(o)}$ such that $\sigma(v)=1$ for $v\in C^*$, and $\sigma(v),v\in C_i$ take the same values that are uniformly and independently chosen from $[q]$ for $1\le i\le \ell$. Moreover, for any boundary condition $\mathscr C\in \mathfrak C(r)$, let 
$$
\theta_r(\varphi_{r}^{\mathscr C}):=\mathbb{E}_{\eta\sim \varphi_r^{\mathscr C}}[\mathtt{u}_r(\eta)]
$$ 
be the averaged distribution on $[q]^{\ttV_r^*}$ of $u_r(\eta)$ over $\eta\sim \varphi_r^{\mathscr C}$ viewed as to have boundary condition $\mathscr C$. Furthermore, for $\varphi_r\in \mathcal R_r$ with $\varphi_r=\sum_{\mathscr C\in \mathfrak C_r}\lambda(\mathscr C)\varphi_r^{\mathscr C}$, we define 
\[
\theta_r(\varphi_r):=\sum_{\mathscr C\in \mathfrak C(r)}\lambda(\mathscr C)\theta_r(\varphi_r^{\mathscr C})\,.
\]
For $\varphi\in \widetilde{R}$, define $\theta(\varphi)$ as the probability measure on $[q]^{\ttV}$ such that the marginal distribution of $\theta(\varphi)$ on $[q]^{\ttV_r}$ agrees with $\theta_r(\varphi_r)$, where $\varphi_r\in \mathcal R_r$ is the marginal of $\varphi$ on $\{0,1\}^{\ttE_r^*}$ as before. Note that the existence and uniqueness of $\theta(\varphi)$ is guaranteed by Kolmogrov's extension theorem.

Recall Lemma~\ref{lem-subsequential-limit} establishes that any subsequential local weak limit points of $\varpi_{n,b}^\dagger=\varphi_n^\dagger$ lies in $\widetilde{\mathcal R}$. The next lemma gives further relation between $\varpi_{\operatorname{b}}, \varpi_{\operatorname{s}}$ and $\varphi^{\dagger},\mu^{\dagger}$.

\begin{lemma}\label{lem-fact}
(a) Any subsequence of $\{\varpi_n^\dagger\}$ has a locally weakly convergent subsequence.\\
(b) any subsequential local weak limit point of $\{\varpi^\dagger_n\}$ is supported on the set $\{\varpi:\varpi_{\operatorname{s}}=\theta(\varpi_{\operatorname{b}})\}$, where $\varpi_{\operatorname{s}},\varpi_{\operatorname{b}}$ are the marginals of $\varpi$ on the spin and bond configurations. \\
(c) $\theta(\varphi^{\free})=\mu^{\free}$ and $\theta(\varphi^{\wired})=\mu^{\wired}$. 
\end{lemma}
Items-(a), (b) and (c) of Lemma~\ref{lem-fact} can be proved as in \cite[Lemma~2.9(a), Remark~2.10 and Lemma~2.11(b)]{BDS23}, respectively, and we omit the details here. Now we are ready to prove Proposition~\ref{prop-lwcp}.

\begin{proof}[Proof of Proposition~\ref{prop-lwcp}]
We have shown in Lemma~\ref{lem-lwcp-RCM} that $\varphi_n^\dagger\stackrel{\operatorname{lwcp}}{\longrightarrow}\varphi^\dagger$, and it remains to prove $\mu_n^\dagger\stackrel{\operatorname{lwcp}}{\longrightarrow}\mu^\dagger$. In light of Lemma~\ref{lem-TV-ES}, it suffices to prove that $\varpi_{n,s}^\dagger\stackrel{\operatorname{lwcp}}{\longrightarrow}\mu^\dagger$. 
  Let $\mathtt m$ be a subsequential local weak limit of $\varpi_{n,s}^\dagger$. By Lemma~\ref{lem-fact} (a), we may pick a sequence $\{\varpi_{n_k}\}$ such that $\varpi_{n_k}\stackrel{\text{lwc}}{\longrightarrow}\mathtt p$, where $\mathtt m$ is the margianl of $\mathtt p$ on $\mathcal P([q]^{\ttV})$. 
   
   By Lemma~\ref{lem-lwcp-RCM}, any $\varpi\in \mathcal P([q]^{\ttV}\times \{0,1\}^{\ttE^*})$ lying in the support of $\mathtt p$ satisfies that $\varpi_{\operatorname{b}}=\varphi^\dagger$. Thus, by Lemma~\ref{lem-fact} (b) and (c) we have $\varpi_{\operatorname{s}}=\theta(\varpi_{\operatorname{b}})=\theta(\varphi^{\dagger})=\mu^\dagger$. This means that $\mathtt m$ is supported on the single measure $\mu^\dagger$. Hence by definition, we have $\mu_n^\dagger\stackrel{\operatorname{lwcp}}{\longrightarrow}\mu^\dagger$, as desired. This concludes the proof.
\end{proof}

\section{Strong phase coexistence with arbitrary weights}\label{sec-phase-coexistence}

In this section we prove Theorem~\ref{thm-main}. Fix two integers $d,q\ge 3$ and $(\beta,B)\in \mathsf R_c$. For a sequence of uniform edge-expander graphs graphs $G_n\stackrel{\operatorname{loc}}{\longrightarrow}\ttT_d$, recall the events $\mathcal V_{n,\varepsilon}^{\dagger},\dagger\in \{\free,\wired\},\varepsilon>0$ as defined in \eqref{eq-def-V-n-eps}. Fix a small constant $\varepsilon_0>0$ such that $\mathcal V_{n,\varepsilon_0}^{\free}\cap\mathcal V_{n,\varepsilon_0}^{\wired}=\emptyset$, we further define the free and wired \emph{partial Potts partition functions} as follows:
\[
\mathcal Z_{G_n}^{\dagger}=\mathcal Z_{G_n}^{\Po,\beta,B}(\mathcal V_{n,\varepsilon_0}^\dagger):=\sum_{\sigma\in \mathcal V_{n,\varepsilon_0}^\dagger}\exp\Bigg(\beta\sum_{(u,v)\in E_n}\delta_{\sigma(u),\sigma(v)}+B\sum_{v\in V_n}\delta_{\sigma(v),1}\Bigg)\,,\quad\dagger\in \{\free,\wired\}\,.
\]
We start with a lemma that characterizes the weights of the free and wired Gibbs measures in the in-probability local weak limit of $\mu_n=\mu_{G_n}^{\beta,B}$ in terms of $\mathcal Z_{G_n}^{\dagger},\dagger\in \{\free,\wired\}$. 

\begin{lemma}\label{lem-mixture-ratio-characterization}
    For uniform edge-expander graphs $G_n\stackrel{\operatorname{loc}}{\longrightarrow}\ttT_d$, if there exists $\alpha\in (0,1)$ such that 
    \[
    \lim_{n\to\infty}\frac{\mathcal Z_{G_n}^{\free}}{\mathcal Z_{G_n}^{\wired}}=\frac{\alpha}{1-\alpha}\,,
    \]
    then $\mu_n\stackrel{\operatorname{lwcp}}{\longrightarrow}\alpha{\mu^{\free}}+(1-\alpha){\mu^{\wired}}$. 
\end{lemma}
\begin{proof}
By Lemma~\ref{lem-relation-ES} and Proposition~\ref{prop-admissibility-RCM}, we have that as $n\to\infty$,
\[
\mu_n[\sigma\in \mathcal V_{n,\varepsilon_0}^{\free}\cup \mathcal V_{n,\varepsilon_0}^{\wired}]=1-o(1)\,.
\]
By our assumption we have as $n\to\infty$,
\[
\frac{\mu_n[\mathcal V_{n,\varepsilon_0}^{\free}]}{\mu_n[\mathcal V_{n,\varepsilon_0}^{\wired}]}=\frac{\alpha}{1-\alpha}+o(1)\,.
\]
Altogether we have that as $n\to\infty$,
$
\mu_n[\mathcal V_{n,\varepsilon_0}^{\free}]=\alpha+o(1)$, and $\mu_n[\mathcal V_{n,\varepsilon_0}^{\wired}]=1-\alpha+o(1)
$.

Recalling our choice of the sequence $\{\varepsilon_n\}$ such that \eqref{eq-choice-eps_n-1}-\eqref{eq-choice-eps_n-3} hold. We have for large $n$, $\varepsilon_n\le \varepsilon_0$ and thus
$
\mu_n[\mathcal V_{n,\varepsilon_n}^\dagger]\le \mu_n[\mathcal V_{n,\varepsilon_0}^\dagger], \dagger\in \{\free,\wired\}
$. Meanwhile, since $\mu_n[\mathcal V_{n,\varepsilon_n}^{\free}]+\mu_n[\mathcal V_{n,\varepsilon_n}^{\wired}]=1-o(1)$, it follows that as $n\to\infty$,
$
\mu_n[\mathcal V_{n,\varepsilon_n}^{\free}]=\alpha+o(1)$, and $\mu_n[\mathcal V_{n,\varepsilon_n}^{\wired}]=1-\alpha+o(1)
$.
Therefore, as shown in the proof of Theorem~\ref{thm-main-main}, we have
\[
\mu_n\stackrel{\operatorname{lwcp}}{\longrightarrow}\lim_{n\to\infty}\mu_n[\mathcal V_{n,\varepsilon_n}^{\free}]\cdot \mu^{\free}+\lim_{n\to\infty}\mu_n[\mathcal V_{n,\varepsilon_n}^{\wired}]\cdot \mu^{\wired}=\alpha\mu^{\free}+(1-\alpha)\mu^{\wired}\,.\qedhere
\]
\end{proof}

In light of Lemma~\ref{lem-mixture-ratio-characterization}, Theorem~\ref{thm-main} reduces to showing that for any $\alpha\in (0,1)$, there are uniform edge-expander graphs $G_n\stackrel{\operatorname{loc}}{\longrightarrow}\ttT_d$, such that the ratio of $\mathcal Z_{G_n}^{\free}$ to $\mathcal Z_{G_n}^{\wired}$ converges to $\gamma:=\tfrac{\alpha}{1-\alpha}$. Provided this is true, the edge case $\alpha\in \{0,1\}$ follows subsequently by a diagonal argument. In what follows we fix $\alpha\in (0,1)$ and $\gamma=\tfrac{\alpha}{1-\alpha}\in (0,\infty)$. 

A natural starting point is to consider random $d$-regular graphs, which are known to have good expansion properties and are locally $\ttT_d$-like. We perform this analysis in Section~\ref{subsec-random-d-regular-graphs}, where we compute sharp asymptotics of the partial Potts partition functions for typical random $d$-regular graphs. Surprisingly, we find that these are \emph{not sufficient} to establish Theorem~\ref{thm-main}. Specifically, there appears to be a non-trivial a.a.s.\ lower bound on the ratio $\mathcal Z_{G_n}^{\free}/\mathcal Z_{G_n}^{\wired}$ for $G_n$ sampled as a random $d$-regular graph.\footnote{While we do not include a full proof of this fact, it is supported by numerical evidence.} To overcome this obstacle, in Section~\ref{subsec-modification} we introduce and analyze certain modified graphs derived from random $d$-regular graphs. Using the cavity method, we study the asymptotics of their partial Potts partition functions. We then complete the proof of Theorem~\ref{thm-main} in Section~\ref{subsec-proof-of-main}.

\subsection{Asymptotic partial Potts partition functions of random $d$-regular graphs}\label{subsec-random-d-regular-graphs}

This section collects our results for random $d$-regular graphs. We start by introducing the precise random graph model of interest.

\begin{definition}[The pairing model $\calG_{n,d}$]
\label{def:pariing_model}
    For two integers $n,d$ such that $nd$ is even, we define a random $d$-regular graph $G_n$ on $V_n$ with $|V_n|=n$ as follows. For each $v\in V_n$, we assign $d$ half-edges connecting $v$. We then make a random pairing between the $dn$ half edges to form a $d$-regular graph $G_n$. We denote the distribution of $G_n$ as $\calG_{n,d}$. 
\end{definition}

\begin{remark}
Note that a graph $G_n \sim \mathcal G_{n,d}$ is not necessarily simple; self-loops and multiple edges may occur. However, the probability that $G_n$ is simple remains bounded away from zero as $n \to \infty$. Moreover, conditioned on the event that $G_n$ is simple, its distribution coincides with that of a uniform random simple $d$-regular graph on $V_n$.
\end{remark}

To state our results we need to introduce some notations. Define the Potts interaction matrix $B=(B_{ij})_{1\le i,j\le q}$ by
\begin{equation}\label{eq-B-ij}
    B_{ij} = e^{\beta\delta_{i,j}+\tfrac B d(\delta_{i,1}+\delta_{j,1})}\,,\quad\forall i,j\in [q]\,.
\end{equation}
For $\dagger \in \{{\free},{\wired}\}$, define the matrix $Q^{\dagger}=(Q^\dagger_{ij})_{1\le i,j\le q}$ by 
\begin{equation}\label{eq-Q-ij}
Q^\dagger_{ij}:= \frac{ B_{ij}\sqrt{\nu^\dagger(i)\nu^\dagger(j)}}{\sqrt{[\rightarrow i]^\dagger}\sqrt{[\rightarrow j]^\dagger}}\,,\quad\forall i,j\in [q]\,,
\end{equation}
where $[\rightarrow i]^\dagger:= \sum_{j\in[q]} \nu^\dagger(j) B_{ij}$. Then, $Q^\dagger$ is symmetric, positive semi-definite, and has $1$ as its top eigenvalue. Denote the remaining eigenvalues of $Q^\dagger$ as $\lambda^\dagger_2, \ldots, \lambda^\dagger_{q}$. Moreover, for $k\ge 3$, let $\theta_k=(d-1)^k/2k$ and $\delta_k^\dagger=\sum_{j=2}^q(\lambda_j^\dagger)^j,\dagger\in \{\free,\wired\}$. 

\begin{proposition}\label{prop-small-graph-conditioning-conclusion}
For $G_n\sim \mathcal G_{n,d}$, the following holds with high probability as $n\to\infty$:
\begin{equation}
    \mathcal Z_{G_n}^{\free}+\mathcal Z_{G_n}^{\wired}=(1-o(1))\mathcal Z_{G_n}^{\Po,\beta,B}\,,
\end{equation}
and
\begin{equation}\label{eq-small-graph-conditioning-result}
\frac{\mathcal Z_{G_n}^\dagger}{\mathbb{E}[\mathcal Z_{G_n}^\dagger]}=(1+o(1))\prod_{k\ge 3}(1+\delta_k^\dagger)^{X_{n,k}}e^{-\theta_k\delta_k^\dagger}\,,
\end{equation}
where $\mathbb{E}$ is taken over $G\sim \mathcal G_{n,d}$, and $X_{n,k}$ is the number of $k$-cycles in $G_n$, $k\ge 3$. Moreover, as $n\to\infty$, $\mathbb{E}[\mathcal Z_{G_n}^{\free}]/\mathcal \mathbb{E}[Z_{G_n}^{\wired}]$ converges to a constant $\Gamma=\Gamma(d,q,\beta,B)\in (0,\infty)$. 
\end{proposition}

The proof of Proposition~\ref{prop-small-graph-conditioning-conclusion} is based on the small subgraph conditioning method. Our proof closely follows the calculations in \cite{GSVY16}, and the details are deferred to Appendix~\ref{appendix-small-graph-conditioning}. We also record the following technical lemma that addresses the a.a.s. convergence of the infinite product in the right hand side of \eqref{eq-small-graph-conditioning-result}, whose proof is also provided in the appendix.
\begin{lemma}\label{lem-infinite-product}
For any $\dagger\in \{\free,\wired\}$, the infinite product $\prod_{k\ge 3}e^{-\theta_k\delta_k^\dagger}$ converges. Moreover, for any $\varepsilon>0$, there exists $K_0=K_0(\varepsilon)>0$ such that as $n\to\infty$, with probability at least $1-\varepsilon$ over $G_n\sim \mathcal G_{n,d}$, it holds for any $K\ge K_0$,
    \begin{equation}\label{eq-infinite-product}
    1\le \prod_{k>K}(1+\delta_k^\dagger)^{X_{n,k}}\le 1+\varepsilon\,.
    \end{equation}
\end{lemma}

Proposition~\ref{prop-small-graph-conditioning-conclusion} implies that for large $n$ and a typical graph $G_n$ sampled from $\mathcal G_{n,d}$, it holds
\begin{equation}\label{eq-ratio}
\frac{\mathcal Z_{G_n}^{\free}}{\mathcal Z_{G_n}^{\wired}}=(\Gamma+o(1))\prod_{k\ge 3}\left(\frac{1+\delta_k^{\free}}{1+\delta_k^{\wired}}\right)^{X_{n,k}}e^{-\theta_k(\delta_k^{\free}-\delta_k^{\wired})}\,.
\end{equation}
Numerical results suggest that it is always true that $\delta_k^{\free}>\delta_k^{\wired},\forall k\ge 3$, and thus there is a non-trivial lower bound on the right hand side of \eqref{eq-ratio} (since $X_{n,k}\ge 0,\forall k\ge 3$). This means that for a sequence of typical random $d$-regular graphs, any in-probability local weak limit $\alpha \mu^{\free}+(1-\alpha)\mu^{\wired}$ of $\{\mu_n\}=\{\mu_{G_n}^{\beta,B}\}$ has $\alpha$ bounded away from $0$.

\subsection{Modified graphs and the cavity method}\label{subsec-modification}

The discussion above suggests that, to establish Theorem~\ref{thm-main}, we must go beyond typical random $d$-regular graphs. To this end, we introduce certain modified graphs based on $d$-regular graphs and show that these modifications effectively overcome the aforementioned issue.

To begin, we note that the following dichotomy holds: either for all large $k\in \mathbb N$, $\delta_k^{\free}>\delta_k^{\wired}$, or for all large $k\in \mathbb N$ the reverse holds.\footnote{To see this, recall that $\delta_k^\dagger=\sum_{j=2}^q(\lambda_j^\dagger)^k$, where $\lambda_j^\dagger,2\le j\le q$ are the non-top eigenvalues of $Q^\dagger$. The only way the above dichotomy does not hold is when $\{\lambda_2^{\free},\dots,\lambda_q^{\free}\}=\{\lambda_2^{\wired},\dots,\lambda_q^{\wired}\}$. However, this is not the case as it is known in \cite[Lemma 21]{GSVY16} that $\lambda_q^{\free}\neq \lambda_q^{\wired}$ (see also Appendix~\ref{appendix-lem-infinite-product}).}
Let $\star\in \{\free,\wired\}$ such that $\delta_k^{\star}>{(\delta_k^{\free}+\delta_k^{\wired})}/{2}$ when $k$ is large enough. We denote $d_\star=d+1$ if $\star=\wired$ and $d_\star=d-1$ if $\star=\free$. 

Now we introduce the modified graphs of a $d$-regular graph $G$. 

\begin{definition}\label{def-modified-graph}
    For a $d$-regular graph $G$ and $p\in 2\mathbb N$, we say $\widehat{G}$ and $\widetilde{G}$ are \emph{$p$-modified graphs} of $G$, if $\widehat{G}$ is obtained from $G$ by deleting $m=pd_\star/2$ edges $(v_i,v_{m+i}),1\le i\le m$ in $G$, and $\widetilde{G}$ is then obtained by adding $p$ new vertices $w_1,\dots,w_p$ to $\widehat{G}$, each of which connects to $d_\star$ vertices among $\{v_1,v_2,\dots,v_{2m}\}$, such that each $v_i$ is connected to exactly one $w_k$. 
\end{definition}

In what follows, we fix a sequence of uniform edge-expander graphs $G_n\stackrel{\operatorname{loc}}{\longrightarrow}\ttT_d$ as well as a sequence of $\{\varepsilon_n\}$ chosen to satisfy \eqref{eq-choice-eps_n-1}-\eqref{eq-choice-eps_n-3}. For convenience, we further assume that 
\begin{equation}\label{eq-assumption}
    \liminf_{n\to\infty}\mu_n[\mathcal V_{n,\varepsilon_n}^\dagger]>0,\quad \forall \dagger\in \{\free,\wired\}\,.
\end{equation}

As we will show later, the modification serves for tuning the ratio of the free and wired partial Potts partition functions. Before making this precise, we need the following technical result.
Given $m$ edges $(v_i,v_{m+i}),1\le i\le m$ of $G_n$, we denote $\widehat{G}_n$ as the graph obtained from $G_n$ by deleting the $m$ edges, and let $\hat{\mu}_n$ be the Potts measure on $\widehat{G}_n$. Moreover, for $\dagger\in \{\free,\wired\}$, let $\hat{\nu}_n^{\dagger,2m}$ be the marginal on $(\sigma(v_1),\dots,\sigma(v_{2m}))$ of the measure $\hat{\mu}_n[\cdot\mid \mathcal V_{n,\varepsilon_n}^{\dagger}]$.\footnote{Note that everything here indeed depends on the choices of the edges $(v_i,v_{i+m}),1\le i\le m$, though we suppress the dependence fron the notation.}

\begin{proposition}\label{prop-marginal-for-cavity}
    For any $m\in \mathbb{N}$, let $v_i,1\le i\le m$ be chosen from $V_n$ uniformly and independently at random, and let $v_{m+i}$ be an arbitrary neighbor of $v_i$ in $G_n$, $1\le i\le m$. Recall that $\nu^{\free}$ and $\nu^{\wired}$ are the BP fixed points. Then, as $n\to\infty$, with probability $1-o(1)$ over the choices of $(v_i,v_{m+i}),1\le i\le m$, it holds that 
    \begin{equation}
        \label{eq-marginal-for-cavity}
        \operatorname{TV}\big(\hat{\nu}_n^{\dagger,2m},({\nu}^{\dagger})^{\otimes 2m}\big)=o(1)\,.
    \end{equation}
\end{proposition}

Recall that $\ttT_d=\ttT_d(o)$ is the $d$-regular tree rooted at $o$. Denote $1,\dots,d$ as the $d$ children of $o$ in $\ttT_d$, and for $v\in \{1,\dots,d\}$, let $\ttT_d(v\to o)$ be the $(d-1)$-array tree rooted at $i$ not containing $o$. Let $\mu_{v\to o}=\mu_{v\to o}^{\beta,B}$ be the Potts measure with inverse temperature $\beta$ and external field $B$ on $\ttT_d(v\to o)$. The proof of Proposition~\ref{prop-marginal-for-cavity} relies on the following non-reconstruction lemma.

\begin{lemma}\label{lem-non-recostruction}
    Whenever $(\beta,B)\in \mathbb{R}_+^2\setminus \mathsf R_{=}$, it holds that for any $\varepsilon>0$ and $\dagger\in \{\free,\wired\}$, 
    \begin{equation}\label{eq-non-reconstruction}
        \lim_{r\to\infty}\mu^{\dagger}\Big[\sigma_r\in [q]^{\partial \ttV_r}:\exists v\in \{1,\dots,d\},\operatorname{TV}\bigl( \mu_{v \to o} \bigl[ \sigma(v) = \cdot \,\big|\, \sigma|_{\partial \ttV_r} = \sigma_r \bigr] \bigr)
,\nu^\dagger\big)\ge \varepsilon\Big]=0\,.
    \end{equation}
\end{lemma}

\begin{proof}[Proof of Proposition~\ref{prop-marginal-for-cavity}]
    Fix an arbitrary $\delta>0$. We extend the mapping $\operatorname{BP}:\Delta^{\star}\to\Delta^\star$ defined as in \eqref{def:BP} to the mapping $\operatorname{BP}:(\Delta^*)^{d-1}\to \Delta^*$ defined by
    \[
        \operatorname{BP}(\nu_1,\dots,\nu_{d-1})(i)\propto e^{B\delta_{i,1}}\prod_{k=1}^{d-1}\sum_{j=1}^qe^{\beta\delta_{i,j}} \nu_k(j)\,,\quad \forall i\in [q]\,.
    \]
    Since the mapping is continuous and for $\dagger\in \{\free,\wired\}$, $\nu^\dagger$ satisfies $\operatorname{BP}(\nu^\dagger,\dots,\nu^\dagger)=\nu^\dagger$, we may choose $0<\varepsilon<\delta$ small enough such that for any $\dagger\in \{\free,\wired\}$, 
    \begin{equation}\label{eq-choice-delta}
    \operatorname{TV}(\nu_i,\nu^\dagger)\le \varepsilon\,,\forall 1\le i\le d-1\quad\Rightarrow\quad\operatorname{TV}\big(\operatorname{BP}(\nu_1,\dots,\nu_{d-1}),\nu^\dagger\big)\le \varepsilon\,.
    \end{equation}
    By Lemma~\ref{lem-non-recostruction}, we can choose $r=r(\varepsilon)\in \mathbb N$ such that 
    \[
    \mu^{\dagger}\Big[\sigma_r\in [q]^{\partial \ttV_r}:\exists v\in \{1,\dots,d\},\operatorname{TV}\big(\mu_{v\to o}[\sigma(v)=\cdot\mid \sigma\!\mid_{\partial \ttV_r}=\sigma_r],\nu^\dagger\big)\ge \varepsilon\Big]\le \delta\,,\quad\forall \dagger\in \{\free,\wired\}\,.
    \]
For a vertex $v\in V_n$ which satisfies $\NB_r(G_n,v)\cong \NB_r(\ttT_d,o)$, we denote by $v_1,\dots,v_d$ its neighbors in $G_n$. Moreover, we let $\mu_{G_n-(v,v_i)}$ be the Potts measure with inverse temperature $\beta$ and external field $B$ on the graph obtained from $G_n$ by deleting the edge $(v,v_i)$. For $\dagger\in \{\free,\wired\}$, define the set
\[
\mathcal G^\dagger(v):=\Big\{\sigma_r\in [q]^{\partial V_{n,r}(v)}:\operatorname{TV}\big(\mu_{G_n-(v,v_i)}[\sigma(v_i)=\cdot\mid \sigma\!\mid_{\partial V_{n,r}(v_i)}=\sigma_r],\nu^\dagger\big)\le \varepsilon,\forall 1\le i\le d\Big\}\,.
\]
We call a vertex $v\in V_n$ good, if $\NB_r(G_n,v)\cong \NB_r(\ttT_d,o)$, and $$\mu_n^\dagger\big[\sigma\!\mid_{\partial V_{n,r}(v)}\notin \mathcal G^\dagger(v)\big]\le 2\delta,\quad\forall \dagger\in \{\free,\wired\}\,.$$
%(denote by $v_1,\dots,v_d$ the $d$ neighbors of $v$ in $G_n$)
 %   \[
 % \mu_n^\dagger\Big[\sigma_r\in [q]^{\partial V_{n,r}(v)}:\operatorname{TV}\big(\mu_{v_i\to v}[\sigma(v_i)=\cdot\mid \sigma\!\mid_{\partial V_{n,r}(v_i)}=\sigma_r],\nu^\dagger\big)\ge \varepsilon\Big]<2\delta
  %  \]
  %  holds for any $1\le i\le d$ and $\dagger\in \{\free,\wired\}$. Here, $\mu_{v_i\to v}$ denotes the Potts measure restrict to the $(d-1)$-array sub-tree of $G_n$ rooted at $v_i$ not containing $v$ with depth $r$. 
    
    By our assumption and Proposition~\ref{prop-lwcp}, we have $\mu_n^\dagger\stackrel{\operatorname{lwcp}}{\longrightarrow}\mu^\dagger$, $\forall \dagger\in \{\free,\wired\}$. It follows that as $n\to\infty$, a uniform vertex $v\in V_n$ is good with probability $1-o(1)$. Therefore, if we pick $v_1,\dots,v_m\in V_n$ uniformly and independently at random, with probability $1-o(1)$, every $v_i$ is good, and $\operatorname{dist}_{G_n}(v_i,v_j)\ge 3r,\forall 1\le i<j\le m$. We claim that this implies
    \[
    \operatorname{TV}\big(\hat{\nu}_n^{\dagger,2m},(\nu^\dagger)^{\otimes 2m}\big)= O(\delta)\,,\quad\forall \dagger\in \{\free,\wired\}\,,
    \]
    and thus the desired result follows by sending first $n\to\infty$ and then $\delta\to 0$. 

    To show the claim, we fix $m$ good vertices $v_1,\dots,v_m\in V_n$ such that they are $3r$-away from each other on $G_n$. We denote by $V_n'=V_n\setminus \cup_{i=1}^m V_{n,r-1}(v_i)$, and we define the sets
    \[
    \mathcal G_n^\dagger:=\Big\{\check{\sigma}\in [q]^{V_n'}:\check{\sigma}\!\mid_{\partial V_{n,r}(v_i)}\in \mathcal G^\dagger(v_i),\forall 1\le i\le m\Big\}
    \]
    and 
    \[
    \mathcal H_n^\dagger:=\Big\{\check{\sigma}\in [q]^{V_n'}:\min_{a\in \{\pm 1\}}\big|\#\{v\in V_n':\sigma(v)=i\}-(\check{\nu}^\dagger(i)+a\varepsilon_n)|V_n|\big|\ge 2md^{r},\forall i\in [q]\Big\}\,.
    \]
    Since each $v_i$ is good, the union bound yields that $$\mu_n^\dagger\big[\sigma\!\mid_{V_n'}\in \mathcal G_n^\dagger\big]\ge 1-O(\delta),\quad\forall \dagger\in \{\free,\wired\}$$
    Additionally, by Lemma~\ref{lem-local-CLT} and similar arguments as in the proof of Lemma~\ref{lem-subsequential-limit} (recall our assumption \eqref{eq-assumption}), we obtain that as $n\to\infty$, $$\mu_n^\dagger\big[\sigma\!\mid_{V_n'}\in \mathcal H_n^\dagger\big]=1-o(1)\,,\quad\forall \dagger\in \{\free,\wired\}\,.$$ 
    Consequently, we have for large enough $n$, $\mu_n^\dagger\big[\sigma\!\mid_{V_n'}\in \mathcal G_n^\dagger\cap \mathcal H_n^\dagger\big]\ge 1-O(\delta)$, $\forall \dagger\in \{\free,\wired\}$.

    Let $v_{ik},1\le k\le d$ be the neighbors of $v_i$ in $G$. For any choices of $v_{m+i}\in \{v_{i1},\dots,v_{id}\},1\le i\le m$ and the graph $\widehat{G}_n$ obtained from $G_n$ by deleting the edges $(v_i,v_{m+i}),1\le i\le m$, it is straightforward to check that the two measures $\mu_n^\dagger,\hat{\mu}_n^\dagger$ restricting on $\sigma\!\mid_{V_n'}$ are contiguous with each other (i.e., their likelihood-ratio is uniformly bounded both from below and from above). Therefore, it follows that
    \begin{equation}\label{eq-hat-mu}
\hat{\mu}_n^\dagger\big[\sigma\!\mid_{V_n'}\in \mathcal G_n^\dagger\cap \mathcal H_n^\dagger\big]\ge 1-O(\delta)\,,\quad\forall \dagger\in \{\free,\wired\}\,.
    \end{equation}

    Now, fix $\dagger\in \{\free,\wired\}$ and a realization of $\check{\sigma}=\sigma\!\mid_{V_n'}$ that lies in $\mathcal G_n^\dagger\cap \mathcal H_n^\dagger$. Conditioned on that $\sigma\!\mid_{V_n'}=\check{\sigma}$ under $\mu_n^\dagger$, $\check{\sigma}\in \mathcal H_n^\dagger$ implies that the remaining spin configurations $\sigma\!\mid_{V_{n,r-1}(v_i)},1\le i\le m$ are conditionally independent, and have the distributions of $$\mu_{G_n-(v_i,v_{m+i})}\big[\sigma\!\mid_{V_{n,r-1}(v_i)}=\cdot\mid \sigma\!\mid_{\partial V_{n,r}(v_i)}=\check{\sigma}\!\mid_{\partial V_{n,r}(v_i)}\big],\quad 1\le i\le m\,.$$
    Additionally, we have $\sigma(v_i),\sigma(v_{m+i}),1\le i\le m$ are conditionally independent of each other.
    Since $\check{\sigma}\in \mathcal G_n^\dagger$, we have
    \[
    \operatorname{TV}\big(\mu_{G_n-(v_i,v_{ik})}[\sigma(v_{ik})=\cdot\mid \sigma\!\mid_{\partial V_{n,r}(v_i)}=\check{\sigma}\!\mid_{\partial V_{n,r}(v_i)}],\nu^\dagger\big)\le \varepsilon\,,\quad \forall 1\le i\le m,1\le k\le d\,.
    \]
    In particular, we have for all $1\le i\le m$, $\mu_{G_n-(v_i,v_{m+i})}[\sigma(v_{m+i})=\cdot\mid \sigma\!\mid_{\partial V_{n,r}(v_i)}=\check{\sigma}\!\mid_{\partial V_{n,r}(v_i)}]$ is within $\delta$ TV-distance to $\nu^\dagger$ (as $\varepsilon<\delta$). Moreover, for all $1\le i\le m$, by \eqref{eq-choice-delta} we get
    \begin{align*}
    &\ \mu_{G_n-(v_i,v_{m+i})}\big[\sigma(v_i)=\cdot\mid \sigma\!\mid_{\partial V_{n,r}(v_i)}=\check{\sigma}\!\mid_{\partial V_{n,r}(v_i)}\big]\\
    =&\ \operatorname{BP}\Big(\big\{\mu_{G_n-(v_i,v_{ik})}[\sigma[v_{ik}=\cdot\mid \sigma\!\mid_{\partial V_{n,r}(v_i)}=\check{\sigma}\!\mid_{\partial V_{n,r}(v_i)}]\big\}_{1\le k\le d:v_{ik}\neq v_{m+i}}\Big)
    \end{align*}
    is also within $\delta$ TV-distance to ${\nu}^\dagger$. This proves that the joint distribution $\hat{\nu}^{\dagger,2m}$ of $\sigma(v_i),1\le i\le 2m$ conditioned on $\check{\sigma}\in \mathcal G_n^\dagger\cap \mathcal H_n^\dagger$, is $O(\delta)$ TV-close to $(\nu^\dagger)^{\otimes 2m}$. Combining with \eqref{eq-hat-mu}, the claim follows and we conclude the proof.
\end{proof}

By Proposition~\ref{prop-marginal-for-cavity}, for any $p\in 2\mathbb N$, we can construct $p$-modified graphs $\widehat{G}_n,\widetilde{G}_n$ as defined in Definition~\ref{def-modified-graph}, such that \eqref{eq-marginal-for-cavity} holds as $n\to\infty$. The following Proposition characterizes how this modification alters the partial Potts partition functions.
For $\dagger\in \{\free,\wired\}$ and $d_\star\in\{d-1,d+1\}$, we define
\[
\Delta_{d_\star,\dagger}:=\frac{\sum_{i=1}^qe^{B\delta_{i,1}}\Big(\sum_{j=1}^q{\nu}^\dagger(j)e^{\beta\delta_{i,j}}\Big)^{d_\star}}{\big(\sum_{i,j=1}^q{\nu}^\dagger(i){\nu}^\dagger(j)e^{\beta\delta_{i,j}}\big)^{d_\star/2}}\,.
\]

\begin{proposition}\label{prop-partition-function-ratio}
    For uniform edge-expander graphs $G_n\stackrel{\operatorname{loc}}{\longrightarrow}\ttT_d$ and any $p\in 2\mathbb N$, assume that the $p$-modified graphs $\{\widehat{G}_n\},\{\widetilde{G}_n\}$ of $\{G_n\}$ are such that \eqref{eq-marginal-for-cavity} holds as $n\to\infty$. Then, as $n\to\infty$,
\begin{equation}\label{eq-partition-function-ratio}
\frac{\mathcal Z_{\widetilde{G}_n}^{\dagger}}{\mathcal Z^\dagger_{G_n}}= \Delta_{d_\star,\dagger}^{p}+o(1)\,,\quad\dagger\in \{\free,\wired\}\,.
\end{equation}
\end{proposition}

\begin{proof}
    %By Proposition~\ref{prop-marginal-for-cavity}, we can choose $m$ edges $(v_i,v_{m+i}),1\le i\le m$ of each graph $G_n$, such that for $\widehat{G}_n$ defined as above, the marginal distribution on $\sigma(v_i),1\le i\le 2m$ of $\hat{\mu}_n[\cdot\mid \mathcal V_{n,\varepsilon_0}^{\dagger}]$ is $o(1)$ TV-close to $(\nu^{\dagger})^{\otimes 2m}$ as $n\to\infty$. 
    Fix $\dagger\in \{\free,\wired\}$. By definition, we have
\begin{equation}\label{eq-cavity-2}
\frac{\mathcal Z_{G_n}^{\dagger}}{\mathcal Z_{\widehat{G}_n}^{\dagger}}=\mathbb{E}_{\sigma\sim\hat{\mu}_n^\dagger}\Big[\exp\Big(\beta\sum_{i=1}^m\delta_{\sigma(v_i),\sigma(v_{m+i})}\Big)\Big]\,.
\end{equation}
Note that the above expression only depends on the joint distribution of $\sigma(v_i),1\le i\le 2m$ under $\hat{\mu}_n^\dagger$. \eqref{eq-marginal-for-cavity} yields that as $n\to\infty$,
\begin{equation}\label{eq-cavity-result-1}
\frac{\mathcal Z_{G_n}^{\dagger}}{\mathcal Z_{\widehat{G}_n}^{\dagger}}=\Bigg(\sum_{i,j=1}^q\nu(i)\nu(j)e^{\beta\delta_{i,j}}\Bigg)^{m}+o(1)\,. 
\end{equation}

Analogously, it is easy to see from Lemma~\ref{lem-local-CLT} that for $\dagger\in \{\free,\wired\}$, $\mathcal Z_{\widetilde{G}_n}^\dagger$ is $o(\mathcal Z_{G_n}^{\Po,\beta,B})$ plus
\[
\sum_{\substack{\sigma\in \mathcal V_{n,\varepsilon_0}^{\dagger}\\\tilde{\sigma}\in [q]^p}}\exp\Bigg(\beta\Big(\sum_{(u,v)\in \widehat{E}_n}\delta_{\sigma(u),\sigma(v)}+\sum_{k=1}^p\sum_{v\sim w_k}\delta_{\sigma(v),\tilde{\sigma}(w_k)}\Big)+B\Big(\sum_{v\in V_n}\delta_{\sigma(v),1}+\sum_{k=1}^p\delta_{\tilde{\sigma}(w_k),1}\Big)\Bigg)\,,
\]
where $w_1,\dots,w_p$ are the $p$ vertices vertices in $\widetilde{V}_n\setminus V_n$. Since we assume in \eqref{eq-assumption} that $$\liminf_{n\to\infty}\mu_n[\mathcal V_{n,\varepsilon_n}^\dagger]=\liminf_{n\to\infty}\frac{\mathcal Z_{G_n}^{\dagger}}{\mathcal Z_{G_n}^{\Po,\beta,B}}>0\,,$$ it follows that as $n\to\infty$, 
\begin{equation}\label{eq-cavity-3}
\begin{aligned}
\frac{\mathcal Z_{\widetilde{G}_n}^{\dagger}}{\mathcal Z_{\widehat{G}_n}^{\dagger}}=&\ \mathbb{E}_{\sigma\sim \hat{\mu}_n^\dagger}\Bigg[\sum_{\tilde{\sigma}\in [q]^p}\exp\Big(\beta\sum_{k=1}^p\sum_{v\sim w_k}\delta_{\sigma(v),\tilde{\sigma}(w_k)}+B\sum_{k=1}^p\delta_{\tilde{\sigma}(w_k),1}\Big)\Bigg]+o(1)\\
\stackrel{\eqref{eq-marginal-for-cavity}}{=}&\ \Bigg(\sum_{i=1}^qe^{B\delta_{i,1}}\Big(\sum_{j=1}^q\nu(j)e^{\beta\delta_{i,j}}\Big)^{d_\star}\Bigg)^{p}+o(1)\,.
\end{aligned}
\end{equation}
Combining \eqref{eq-cavity-result-1} and \eqref{eq-cavity-3}, we obtain that as $n\to\infty$, 
\[
\frac{\mathcal Z_{\widetilde{G}_n}^{\dagger}}{\mathcal Z_{G_n}^{\dagger}}=\frac{\mathcal Z_{\widetilde{G}_n}^{\dagger}}{\mathcal Z_{\widehat{G}_n}^{\dagger}}\Bigg/\frac{\mathcal Z_{G_n}^{\dagger}}{\mathcal Z_{\widehat{G}_n}^{\dagger}}=\Delta_{d_\star,\dagger}^p+o(1)\,.\qedhere
\]
\end{proof}

\subsection{Proof of Theorem~\ref{thm-main}}\label{subsec-proof-of-main}

Now we are ready to give the proof of Theorem~\ref{thm-main}. We first record the following numerical lemma, whose proof is purely elementary and is deferred to Appendix~\ref{appendix-lem-Delta_{d+1}}. 

\begin{lemma}\label{lem-Delta_{d+1}}
    $\Delta_{d-1,\free}>\Delta_{d-1,\wired}$ and $\Delta_{d+1,\free}<\Delta_{d+1,\wired}$.
\end{lemma}

\begin{proof}[Proof of Theorem~\ref{thm-main}]
We only need to prove the case $\alpha\in (0,1)$ as the edge case then follows from a diagonal argument. Moreover, we only prove the case that $\star=\wired$, as the case $\star=\free$ follows from an almost identical argument.

Let $\gamma=\tfrac{\alpha}{1-\alpha}\in (0,\infty)$ be fixed.
Recall $\theta_k,\delta_k^\dagger,k\ge 3,\dagger\in \{\free,\wired\}$ as defined in Section~\ref{subsec-random-d-regular-graphs}, and $\Gamma=\lim_{n\to\infty}\mathbb{E}[\mathcal Z_{G_n}^{\free}]/\mathbb{E}[\mathcal Z_{G_n}^{\wired}]$ where the expectation is taken over $G_n\sim \mathcal G_{n,d}$. 
By Lemma~\ref{lem-Delta_{d+1}}, we can choose $p\in 2\mathbb N,p=O(1)$ such that $$\Gamma\cdot\prod_{k\ge 3}e^{-\theta_k(\delta_k^{\free}-\delta_k^{\wired})}\cdot  \frac{\Delta_{d+1,\free}^p}{\Delta_{d+1,\wired}^{p}}<\gamma\,.$$ 

Fix an integer $n$. Since we have $\delta_k^{\free},\delta_k^{\wired}\to 0$ as $k\to\infty$, we can pick $K\in \mathbb N$ large enough such that $K>K_0(n^{-1})$ (as defined in Lemma~\ref{lem-infinite-product}) and $1<(1+\delta_{K}^{\free})/(1+\delta_{K}^{\wired})<1+n^{-1}$. Then, there exists $x\in \mathbb N$ such that 
    \begin{equation}\label{eq-choice-x_kn}
    \Gamma\cdot\frac{\prod_{k\ge 3}e^{-\theta_k\delta_k^{\free}}\cdot \Delta_{d+1,\free}^p}{\prod_{k\ge 3}e^{-\theta_k\delta_k^{\wired}}\cdot \Delta_{d+1,\wired}^{p}}\cdot\left(\frac{1+\delta_{K}^{\free}}{1+\delta_{K}^{\wired}}\right)^{x}\in [\gamma,(1+n^{-1})\gamma)\,.
    \end{equation}

    We claim that for each $n\in \mathbb{N}$, there exists $N=N(n)\in \mathbb N$ such that for $G_N\sim \mathcal G_{N,d}$, the following holds with positive probability:\\
    (i) $\mathcal Z_{G_N}^{\free}+\mathcal Z_{G_N}^{\wired}\ge (1-n^{-1})\mathcal Z_{G_N}^{\Po,\beta,B}$, and for $\dagger\in \{\free,\wired\}$, it holds that
    \begin{equation}\label{eq-ingrdient-1}
\frac{\mathcal Z_{G_N}^\dagger}{\mathbb{E}[\mathcal Z_{G_N}^\dagger]\cdot \prod_{k\ge 3}(1+\delta_k^\dagger)^{X_{N,k}}e^{-\theta_k\delta_k^\dagger}}\in(1-n^{-1},1+n^{-1})\,.
    \end{equation}
    \noindent(ii) $X_{N,3}=\cdots=X_{N,K-1}=0$ and $X_{N,K}=x$, and for $\dagger\in \{\free,\wired\}$,   \begin{equation}\label{eq-ingredient-2}
    1\le \prod_{k>K}(1+\delta_k^\dagger)^{X_{N,k}}\le 1+n^{-1}\,.
    \end{equation}
    \noindent (iii) the second eigenvalue of the adjacency matrix of $G_N$ is no more than $2\sqrt{d-1}+0.1$. 
    
    To see this, we have from Proposition~\ref{prop-small-graph-conditioning-conclusion} that as $N\to\infty$, item-(i) holds with probability $1-o(1)$ over $G_N\sim \mathcal G_{N,d}$.  Additionally, since it is well-known (see, e.g., \cite{bollobas1980probabilistic}) that $\{X_{n,k}\}_{k\ge 3}$ weakly converges to independent $\operatorname{Poi}(\theta_k)$ variables, it follows from Lemma~\ref{lem-infinite-product} that item-(ii) holds with non-vanishing probability for large $N$. Finally, it is well known by \cite{Fri08} that as $N\to\infty$, item-(iii) holds with probability $1-o(1)$. Thus, the claim follows from the union bound.

     Consequently, we see that for $N=N(n)$ large enough, there exists a $d$-regular graph $G_n$ on $N(n)$ vertices, such that items-(i), (ii), (iii) all hold. Furthermore, by the definition of $\Gamma$, we may pick $N(n)$ large enough, such that $
            {\mathbb{E}[\mathcal Z_{G_N}^{\free}]}/{\mathbb{E}[\mathcal Z_{G_N}^{\wired}]}\in ((1-n^{-1})\Gamma,(1+n^{-1})\Gamma)$. 
 Combining this with \eqref{eq-ingrdient-1}, \eqref{eq-ingredient-2}, we obtain that 
      \begin{equation}\label{eq-ingredient-3}
    (1-n^{-1})^3\cdot \gamma\cdot \frac{\Delta_{d+1,\free}^{-p}}{\Delta_{d+1,\wired}^{-p}}\le \frac{\mathcal Z_{G_n}^{\free}}{\mathcal Z_{G_n}^{\wired}}\le     (1+n^{-1})^3\cdot \gamma\cdot \frac{\Delta_{d+1,\free}^{-p}}{\Delta_{d+1,\wired}^{-p}}\,.
    \end{equation}

    Now consider the graph sequence $\{G_n\}$ constructed as above. By item-(i) we have that as $n\to\infty$, $\mathcal Z_{G_n}^{\Po,\beta,B}=(1+o(1))(\mathcal Z_{G_n}^{\free}+\mathcal Z_{G_n}^{\wired})$, and thus \eqref{eq-ingredient-3} implies that the assumption \eqref{eq-assumption} is satisfied. Additionally, by item-(ii) we have $\operatorname{girth}(G_n)\to\infty$ as $n\to\infty$, and by item-(iii) we see that $G_n$ are uniform edge-expander graphs. Applying Propositions~\ref{prop-marginal-for-cavity} and~\ref{prop-partition-function-ratio} to the graph sequence $\{G_n\}$ with $p$ chosen as above, we obtain that there is a sequence of modified graphs $\widetilde{G}_n$ of $G_n$ that satisfies $\operatorname{girth}(\widetilde{G}_n)\to\infty$, and
    \[
     \lim_{n\to\infty}\frac{\mathcal Z_{\widetilde{G}_n}^{\free}}{\mathcal Z_{\widetilde{G}_n}^{\wired}}=\gamma\cdot \frac{\Delta_{d+1,\free}^{-p}}{\Delta_{d+1,\wired}^{-p}}\cdot \frac{\Delta_{d+1,\free}^{p}}{\Delta_{d+1,\wired}^{p}}=\gamma\,.
    \]
    Since $\widetilde{G}_n$ are still uniform edge-expander graphs such that $\widetilde{G}_n\stackrel{\operatorname{loc}}{\longrightarrow}\ttT_d$,
it follows from Lemma~\ref{lem-mixture-ratio-characterization} that the in-probability local weak limit of $\widetilde{\mu}_n=\mu_{\widetilde{G}_n}^{\beta,B}$ is $\alpha{\mu^{\free}}+(1-\alpha){\mu^{\wired}}$, as desired. 
\end{proof}

\appendix

\section{Potts partition functions of random $d$-regular graphs}\label{appendix-small-graph-conditioning}

In this section we prove Proposition~\ref{prop-small-graph-conditioning-conclusion}. We first introduce some notations. Fix $d,q\ge 3$ and $(\beta,B)\in \mathsf R_c$. Recall the BP fixed points $\nu^{\free,\beta,B},\nu^{\wired,\beta,B}$ as well as the measures $\check{\nu}^{\free,\beta,B},\check{\nu}^{\wired,\beta,B}\in \Delta^\star$, which we denote as $\nu^{\free},\nu^{\wired},\check{\nu}^{\free},\check{\nu}^{\wired}$ for simplicity. For any measure $\alpha\in \Delta^\star$ and $\delta>0$, we write $\operatorname{B}_{\infty}(\alpha,\delta)$ as the $\ell^\infty$-ball of radius $\delta$ in $\Delta^*$ centered at $\alpha$. 
For any $n\in \mathbb N$, we define
\begin{equation*}
    \Delta^\star(n):=\{\alpha \in \Delta^\star: n\alpha(i) \in \mathbb N\,,\ \forall i\in [q]\}\,.
\end{equation*}
Additionally, for $G_n=(V_n,E_n)$ sampled from the pairing model $\mathcal G_{n,d}$, we define for a spin configuration $\sigma\in [q]^{V_n}$ its color profile
$
\operatorname{L}_\sigma\in \Delta^\star
$ by
\[
\operatorname{L}_\sigma(i)={n^{-1}}\#\{v\in V_n:\sigma(v)=i\}\,,\quad\forall i\in [q]\,.
\]
Moreover, for $\alpha\in \Delta^\star(n)$, we define $\mathcal V_{n}^\alpha:=\{\sigma\in [q]^{V_n}:\operatorname{L}_\sigma=\alpha\}$. 

We consider the $q$-state Potts model with inverse temperature $\beta$ and external field $B$ on $G\sim \mathcal G_{n,d}$.
For $\alpha\in \Delta^\star(n)$, we let
\begin{equation*}
    \mathcal Z^{\alpha}_{G_n}:= \sum_{\sigma\in\mathcal{V}_{n}^{\alpha}} \exp\Big(\beta\sum_{(u,v)\in E_n}\delta_{\sigma(u),\sigma(v)}+B\sum_{v\in V_n}\delta_{\sigma(v),1}\Big)\,,
\end{equation*}
Note that by our definition,
\begin{equation}\label{eq-sum-Z_G_n^dagger}
\mathcal Z_{G_n}^{\dagger}=\sum_{\alpha\in \Delta^\star(n)\cap \operatorname{B}_\infty(\check{\nu}^\dagger,\varepsilon_0)}\mathcal Z_{G_n}^\alpha\,.
\end{equation}

Take $\delta_n:=n^{-1/4}$ and for $\dagger\in \{\free,\wired\}$, we let
\begin{equation*}
\label{eq:Zdagger}
    \mathcal Z^{\dagger}(G_n):= \sum_{\alpha\in \Delta^\star(n)\cap \operatorname{B}_\infty(\check{\nu}^\dagger,\delta_n)}\mathcal Z_{G_n}^\alpha\,.
\end{equation*}
The key intuition is that $\mathcal Z_{G_n}^{\dagger}$ is well-approximated by $\mathcal Z^\dagger(G_n)$, as the dominant contributions to the sum \eqref{eq-sum-Z_G_n^dagger} come from $\mathcal Z_{G_n}^\alpha$ with $\alpha\in \operatorname{B}_\infty(\check{\nu}^\dagger,\delta_n)$.
Additionally, we write
\[
\mathcal Z^0(G_n):=\mathcal Z_{G_n}^{\Po,\beta,B}-\mathcal Z^{\free}(G_n)-\mathcal Z^{\wired}(G_n)\,.
\]

Our plan is to show that as $n\to\infty$,
\begin{equation}\label{eq-random-d-regular-graphs-goal-1}
\mathbb{E}[\mathcal Z^{\free}(G_n)]\asymp \mathbb{E}[\mathcal Z^{\wired}(G_n)]\gg \mathbb{E}[\mathcal Z^0(G_n)]\,,
\end{equation}
and with probability $1-o(1)$ over $G_n\sim \mathcal G_{n,d}$, 
\begin{equation}\label{eq-random-d-regular-graph-goal-2}
\frac{\mathcal Z^\dagger(G_n)}{\mathbb{E}[\mathcal Z^\dagger(G_n)]}=(1+o(1))\prod_{k\ge 3}(1+\delta_k^\dagger)^{X_{n,k}}e^{-\theta_k\delta_k^\dagger}\,.
\end{equation}
Provided that \eqref{eq-random-d-regular-graphs-goal-1} and \eqref{eq-random-d-regular-graph-goal-2} are true, Proposition~\ref{prop-small-graph-conditioning-conclusion} follows readily. The proof of the two claims are detailed in Sections~\ref{subsec-first-moment} and~\ref{subsec-small-graph-conditioning}, respectively. 

%The key intuition that lies behind the definition of $\mathcal Z^{\free}(G_n),\mathcal Z^{\wired}(G_n)$ and $\mathcal Z^0(G_n)$ is that for typical $G_n\sim \mathcal G_{n,d}$, $\mathcal Z^{\free}(G_n)$ and $\mathcal Z^{\wired}(G_n)$ contribute the majority of $\mathcal Z(G_n)$, and these two parts are comparable with each other (only) at the critical temperature. 

%There are two main steps to make this intuition precise (as detailed respectively in Section~\ref{subsec-first-moment} and Section~\ref{subsec-small-graph-conditioning}). We first compute the expected values of $\mathcal Z^{\free}(G_n),\mathcal Z^{\wired}(G_n)$ and $Z^0(G_n)$ over $G_n\sim \mathcal G_{n,d}$. It turns out that

%We then argue, using the small graph conditioning method, that $\mathcal Z^{\free}(G_n)/\mathbb{E}[\mathcal Z^{\free}(G_n)],\mathcal Z^{\free}(G_n)/\mathbb{E}[\mathcal Z^{\free}(G_n)]$ are two random constants in $(0,\infty)$ determined by the short cycle counting statistics of $G_n$. Consequently, from Markov's inequality, we conclude that with high probability over $G_n\sim \mathcal G_{n,d}$,
%\[
%\mathcal Z^{\free}(G_n)
%\asymp \mathcal Z^{\wired}(G_n)\gg \mathcal Z^0(G_n)\,.
%\]
%Furthermore, the small graph conditioning method yields a precise characterization of the ratio $\mathcal Z^{\free}(G_n)/\mathcal Z^{\wired}(G_n)$, which in turn relates to the mixture ratio of the free and wired measures in the limiting local weak limit of $\mu_{\beta,h}(G_n)$.  

\subsection{First moment calculation}\label{subsec-first-moment}

To calculate the asymptotics of $\mathbb{E}[\mathcal Z^{\dagger}(G_n)],\dagger\in \{\free,\wired\}$, we need to understand the asymptotics of $\mathbb{E}[\mathcal Z_{G_n}^\alpha]$ for $\alpha\in \operatorname{B}_\infty(\check\nu^\dagger,\delta_n)$,
\begin{comment}
we write it as a sum:
\begin{equation*}
    \mathbb{E}[\mathcal Z^{\dagger}(G_n)] = \sum_{\alpha\in\Delta^\star(n,q): \|\alpha - \check{\nu}^\dagger\|_\infty < \delta_n } \mathbb{E} [\mathcal Z^{\alpha}(G_n)],
\end{equation*}
where 
\begin{equation*}
    \mathcal Z^{\alpha}(G_n) := \sum_{\sigma\in\mathcal{V}_{0}^{\alpha}(G_n)} \exp\Big(\beta\sum_{i\sim j}\mathbf{1}\{\sigma(i)=\sigma(j)\}+h\sum_{i}\{\sigma(i)=1\}\Big)\,,
\end{equation*}
and 
\begin{equation*}
    \Delta^\star(n,q):=\{\alpha \in \Delta^\star: \alpha(i) = k(i)/n, \text{ for some integer $k(i)$, for all $i\in [q]$}\}.
\end{equation*}
\end{comment}
which has essentially been obtained in \cite{GSVY16}. Recall the $q\times q$ matrices $B$ and $Q^\dagger,\dagger\in\{\free,\wired\}$ as defined in \eqref{eq-B-ij} and \eqref{eq-Q-ij}. The matrix $B$ can be viewed as the interaction matrix of the Potts measure with inverse temperature $\beta$ and external field $B>0$ on a $d$-regular graph. Also recall that $\lambda_1^\dagger=1,\lambda_2^\dagger,\dots,\lambda_q^\dagger$ are the eigenvalues of $Q^\dagger$. 
\begin{comment}
Let $\{B_{ij}\}_{i,j\in [q]}$ be the interaction matrix of Potts on a $d$-regular graph, i.e., 
\begin{equation*}
    B_{ij} = \bigg(e^{\beta + 2h/d}\mathbf{1}_{\{i=1\}} + e^\beta \mathbf{1}_{\{i\neq 1\}} \bigg)\mathbf{1}_{\{i=j\}} + \bigg(1+(e^{h/d}-1)(\mathbf{1}_{\{i=1\}} + \mathbf{1}_{\{j=1\}} )\bigg)\mathbf{1}_{\{i\neq j\}}.
\end{equation*}
Let $E_{ij}$ be the number of edges between vertices of color $i$ and $j$. We renormalize it by setting $x_{ij} = E_{ij} / nd$ if $i\neq j$ and $x_{ii} = 2E_{ii} / nd$ so that $\sum_j x_{ij} = \alpha(i)$ for all $i$. For $\dagger \in \{{\free},{\wired}\}$, define matrix $Q^{\dagger}$ by 
$Q^\dagger_{ij}:= \frac{ B_{ij}\sqrt{\nu^\dagger_i\nu^\dagger_j}}{\sqrt{[\rightarrow i]^\dagger}\sqrt{[\rightarrow j]^\dagger}} $, where $[\rightarrow i]^\dagger:= \sum_{\sigma\in[q]} \nu^\dagger_\sigma B_{i\sigma}.$ Then $Q^\dagger$ is symmetric and has $1$ as an eigenvalue. Denote the remaining eigenvalues as $\lambda^\dagger_1, \ldots, \lambda^\dagger_{q-1}$. Lastly, let $B_{\infty}(\alpha, r)$ be the $\ell^\infty$-ball of radius $r$ around $\alpha$.
\end{comment}

\begin{lemma}[Equations (74), (75) of \cite{GSVY16}]\label{lemma:first_second_moments_accurate}
    For $G_n\sim \mathcal G_{n,d}$ and $\balpha\in \operatorname{B}_{\infty}(\check{\nu}^{\dagger},\delta_n)$, we have as $n\to\infty$
    \begin{equation}\label{eq:one-profile-asym}
        \mathbb{E}[\mathcal Z^{\balpha}_{G_n}] = (1 + o(1)) (2\pi n)^{-(q-1)/2} \bigg(\prod_{i=1}^q \alpha(i) \prod_{i=2}^q (1 + \lambda^\dagger_i) \bigg)^{-1/2} e^{n\Upsilon_1(\balpha,x^\star)},
    \end{equation}
    \begin{multline}
        \mathbb{E}[(\mathcal Z^{\balpha}_{G_n})^2] = (1 + o(1)) (2\pi n )^{-(q-1)} \bigg(\prod_{i=1}^q \alpha(i) \prod_{i=2}^q (1 + \lambda^\dagger_i) \bigg)^{-1}\\
        \times\prod_{i=2}^{q} \prod_{j=2}^{q} (1 - (d-1)\lambda_i^\dagger\lambda_j^\dagger)^{-1/2} e^{n\Upsilon_2(\balpha, \bgamma^\star,\by^\star)},
    \end{multline}
    where 
\begin{equation*}
        \Upsilon_1(\balpha, \bx) := (d-1) \sum_i \alpha(i) \log \alpha(i)
        +  d\bigg(\frac{1}{2} \sum_{i,j} x_{ij} \log B_{ij} - \frac{1}{2} \sum_{i,j} x_{ij} \log x_{ij}\bigg),
    \end{equation*}
    \begin{equation*}
        \Upsilon_2(\balpha, \bgamma, \by):= (d-1)\sum_{i,k}\gamma_{ik}\log \gamma_{ik} + d \bigg(\frac{1}{2}\sum_{i,k,j,l} y_{ikjl} \log (B_{ij}B_{kl}) - \frac{1}{2}\sum_{i,k,j,l}y_{ikjl}\log y_{ikjl}\bigg),
    \end{equation*}
    and $\bx^\star:= \arg\max_{\bx} \Upsilon_1(\balpha,\bx),(\bgamma^\star,\by^\star) = \arg\max_{\bgamma,\by} \Upsilon_2(\balpha, \bgamma,\by)$ are the unique maximizers subject to the constraints $\sum_{j} x_{ij} = \alpha(i),\, \forall i \in [q],
        x_{ij} = x_{ji} \geq 0,\, \forall i,j \in [q],
        \sum_{k} \gamma_{ik} = \alpha(i),\, \forall i \in [q],
        \sum_{i} \gamma_{ik} = \alpha(k),\, \forall k \in [q],
        \sum_{j,l} y_{ikjl} = \gamma_{ik},\, \forall i,k \in [q],
        \gamma_{ik} \geq 0, \forall i,j \in [q], y_{ikjl} = y_{jlik} \geq 0, \forall i,k,j,l \in [q].$
\end{lemma}

Note that since $\bx^\star$ and $(\bgamma^\star,\by^\star)$ are functions of $\balpha$, with slight abuse of notation we can denote $\Upsilon_1(\balpha):= \Upsilon_1(\balpha,\bx^\star)$ and $\Upsilon_2(\balpha):= \Upsilon_2(\balpha,\bgamma^\star,\by^\star)$. We have the following relation.

\begin{lemma}[Theorem 6 of \cite{GSVY16}]\label{lemma:twice}
    For $\dagger \in \{\free,\wired\}$, $\Upsilon_2(\check{\nu}^\dagger) = 2 \Upsilon_1(\check{\nu}^\dagger).$
\end{lemma}

The same paper also gives exact formulas for $\bx^\star(\check{\nu}^\dagger)$ and $\by^\star(\check{\nu}^\dagger)$.

\begin{lemma}[In the paragraphs after Lemma 26 and Remark 9 in \cite{GSVY16}]\label{lemma:xstar}
    For $\dagger \in  \{\free, \wired\}$, the unique maximizers $\bx^\star(\check{\nu}^\dagger)$ and $(\bgamma^\star(\check{\nu}^\dagger),\by^\star(\check{\nu}^\dagger))$ have the following expressions:
    \begin{gather}\label{eq:xstar}
        x_{ij}^\star(\check{\nu}^\dagger) = \frac{B_{ij} \nu^\dagger_i\nu^\dagger_j}{\sum_{k,l} B_{kl}\nu^\dagger_k\nu^\dagger_l}\,,\quad
        \gamma_{ij}^\star(\check{\nu}^\dagger) = \check{\nu}^\dagger_{i}\check{\nu}^\dagger_{j}\,,\quad
        y^\star_{ikjl}(\check{\nu}^\dagger) = x^\star_{ij}(\check{\nu}^\dagger)x^\star_{kl}(\check{\nu}^\dagger).
    \end{gather}
\end{lemma}

In fact, we can show the following more general result via a perturbation argument.

\begin{corollary} \label{cor:xstar}
For every $\balpha \in B_\infty(\check{\nu}^\dagger, \delta_n)$, there exists $\nu \in \Delta^\star$ such that $\|\nu - \nu^\dagger\|_\infty = o(1)$ and 
\begin{gather}\label{eq:x_star_general}
    x_{ij}^\star(\balpha):=\frac{B_{ij}\nu_i \nu_j}{\sum_{k,\ell} B_{k\ell}\nu_k\nu_\ell}\,, \quad
    \gamma_{ij}^\star(\balpha) = \alpha(i)\alpha(j)\,,\quad
        y^\star_{ikjl}(\balpha) = x^\star_{ij}(\balpha)x^\star_{kl}(\balpha)\,.
\end{gather}
are the maximizers of $\Upsilon_1(\balpha,\bx)$ and $\Upsilon_2(\balpha,\bgamma,\by)$ subject to the constraints in lemma \ref{lemma:first_second_moments_accurate}.
Moreover, $\Upsilon_1(\balpha)$ and $\Upsilon_2(\balpha)$ are differentiable in the neighborhood $B_\infty(\check{\nu}^\dagger, \delta_n)$.
\end{corollary}

\begin{proof}
We begin by observing that $\Upsilon_1(\alpha,x)$ can be decomposed as $\Upsilon_1(\alpha,x) = g_1(x) + f_1(\alpha)$,
where $g_1(x) =  d\big(\frac{1}{2} \sum_{i,j} x_{ij} \log B_{ij} - \frac{1}{2} \sum_{i,j} x_{ij} \log x_{ij}\big)$ and 
$f_1(\alpha) =  (d-1) \sum_i \alpha(i) \log \alpha(i)$. Thus, $\arg\max_x \Upsilon_1(\alpha,x) = \arg\max_x g_1(x)$. Since $g_1$ is strictly concave in its convex domain, it has a unique maximizer. One can check that the expression given in \eqref{eq:x_star_general} defines a critical point of $g_1$ for any $\nu$. Once we determine $\nu\in\Delta^\star$ such that the corresponding $x^\star$ satisfies the corresponding constraints, this critical point will be the unique maximizer of $g_1(x)$, and hence $\Upsilon_1(\alpha,x)$. Equivalently, we now seek $\nu\in\Delta^\star$ such that $\nu_i \sum_j B_{ij}\nu_j = C\alpha(i), \forall i \in [q]$, where $C$ is a normalizing constant. We now incorporate the interaction $B_{ij}$ into $\nu$ by setting $\tilde{\nu_i} = \nu_i e^{\frac{h}{d} \mathbf{1}_{j=1}}$. Without loss of generality, we may normalize so that $\sum_i \tilde{\nu_i} = 1$, by dividing through by the appropriate normalizing constant. With this reparametrization, the defining equations reduce to $\tilde{\nu_i}((e^\beta - 1)\tilde{\nu_i} + 1) = \tilde{C}\alpha(i)$, for some $\tilde{C}>0$, for all $i\in [q]$. By choosing $\tilde{C}$ appropriately, this gives a unique positive root $\tilde{\nu}$ such that $\sum_i \tilde{\nu}_i = 1$. The explicit form of $\nu$ shows that it is differentiable w.r.t. $\balpha$. Thus, $\Upsilon_1$ is differentiable w.r.t. $\balpha$ as is it a composite of differentiable maps. The case for $\gamma^\star_{ij}$, $y^\star_{ikjl}$ and $\Upsilon_2$ can be similarly proven. Finally, it remains to show that $\nu$ is close to $\nu^\dagger$. Consider the map $F: \Delta^\star \to \Delta^\star$ that takes $\nu$ to $\alpha$. As shown previously, $F$ is continuous and bijective. As $\Delta^\star$ is compact, $F^{-1}$ is also continuous. Therefore, small perturbations in $\alpha$ imply small perturbations in $\nu$. In particular, we have $\|\nu - \nu^\dagger\|_\infty = o(1)$.
\end{proof}

An immediate consequence of the above corollary is an extension of lemma \ref{lemma:twice} to $\balpha \in B_\infty(\check{\nu}^\dagger, \delta_n)$.
\begin{corollary}\label{cor:twice_extension}
    For all $\balpha \in B_\infty(\check{\nu}^\dagger, \delta_n)$, we have $\Upsilon_2({\balpha}) = 2 \Upsilon_1(\balpha)$.
\end{corollary}

Based on these calculations, we have the following.
\begin{proposition}
Assume that $(\beta,B)\in \mathbb{R}_+^2\setminus\mathsf R_=$. For $G_n\sim \mathcal G_{n,d}$ and $\dagger \in \{\free,\wired\}$, it holds that as $n\to\infty$,
    \[
    \mathbb{E}[\mathcal Z^{\dagger}(G_n)]=(1+o(1))M_n^{\dagger}\,, 
    \]
    where
    \begin{equation}\label{eq-Mf}
        M_n^{\dagger}:= \bigg(\prod_{i=1}^q \check{\nu}^\dagger[i] \prod_{i=2}^q (1 + \lambda^\dagger_i) \bigg)^{-1/2} \det\bigg(-\nabla^2 \Upsilon_1(\check{\nu}^{\dagger})\bigg) e^{n\Upsilon_1(\check{\nu}^{\dagger})}\,.
    \end{equation}

    Moreover, there exists $\theta=\theta(d,q,h)>0$ such that for all large $n\in \mathbb N$, 
    \[
    \mathbb{E}[\mathcal Z^0(G)]\le \exp(-\theta\sqrt n)\max\{M_n^{\free},M_n^{\wired}\}\,.
    \]
\end{proposition}

\begin{proof}
By corollary \ref{cor:xstar}, for all $\balpha\in \operatorname{B}_{\infty}(\check{\nu}^{\dagger},\delta_n)$, we have
   \begin{equation*}
   \begin{aligned}
       \Upsilon_1(\balpha) =&\ (d-1) \sum_i \check{\nu}^\dagger_i \log \check{\nu}^\dagger_i - \frac{d}{2} \bigg(\sum_{k,l} B_{kl}\nu^\dagger_k\nu^{\dagger}_l\bigg)^{-1} \bigg[\sum_{k,l} B_{kl}\nu^\dagger_k\nu^{\dagger}_l \log(\nu^\dagger_k\nu^{\dagger}_l)\bigg] \\
       &\ +\frac{d}{2} \log\bigg(\sum_{k,l} B_{kl}\nu^\dagger_k\nu^{\dagger}_l\bigg) + o(1)\,.
       \end{aligned}
   \end{equation*}
    Moreover, if we define
   \begin{equation*}
       C^\dagger(n,\balpha):= (2\pi n)^{-(q-1)/2} \bigg(\prod_{i=1}^q \alpha(i) \prod_{i=2}^q (1 + \lambda^\dagger_i) \bigg)^{-1/2},
   \end{equation*}
   then $C^\dagger(n,\balpha) = (1 + o(1))C(n,\check{\nu}^\dagger)$, for all $\balpha\in \operatorname{B}_{\infty}(\check{\nu}^{\dagger},\delta_n)$. It then follows from \eqref{eq:one-profile-asym} that 
   \begin{equation*}
       \mathbb{E}[\mathcal Z^{\dagger}(G_n)] = \sum_{\balpha\in \operatorname{B}_{\infty}(\check{\nu}^{\dagger},\delta_n) \cap \Delta^\star(n,q)} \mathbb{E}[\calZ^{\balpha}(G_n)] = (1 + o(1)) C^\dagger(n,\check{\nu}^\dagger) \sum_{\balpha\in \operatorname{B}_{\infty}(\check{\nu}^{\dagger},\delta_n) \cap \Delta^\star(n,q)} e^{n\Upsilon_1(\balpha)}.
   \end{equation*}
    We can rewrite the sum as integrals,
   \begin{equation*}
       \mathbb{E}[\mathcal Z^{\dagger}(G_n)] = (1+o(1)) C^\dagger(n,\check{\nu}^\dagger) \int_0^n \cdots \int_0^n   e^{n\Upsilon_1(\lfloor \balpha \rfloor/n)} \mathbf{1}_{\{\|\lfloor \balpha \rfloor/n - \check{\nu}^\dagger\| < \delta_n\}} \mathbf{1}_{\{\lfloor \balpha \rfloor/n \in \Delta^\star(n,q)\}} d\balpha,
   \end{equation*}
   where there are $q-1$ integrals on the RHS as this is the number of free variables and $\lfloor \balpha \rfloor := (\lfloor\alpha_1 \rfloor, \ldots, \lfloor\alpha_q \rfloor)$. 
   Using the fact that $\lfloor \balpha\rfloor/n=\balpha/n+O(1/n)$ and applying the change of variable $\balpha = n \check{\nu}^\dagger + y \sqrt{n}$, we get 
   \begin{equation*}
       \mathbb{E}[\mathcal Z^{\dagger}(G_n)]=(1+o(1)) C^\dagger(n,\check{\nu}^\dagger) \int_{-n\delta_n}^{n\delta_n} \ldots  \int_{-n\delta_n}^{n\delta_n}  e^{n \Upsilon_1(\check{\nu}^\dagger + y/\sqrt{n} + O(1/n))}(\sqrt{n})^{q-1}\mathbf{1}_{\mathcal{Y}_n}\, dy,
   \end{equation*}
   where $\mathcal{Y}_n$ denotes the simplex $\Delta^\star(n,q)$ after change of variables. Note that $\mathcal{Y}_n \to \bbR^{q-1}$ as $n\to\infty$.
    For a fixed $y$, we Taylor expand $\Upsilon_1$ around $\check{\nu}^\dagger$ using $\nabla \Upsilon_1(\check{\nu}^\dagger)=0$,
   \begin{equation*}
       n\Upsilon_1(\check{\nu}^\dagger + y/\sqrt{n} + O(1/n))) = n\Upsilon_1(\check{\nu}^\dagger) + \frac{1}{2} y^T\nabla^2 \Upsilon_1 (\check{\nu}^\dagger)y  + o(1).
   \end{equation*}
   Consequently, 
   \begin{equation*}
       \begin{split}
           \mathbb{E}[\mathcal Z^{\dagger}(G_n)] &= (1+o(1))C^\dagger(n,\check{\nu}^\dagger)  \int_{-n\delta_n}^{n\delta_n} \ldots \int_{-n\delta_n}^{n\delta_n} e^{n\Upsilon_1(\check{\nu}^\dagger) + \frac{1}{2} y^T\nabla^2 \Upsilon_1 (\check{\nu}^\dagger)y +o(1)} (\sqrt{n})^{q-1}\mathbf{1}_{\mathcal{Y}_n} \, dy\\
            &= (1+o(1)) (2\pi n)^{-(q-1)/2} \bigg(\prod_{i \in [q]} \check{\nu}^\dagger[i] \prod_{i=2}^q (1 + \lambda^\dagger_i) \bigg)^{-1/2} e^{n \Upsilon_1(\check{\nu}^\dagger)} (\sqrt{n})^{q-1}\\
            &\qquad \times\int_{-\infty}^\infty \ldots \int_{-\infty}^\infty e^{\frac{1}{2} y^T \nabla^2 \Upsilon_1(\check{\nu}^\dagger)y(1+o(1))} \, dy \\
            &= (1+o(1)) \bigg(\prod_{i \in [q]} \check{\nu}^\dagger[i] \prod_{i=2}^q (1 + \lambda^\dagger_i) \bigg)^{-1/2} \det(-\nabla^2\Upsilon_1(\check{\nu}^\dagger)) \, e^{n\Upsilon_1(\check{\nu}^\dagger)}.
       \end{split}
   \end{equation*}
For the second statement, since $\nabla \Upsilon_1 (\check{\nu}^\dagger)$ is negative definite for $\dagger\in \{\free,\wired\}$,\footnote{This follows \cite[Theorem 9]{GSVY16} and Claim~\ref{claim-lambda_2} in Appendix~\ref{appendix-lem-infinite-product}.} from Taylor's expansion we obtain that $\sup_{B^c_{\infty}(\check{\nu}^{\free},\delta_n) \cap B^c_{\infty}(\check{\nu}^{\wired},\delta_n)} \Upsilon_1(\balpha) < \max\{\Upsilon_1(\check{\nu}^{\free}), \Upsilon_1(\check{\nu}^{\wired})\} - c \delta_n^2$, for some $c>0$. Thus,  
\[
\bbE \calZ^0(G) = \sum_{\balpha \in B^c_{\infty}(\check{\nu}^{\free},\delta_n) \cap B^c_{\infty}(\check{\nu}^{\wired},\delta_n)} e^{n \Upsilon_1(\balpha) + O(\log n)} \leq [e^{n \Upsilon_1(\check{\nu}^{\free}) - c\sqrt{n} } +e^{n \Upsilon_1(\check{\nu}^{\wired}) - c\sqrt{n} } ]\cdot \operatorname{poly}(n),
\]
where the first equality comes from (22) of \cite{DMSS14} and the polynomial factor comes from the number of terms in the sum. Thus, $\bbE \calZ^0(G) \leq e^{-c\sqrt{n}}\max\{M_n^{\free},M_n^{\wired}\}$.
\end{proof}

Note that when $(\beta,B)\in \mathsf R_c$, it holds that $\Upsilon_1(\check{\nu}^{\free})=\Upsilon_1(\check{\nu}^{\wired})$, so $M_n^\dagger\asymp M_n^{\wired}$. Moreover, on the critical line we have $\mathbb{E}[\mathcal Z^0(G_n)]=o(\min\{\mathbb{E}[\mathcal Z^{\free}(G_n)],\mathbb{E}[\mathcal Z^{\wired}(G_n)\})]$, and thus $\mathbb{E}[\mathcal Z_{G_n}^{\dagger}]=(1+o(1))\mathbb{E}[\mathcal Z^{\dagger}(G_n)]=(1+o(1))M_n^\dagger$,  $\dagger\in \{\free,\wired\}$. It follows that as $n\to\infty$, $\mathbb{E}[\mathcal Z^{\free}_{G_n}]/\mathbb{E}[\mathcal Z_{G_n}^{\wired}]$ converges to a constant $\Gamma\in (0,\infty)$, verifying the last statement in Proposition~\ref{prop-small-graph-conditioning-conclusion}.

\subsection{The small subgraph conditioning method}\label{subsec-small-graph-conditioning}
To obtain the sharp characterization of $\mathcal Z^\dagger(G_n)$ as in \eqref{eq-random-d-regular-graph-goal-2}, we opt for a the small subgraph conditioning method introduced by \cite{wormald1999models}. We will use the following version of the theorem.

\begin{proposition}[Small subgraph conditioning] \label{thm:ssc}
For $ k \in \mathbb N$, let $\theta_k > 0 $ and $\delta_k \geq -1$ be constants. Suppose for each $n \in \mathbb N$, we have random variables $\{X_{n,k}\}_{k=1}^\infty, Y_n$  such that 
\begin{enumerate}
    \item $X_{n,k} \in \mathbb N$ and $X_{n,k} \stackrel{n \to \infty}{\longrightarrow} Z_k$, where $Z_k \sim \text{Poi}(\theta_k)$ are independent,
    \item For any non-negative integers $0<k_1<\cdots<k_t$ and $x_1, \ldots, x_t$, it holds that as $n\to\infty$,
    \begin{equation*}
        \frac{\bbE \left[Y_n | X_{n,k_1} = x_1, \ldots, X_{n,k_t} = x_t\right]}{\bbE[Y_n]} \to \prod_{i=1}^t (1+\delta_{k_i})^{x_i} e^{-\theta_{k_i}\delta_{k_i}},
    \end{equation*}
    \item $\sum_{k=1}^\infty \theta_k \delta_k^2 < \infty$,
    \item $\bbE [Y_n^2] / (\bbE [Y_n])^2 \leq \exp(\sum_{k=1}^\infty \theta_k\delta_k^2) + o(1)$ as $n \to \infty$.
\end{enumerate}
Then we have as $n\to\infty$,
\begin{equation}
\frac{Y_n}{ \bbE [Y_n]\prod_{k=1}^\infty (1 + \delta_k)^{X_{n,k}} e^{-\theta_k \delta_k}}  \stackrel{\text{in probability}}{\longrightarrow} 1\,. 
\end{equation}
%as $n \to \infty$, where with a slight abuse of notation, we identify the embedded variables with the original ones. Moreover, the RHS converges a.s. to $\prod_{i=1}^\infty (1 + \delta_i)^{Z_i} e^{-\theta_i \delta_i}$ which is strictly positive iff $\delta_i>-1$ for all $i$.
\end{proposition}

\begin{proof}
    By Skorokhod’s embedding theorem, we may assume that $\{X_{n,k}\}_k, \{Y_n\}_n$, and $\{Z_k\}_k$ are defined on a common probability space, with each $(Y_n, \{X_{n,k}\}_k)$ preserving its joint distribution and such that the weak convergences all become almost sure convergences. We continue to use the same notation for the embedded variables. To ease notation, we define the following quantities:
    \begin{align*}
        A_n &:=  \prod_{k=1}^\infty (1 + \delta_k)^{X_{n,k}} e^{-\theta_k\delta_k}, \qquad A_n^m:=  \prod_{k\leq m} (1 + \delta_k)^{X_{n,k}} e^{-\theta_k \delta_k} \\
        W &:=  \prod_{k=1}^\infty (1 + \delta_k)^{Z_k} e^{-\theta_k \delta_k}, \qquad W^m:= \prod_{k\leq m} (1 + \delta_k)^{Z_k} e^{-\theta_k \delta_k}\\
        \calF_n^m &:= \sigma(X_{n,1}, \ldots, X_{n,m}), \qquad Y_n^m:= \bbE [Y_n | \calF_n^m].
    \end{align*}
By replacing $Y_n$ with $Y_n/\bbE[ Y_n]$, we may assume without loss of generality that $\bbE[Y_n] = 1$. We want to show that for any $\varepsilon>0$. $\bbP \big[|Y_n - A_n| > 4\varepsilon \big] \to 0$ as $n \to \infty$. An application of triangle inequality upper bounds this probability by four other quantities
\begin{align}\label{eq:SSC-tria}
    \bbP \big[|Y_n - A_n| > 4\varepsilon \big] & \leq \bbP \big[|Y_n - Y_n^m| > \varepsilon \big] + \bbP \big[|Y_n^m - W^m| > \varepsilon \big] \\
    & + \bbP \big[|W^m - A_n^m| > \varepsilon \big] + \bbP \big[|A_n^m - A_n| > \varepsilon \big],
\end{align}
which we will bound one by one. The proof of Theorem 1 in \cite{janson1995random} shows
\begin{equation}\label{eq:SSC1}
    \limsup_{n\to\infty} \bbP \big[|Y_n - Y_n^m| > \varepsilon \big] \leq \varepsilon^{-2} \Bigg[\exp\Big(\sum_{k=1}^\infty \theta_k\delta_k^2\Big) - \exp\Big(\sum_{k=1}^m \theta_k\delta_k^2\Big)\Bigg],
\end{equation}
and 
\begin{equation}\label{eq:SSC2}
    \limsup_{n\to\infty} \bbP \big[|Y_n^m - W^m| > \varepsilon \big] = 0.
\end{equation}
Moreover, by condition 1 and continuous mapping theorem, 
\begin{equation}\label{eq:SSC3}
     \limsup_{n\to\infty} \bbP \big[|W^m - A_n^m| > \varepsilon \big] = 0.
\end{equation}
Finally, 
\begin{align*}
    \limsup_{n\to\infty} \bbE [|A_n^m - A_n|^2] &\leq \limsup_{n\to\infty} \bbE[(A_n^m)^2] + \limsup_{n\to\infty} \bbE[(A_n)^2] - 2\lim\inf_{n\to\infty} \bbE [A_n A_n^m] \\
    &\leq \bbE[(W^m)^2] + \bbE[W^2] + \bbE[WW^m]\\
    &= \exp\Big(\sum_{k=1}^\infty \theta_k\delta_k^2\Big) - \exp\Big(\sum_{k=1}^m \theta_k\delta_k^2\Big),
\end{align*}
where we have used Fatou's lemmas in the second line. Thus by Chebyshev's inequality, we have
\begin{equation}\label{eq:SSC4}
    \limsup_{n\to\infty}\bbP \big[|A_n^m - A_n| > \varepsilon \big] \leq \varepsilon^{-2} \Bigg[\exp\Big(\sum_{k=1}^\infty \theta_k\delta_k^2\Big) - \exp\Big(\sum_{k=1}^m \theta_k\delta_k^2\Big)\Bigg].
\end{equation}
Combining \eqref{eq:SSC-tria}, \eqref{eq:SSC1}, \eqref{eq:SSC2}, \eqref{eq:SSC3} and \eqref{eq:SSC4}, and sending $m\to\infty$ clinch the proof.
\end{proof}

In the pairing model $\mathcal G_{n,d}$, $X_{n,k}$ is the number of cycles of length $k$ in $G_n$. It was shown in \cite{bollobas1980probabilistic} that $X_{n,k}$ are asymptotically independent Poisson with parameter $\theta_k = (d-1)^k/2k$, as $n\to\infty$. We now take $Y_n = \calZ^\dagger(G_n),\dagger\in \{\free,\wired\}$. \eqref{eq-random-d-regular-graph-goal-2} follows upon verifying the conditions in Proposition~\ref{thm:ssc}, as we do next.

\begin{comment}
\begin{proposition} \label{prop:ratio}
    Recall that $G_n \sim \calG_{n,d}$. As $n \to \infty$, for $\dagger \in \{\free,\wired\}$, $\mathcal Z^{\dagger}(G_n)/\mathbb{E}[\mathcal Z^{\dagger}(G_n)]$ converges to some value in $(0,\infty)$ determined by the short-cycle statistics of $G_n$. More specifically, 
    \begin{equation*}
        \frac{\mathcal Z^\dagger(G_n)}{\mathbb{E}[\mathcal Z^\dagger(G_n)]}=(1+o(1))\prod_{k\ge 3}(1+\delta_k^\dagger)^{X_{k,n}}e^{-\theta_k\delta_k^\dagger}\,.
    \end{equation*}
\end{proposition}
\end{comment}

\begin{proof}[Proof of \eqref{eq-random-d-regular-graph-goal-2}]
    Our proof is adapted from the proofs of lemma 29-32 in \cite{GSVY16}. Fix $\dagger\in \{\free,\wired\}$. It remains to verify conditions 2 and 4 in Proposition~\ref{thm:ssc} with the choices $\theta_k = \frac{(d-1)^k}{2k}, \delta_k = \sum_{j=2}^{q} (\lambda_j^\dagger)^k, k\ge 3$. For condition 2, in particular, via a standard reduction\footnote{see the proof of Theorems 9.19 and 9.23 in \cite{janson2011random}.} it suffices to show that for any $k\ge 3$, 
    \begin{equation*}
        \frac{\sum_{\alpha\in \operatorname{B}_{\infty}(\check{\nu}^{\dagger},\delta_n) \cap \Delta^\star(n,q)} \mathbb{E}[\mathcal Z^\alpha(G_n) X_{n,k}]}{\sum_{\alpha\in \operatorname{B}_{\infty}(\check{\nu}^{\dagger},\delta_n) \cap \Delta^\star(n,q)} \mathbb{E}[\mathcal Z^\alpha(G_n)]} \to \theta_k(1+\delta_k).
    \end{equation*}
    We will achieve this by proving 
    \begin{equation}\label{eq:ratio1}
        \frac{\mathbb{E}[\mathcal Z^\alpha(G_n) X_{n,k}]}{\mathbb{E}[\mathcal Z^\alpha(G_n)]} \to \theta_k (1+\delta_k), \, \forall \alpha\in \operatorname{B}_{\infty}(\check{\nu}^{\dagger},\delta_n) \cap \Delta^\star(n,q).
    \end{equation}
    The proof of Lemma 30 in \cite{GSVY16} establishes the above for $\alpha \in \{\check{\nu}^{\free}, \check{\nu}^{\wired}\}$, but the argument depends on this specific choice of $\alpha$ only through the fact that $x_{ij}^\star(\alpha)$ has the form given in \eqref{eq:xstar}. However, we have shown in corollary \ref{cor:xstar} that $x_{ij}^\star(\alpha) = x_{ij}^\star(\check{\nu}^\dagger) + o(1)$, and therefore \eqref{eq:ratio1} holds. 
    
    To check condition 4, we need to show that
    \begin{equation}
        \frac{\sum_{\alpha_1,\alpha_2\in \operatorname{B}_{\infty}(\check{\nu}^{\dagger},\delta_n) \cap \Delta^\star(n,q)} \bbE [\calZ^{\alpha_1}(G_n)\calZ^{\alpha_2}(G_n)]}{\sum_{\alpha_1,\alpha_2\in \operatorname{B}_{\infty}(\check{\nu}^{\dagger},\delta_n) \cap \Delta^\star(n,q)} \bbE [\calZ^{\alpha_1}(G_n)]\bbE [\calZ^{\alpha_2}(G_n)]} \leq \exp \Big(\sum_{k=1}^\infty \theta_k \delta_k^2\Big) + o(1).
    \end{equation}
    Similarly, it suffices to show that
    \begin{equation*}
        \frac{\bbE[ \calZ^{\alpha_1}(G_n)\calZ^{\alpha_2}(G_n)]}{\bbE [\calZ^{\alpha_1}(G_n)]\bbE [\calZ^{\alpha_2}(G_n)]} \leq \exp \Big(\sum_{k=1}^\infty \theta_k \delta_k^2\Big) + o(1),
    \end{equation*}
    for all $\alpha_1,\alpha_2\in \operatorname{B}_{\infty}(\check{\nu}^{\dagger},\delta_n) \cap \Delta^\star(n,q)$. By Jensen's inequality, we have
    \begin{equation*}
        \frac{\bbE [\calZ^{\alpha_1}(G_n)\calZ^{\alpha_2}(G_n)]}{\bbE [\calZ^{\alpha_1}(G_n)]\bbE [\calZ^{\alpha_2}(G_n)]} \leq \frac{\sqrt{\bbE [(\calZ^{\alpha_1}(G_n))^2]}}{\bbE [\calZ^{\alpha_1}(G_n)]} \cdot \frac{\sqrt{\bbE [(\calZ^{\alpha_2}(G_n))^2]}}{\bbE [\calZ^{\alpha_2}(G_n)]}.
    \end{equation*}
    By lemma \ref{lemma:first_second_moments_accurate} and corollary \ref{cor:twice_extension}, we have that for any $\alpha\in \operatorname{B}_{\infty}(\check{\nu}^{\dagger},\delta_n) \cap \Delta^\star(n,q)$,
    \[
    \frac{\bbE [(\calZ^{\balpha}(G_n))^2]}{(\bbE [\calZ^{\balpha}(G_n)])^2} = (1 + o(1)) \prod_{i=2}^{q} \prod_{j=2}^{q} (1 - (d-1)\lambda_i^\dagger\lambda_j^\dagger)^{-1/2}\,.
    \]
 The term on the right hand side is shown to be equal to $\exp(\sum_{k=1}^\infty\theta_k \delta_k^2) + o(1)$ as $n\to\infty$, in \cite{GSVY16}. This proves condition 4 and thus verifies \eqref{eq-random-d-regular-graph-goal-2}, thereby concluding Proposition~\ref{prop-small-graph-conditioning-conclusion}. 
    %Following \cite{GSVY16}, let $\calS = \{S_1,\ldots,S_q\}$ be a partition of $V$ such that $|S_i|=\alpha(i) n$ for all $i \in [q]$ and let $Y_\calS$ be the total weight of configurations with color scheme $\calS$. Further, $X_{n,i}$ admits the following decomposition:%
   %\begin{equation*}
        %X_{n,i} = \frac{1}{2i} \sum_{\zeta} \sum_{\zeta} \mathbf{1}_{\zeta,\zeta},
    %\end{equation*}%
    %where the outer sum is over all rooted and oriented $i$-cycle $\zeta$, the inner sum is over its pre-image the pairing model and the indicator $\mathbf{1}_{\zeta,\zeta}$ signals the presence of $(\zeta,\zeta)$ in $G_n$. Using the definition of the pairing model and Stirling's formula, we have
    %The proof in \cite{GSVY16} shows that
    %\begin{equation*}
       % \frac{\mathbb{E}[\mathcal Z_G^\alpha X_{n,i}]}{\mathbb{E}[\mathcal Z_G^\alpha]} \to \frac{(d-1)^i}{2i} \cdot \sum_{a^\prime} N_{a} \prod_{i<j} \bigg(\frac{x_{ij}^\star(\alpha)}{\sqrt{\alpha(i)\alpha(j)}}\bigg)^{a_{ij}},
    %\end{equation*}
   %where the sum is over all tuples $a= (a_{11},\ldots,a_{qq})$ such that $a_{ij}$ is the number of edges of the type $\{i,j\}$ in a rooted and oriented $i$-cycle in $G_n$, and $N_a$ is the number of possible $\zeta$ with $a_{ij}$ edges of the type $\{i,j\}$.
   %By corollary \ref{cor:xstar}, 
   %\begin{equation*}
       %\frac{x_{ij}^\star}{\sqrt{\alpha(i)\alpha(j)}} = \frac{B_{ij} \sqrt{\alpha(i)\alpha(j)}}{\sum_{k,l} B_{kl} \alpha(k)\alpha_l} = \frac{B_{ij} \sqrt{\check{\nu}^\dagger[i]\check{\nu}^\dagger[j]} + o(1)}{\sum_{k,l} B_{kl} \check{\nu}^\dagger[i]\check{\nu}^\dagger[j] +o(1)}.
   %\end{equation*}
\end{proof}

\section{Proof of technical and numerical lemmas}

\subsection{Proof of Lemma~\ref{lem-numerical}}\label{appendix-lem-numerical}

\begin{proof}
Define the function
\[
H_\bbeta(t)\equiv \left(\frac{e^{2\bbeta} t+1}{t+e^{2\bbeta}}\right)^{d-1}\,,
\]
it is straightforward to check that $\tfrac{x(\bbeta)}{1-x(\bbeta)}$ corresponds to the largest fixed point of $H_\bbeta$. Since $1$ is a trivial fixed point of $H_\bbeta$ (corresponding to $x(\bbeta)=1/2$),
we have that $\beta^{\operatorname{Uni}}_d$ is the unique positive solution of the equation
\[
H'_\bbeta(1)=(d-1)\tanh(\bbeta)=1\,,
\]
and thus $\beta^*>\beta^{\operatorname{Uni}}_d$ is equivalent to that $(d-1)\tanh(\beta^*)>1$. By \eqref{eq-beta^*} we have
\[
\tanh(\beta^*)=\frac{e^{2\beta^*}-1}{e^{2\beta^*}+1}=\frac{\sqrt{(1+w)(1+w/(q-1)}-1}{\sqrt{(1+w)(1+w/(q-1))}+1}\,.
\]
Since $B<B_+$ where $B_+$ is defined in \eqref{eq-B_+}, we have $\sqrt{(1+w)(1+w/(q-1))}>\tfrac d {d-2}$, and thus $\tanh(\beta^*)>\tfrac{1}{d-1}$, as desired. This proves that $\beta^*>\beta^{\operatorname{Uni}}_d$. 

We now proceed to verify \eqref{eq-rel-psi-m}. By the belief propagation of random cluster measure \eqref{eq-BP-RCM}, we see $\psi^{\free}$ and $\psi^{\wired}$ can be characterized as follows: let $p^{\free}$ and $p^{\wired}$ be the smallest and the largest solution in $(0,1)$ to the equation
\begin{equation}\label{eq-BP-fixed-point-RCM}
    p=\frac{e^B\big(p(1+w)+(1-p)(1+w/q)\big)^{d-1}-\big(p+(1-p)(1+w/q)\big)^{d-1}}{e^B\big(p(1+w)+(1-p)(1+w/q)\big)^{d-1}+(q-1)\big(p+(1-p)(1+w/q)\big)^{d-1}}\,,
\end{equation}
then $\psi^{\dagger},\dagger\in \{\free,\wired\}$ is given by
\begin{equation}\label{eq-RCM-marginal}
    \psi^\dagger=\frac{e^B\big(p^\dagger(1+w)+(1-p^\dagger)(1+w/q)\big)^{d}-\big(p^\dagger+(1-p^\dagger)(1+w/q)\big)^{d}}{e^B\big(p^\dagger(1+w)+(1-p^\dagger)(1+w/q)\big)^{d}+(q-1)\big(p^\dagger+(1-p^\dagger)(1+w/q)\big)^{d}}\,.
\end{equation}
Denote $$\lambda^\dagger:=\frac{1+(q-1)p^\dagger}{1-p^\dagger},\quad\dagger\in \{\free,\wired\}\,.$$ Plugging the relation $w=e^\beta-1$, \eqref{eq-BP-fixed-point-RCM} simplifies to that
\begin{equation}\label{eq-numerical-1}
\lambda^\dagger=\frac{e^B}{q-1}\left(\frac{e^\beta \lambda^\dagger+q-1}{\lambda^\dagger+e^\beta+q-2}\right)^{d-1}\,,\quad\forall \dagger\in \{\free,\wired\}\,,
\end{equation}
and \eqref{eq-RCM-marginal} implies that $s^\dagger:=\tfrac{q-1}{q}\psi^\dagger+\tfrac{1}{q}$ satisfies 
\begin{equation}\label{eq-numerical-2}
\frac{s^\dagger}{1-s^\dagger}=\frac{e^B}{q-1}\left(\frac{e^\beta \lambda^\dagger+q-1}{\lambda^\dagger+e^\beta+q-2}\right)^{d}\,,\quad\forall \dagger\in \{\free,\wired\}\,.
\end{equation}
\footnote{Intuitively, we may think $q\in \mathbb{N}$, and thus $\tfrac{q-1}{q}\psi^\dagger+\tfrac{1}{q}=\nu^{\dagger}(1)$, $\lambda^\dagger=\nu^\dagger(1)/\nu^\dagger(2)$, $s^\dagger=\check{\nu}^\dagger(1)$. Then \eqref{eq-numerical-1} and \eqref{eq-numerical-2} follow from the Potts belief-propagation.}
Denote 
\[
t^\dagger:=\sqrt{\frac{e^\beta}{(q-1)(e^\beta+q-2)}}\cdot\lambda^\dagger\,,\quad\dagger\in \{\free,\wired\}\,.
\]
Utilizing the critical condition \eqref{eq-w_c(B)}, we have
\[
\frac{e^B}{q-1}=\left(\frac{1+w/(q-1)}{1+w}\right)^{d/2}=\left(\frac{e^\beta+q-2}{(q-1)e^\beta}\right)^{d/2}\,.
\]
Plugging this into \eqref{eq-numerical-1} and \eqref{eq-numerical-2}, it is straightforward to check that \eqref{eq-numerical-1} becomes 
\begin{equation}\label{eq-numerical-3}
t^\dagger=\left(\frac{e^{2\beta^*}t^\dagger+1}{t^\dagger+e^{2\beta^*}}\right)^{d-1}=H_{\beta^*}(t^\dagger)\,,\quad \forall\dagger\in \{\free,\wired\}\,,
\end{equation}
and \eqref{eq-numerical-2} becomes 
\begin{equation}\label{eq-numerical-4}
\frac{s^\dagger}{1-s^\dagger}=\left(\frac{e^{2\beta^*}t^\dagger+1}{t^\dagger+e^{2\beta^*}}\right)^d\,,\quad\forall \dagger\in \{\free,\wired\}\,.
\end{equation}
\eqref{eq-numerical-3} implies that $t^{\wired}=\tfrac{x(\beta^*)}{1-x(\beta^*)}$ and $t^{\free}=\frac{1-x(\beta^*)}{x(\beta^*)}$. Therefore, by comparing \eqref{eq-numerical-4} with \eqref{eq-x^*(bbeta)} we obtain that $x^*(\beta^*)=s^{\wired}$ and $1-x^*(\beta^*)=s^{\free}$, and thus \eqref{eq-rel-psi-m} follows.
\end{proof}

\subsection{Proof of Lemma~\ref{lem-RCM-Ber-domination}}\label{appendix-stochastic-domination}
Fix a finite graph $G=(V,E)$ and a real number $q\ge 1$ as the cluster parameter (we take $q=2$ in Lemma~\ref{lem-RCM-Ber-domination}).

 let $\varphi^{w}=\varphi_G^{w,0}$ be the random cluster measure on $G$ defined by
\[
\varphi^{w}[\eta]=\frac{w^{|E(\eta)|}q^{|C(\eta)|}}{\mathcal Z_G^{\RC,w,0}}\,,\quad\forall \eta\in\{0,1\}^{E(G)}\,.
\]
Let $\varphi^{w}\oplus \operatorname{Ber}^\varepsilon$ denote the distribution of $\eta\vee \eta'$ where $\eta\sim \phi_{p,q}$ and $\varepsilon\sim \operatorname{Ber}^\varepsilon$ are independent. 

\begin{lemma}
    For any $q\ge 1$, $p_1>p_2>0$, $\xi>0$ such that 
    \begin{equation}\label{eq-eps-assumption}
    \frac{\xi}{1-\xi}<\frac{p_1-p_2}{q}\,,
    \end{equation}
  it holds that $\varphi^{w_1}\succeq_{\operatorname{st}}\varphi^{w_2}\oplus \operatorname{Ber}^\xi$. 
\end{lemma}

\begin{proof}
    By \cite[Theorem 2.6]{Gri06}, it suffices to check that for any $\eta\in \{0,1\}^{E(G)}$ and $e\in E(G)$, 
    \begin{equation}\label{eq-stochastic-domination-goal}
    \varphi^{w_1}[\eta^{e}]\varphi^{w_2}\oplus \operatorname{Ber}^\xi[\eta_e]\ge \varphi^{w_1}[\eta_e]\varphi^{w_2}\oplus \operatorname{Ber}^\xi[\eta^e]\,,
    \end{equation}
    where $\eta^e=\eta\vee \delta_e$ and $\eta_e=\eta\wedge(\mathbf{1}-\delta_e)$. 
    By definition we have
    \begin{align*}
    \varphi^{w_2}\oplus \operatorname{Ber}^\xi[\eta_e]=&\ \sum_{\omega\preceq \eta_e}\varphi^{w_2}[\omega]\cdot \operatorname{Ber}^\xi[\{\omega':\omega\wedge\omega'=\eta_e\}]\\
    =&\ \sum_{\omega\preceq \eta_e}\frac{p_2^{|E(\omega)|}q^{|C(\omega)|}}{\mathcal Z_{G}^{\RC,w_2,0}}\times \xi^{|E(\eta_e)|-|E(\omega)|}(1-\xi)^{|E(G)|-|E(\eta_e)|}\\
    =&\ \frac{(1-\xi)^{|E(G)|}}{\mathcal Z_{G}^{\RC,w_2,0}}\times \left(\frac{\xi}{1-\xi}\right)^{|E(\eta_e)|}\times \sum_{\omega\preceq \eta_e}p_2^{|E(\omega)|}q^{|C(\omega)|}\xi^{-|E(\omega)|}\,,
    \end{align*}
    and similarly,
    \begin{align*}
            \varphi^{w_2}\oplus \operatorname{Ber}^\xi[\eta^e]=&\ \sum_{\omega\preceq \eta^e}\varphi^{w_2}[\omega]\cdot \operatorname{Ber}^\xi[\{\omega':\omega\wedge\omega'=\eta^e\}]\\
            =&\sum_{\omega\preceq \eta_e}\frac{p_2^{|E(\omega)|}q^{|C(\omega)|}}{\mathcal Z_{G}^{\RC,w_2,0}}\times \xi^{|E(\eta^e)|-|E(\omega)|}(1-\xi)^{|E(G)|-|E(\eta^e)|}\\
            +&\ \sum_{\omega\preceq \eta_e}\frac{p_2^{|E(\omega^e)|}q^{|C(\omega^e)}}{\mathcal Z_{G}^{\RC,w_2,0}}\times \xi^{|E(\eta^e)|-|E(\omega^e)|}(1-\xi)^{|E(G)|-|E(\eta^e)|}\\
            =&\ \frac{(1-\xi)^{E(G)}}{\mathcal Z_{G}^{\RC,w_2,0}}\times \left(\frac{\xi}{1-\xi}\right)^{1+E(\eta_e)}\times \sum_{\omega\preceq \eta_e}p_2^{|E(\omega)}q^{|C(\omega)|}\xi^{-|E(\omega)|}\\
            +&\frac{(1-\xi)^{E(G)}}{\mathcal Z_{G}^{\RC,w_2,0}}\times \left(\frac{\xi}{1-\xi}\right)^{E(\eta_e)}\times \sum_{\omega\preceq \eta_e}p_2^{1+|E(\omega)}q^{|C(\omega^e)|}\xi^{-|E(\omega)|}\,.
    \end{align*}
    Therefore, \eqref{eq-stochastic-domination-goal} reduces to showing that
    \begin{align*}
    &\ p_1q^{|C(\eta^e)|-|C(\eta_e)|}\sum_{\omega\preceq \eta_e}p_2^{|E(\omega)|}q^{|C(\omega)|}\xi^{-|E(\omega)|}\\
    \ge&\ \frac{\xi}{1-\xi}\cdot \sum_{\omega\preceq \eta_e}p_2^{|E(\omega)|}q^{|C(\omega)|}\xi^{-|E(\omega)|}+p_2\sum_{\omega\preceq \eta_e}p_2^{|E(\omega)|}q^{|C(\omega^e)|}\xi^{-|E(\omega)|}\,.
    \end{align*}
    Note that for any $\omega\preceq \eta_e$, $|C(\omega^e)|\le |C(\omega)|$. Moreover, if $\eta$ such that $|C(\eta^e)|-|C(\eta)|=-1$, then the two endpoints of $e$ are not connected in $\eta_e$, and hence they are not connected in any $\omega\preceq \eta_e$, either. This implies that $|C(\omega^e)|-|C(\omega)|=-1,\forall \omega\preceq \eta^e$. As a result, we see that
    \begin{align*}
        &\ p_1q^{|C(\eta^e)|-|C(\eta_e)|}\sum_{\omega\preceq \eta_e}p_2^{|E(\omega)|}q^{|C(\omega)|}\xi^{-|E(\omega)|}-p_2\sum_{\omega\preceq \eta_e}p_2^{|E(\omega)|}q^{|C(\omega^e)|}\xi^{-|E(\omega)|}\\
        \ge&\ \frac{p_1-p_2}{q}\sum_{\omega\preceq \eta_e}p_2^{|E(\omega)|}q^{|C(\omega)|}\xi^{-|E(\omega)|}\stackrel{\eqref{eq-eps-assumption}}{\ge} \frac{\xi}{1-\xi}\sum_{\omega\preceq \eta_e}p_2^{|E(\omega)|}q^{|C(\omega)|}\xi^{-|E(\omega)|}\,,
    \end{align*}
    as desired.
    This completes the proof.
\end{proof}

\subsection{Proof of Lemma~\ref{lem-local-CLT}}\label{appendix-lem-local-CLT}
\begin{proof}
We first consider the Potts case. Fix $i \in [q]$ and let $I_n \subset V_n$ be a maximal independent set of $G_n$, so $|I_n| \to \infty$ as $n \to \infty$. For any realization of $\sigma|_{V_n \setminus I_n}$ from $\mu_n$, the spins $\{\sigma(v): v \in I_n\}$ are conditionally independent, with $\mathbb{P}[\sigma(v) = i]$ uniformly bounded away from $0$ and $1$. By the local central limit theorem for sums of independent variables (see \cite{Maller1978LocalLimit}), we have $\mathbb{P}[\#\{v \in V_n : \sigma(v) = i\} = t] = O(|I_n|^{-1/2}) = o(1)$ uniformly for all $t \in \mathbb{N}$, proving \eqref{eq-CLT-Potts}.

The random cluster case follows a similar strategy, but is more involved. Fix $0<\varepsilon<0.1$ and a large integer $N\ge 10$. Let $c<1$ be as in Lemma~\ref{lem-cluster-size}, and choose $R\in\mathbb{N}$ such that $c^R<\varepsilon$. Since $G_n \stackrel{\operatorname{loc}}{\longrightarrow} \ttT_d$, for large $n$ we can select vertices $v_1,\dots,v_N\in V_n$ with pairwise distances at least $2R$. Given a realization $\check{\eta} = \eta|_{E_n^* \setminus \cup_{i=1}^N E_{n,1}^*(v_i)}$ from $\varphi_n$, call $v_i$ \emph{bad} if it connects to $\partial V_{n,R}(v_i)$ but not to $v^*$ in $\check{\eta}$. By Lemma~\ref{lem-cluster-size}, the bad probability is at most $c^R\le\varepsilon$ for each $i$. Let $A\subset \{v_1,\dots,v_N\}$ be the set of good vertices. Markov's inequality yields the probability $|A|\le N/2$ is at most $2\varepsilon$.

On the other hand, given the realization of $\check{\eta}$ and $A$, it holds that after further conditioning on the realization of $\eta\!\mid_{E_n^*\setminus \cup_{v_i\in A}E_{n,1}^*(v_i)}$, the events $\{v_i\leftrightarrow v^*\},v_i\in A$ are conditionally independent with each other. Moreover, these events have conditional probability uniformly bounded away from $0$ and $1$, and thus by the local central limit theorem, the conditional probability that $\#\{v\in V_n:v\leftrightarrow v^*\}=t$ is at most $O(|A|^{-1/2})$ for any $t\in \mathbb N$. We conclude that as $n\to\infty$
\[
\sup_{t\in \mathbb N}\varphi_n\big[\#\{v\in V_n:v\leftrightarrow v^*\}=t\big]\le 2\varepsilon+O(N^{-1/2})\,.
\]
Since this is true for all small $\varepsilon>0$ and large $N\in \mathbb{N}$, the desired result follows.
\end{proof}

\subsection{Proof of Lemma~\ref{lem-infinite-product}}\label{appendix-lem-infinite-product}
\begin{proof}
Let $x^{\free},x^{\wired}$ be the smallest and the largest positive solution of the equation
    \begin{equation}\label{eq-p(x)}
    x=p(x)\equiv e^B\left(\frac{e^\beta x+q-1}{e^\beta+x+q-2}\right)^{d-1}\,.
    \end{equation}
Following similar arguments as in the proof of \cite[Lemma 21]{GSVY16}, we see that the eigenvalues of $Q^\dagger$ are given by
\[
\lambda_1^\dagger=1>\lambda_2^\dagger=\frac{e^\beta x^\dagger}{e^\beta x^\dagger+q-1}-\frac{x^\dagger}{e^\beta+x^\dagger+q-2}>\lambda_3=\cdots=\lambda_q^\dagger=\frac{1}{e^\beta x^\dagger+q-1}>0\,.
\]
    \begin{claim}\label{claim-lambda_2}
        For any $(\beta,B)\in \mathbb R_+^2\setminus \mathsf R_=$ and $\dagger\in \{\free,\wired\}$, $\lambda_2^\dagger<\frac{1}{d-1}$.
        %\begin{equation}\label{eq-lambda_2}
        %\left|\frac{e^\beta x^\dagger}{e^\beta x^\dagger+q-1}-\frac{x^\dagger}{e^\beta+x^\dagger+q-2}\right|<\frac{1}{d-1}\,.
        %\end{equation}
    \end{claim}
    \begin{proof}[Proof of Claim~\ref{claim-lambda_2}]
 Fix $0<B<B_+$ and $\dagger\in \{\free,\wired\}$. Recall that $\beta_{\operatorname{ord}}(B)$ is the threshold at where the ordered phase emerges, and $\beta_{\operatorname{dis}}(B)$ is the threshold above which the disordered phase disappears. Therefore, $\beta_{\operatorname{ord}}(B)$ and $\beta_{\operatorname{dis}}(G)$ are respectively characterized by  $p'(x^{\wired})=1$ and $p'(x^{\free})=1$. Now for $\dagger\in \{\free,\wired\}$, assume that $\beta> 0$ satisfies that 
 \begin{equation}\label{eq-assumption-lambda2}
 \lambda_2^\dagger(\beta,B)=\frac{1}{d-1}\ \iff \ \frac{(d-1)(e^\beta-1)(e^\beta+q-1)x^\dagger}{(e^\beta x^\dagger+q-1)(e^\beta+x^\dagger+q-2)}=1\,.
 \end{equation}
 Since 
 \[
 p'(x^\dagger)=  e^B\cdot (d-1)\left(\frac{e^\beta x^\dagger+q-1}{e^\beta+x^\dagger+q-1}\right)^{d-2}\cdot \frac{(e^\beta-1)(e^\beta+q-1)}{(e^\beta+x^\dagger+q-1)^2}\,,
 \]
 \eqref{eq-p(x)} and \eqref{eq-assumption-lambda2} altogether imply that $p'(x^\dagger)=1$. It follows that $\beta=\beta_{\operatorname{dis}}(B)$ if $\dagger=\free$ and $\beta=\beta_{\operatorname{ord}}(B)$ if $\dagger=\wired$. Consequently, the only zero point of $r_B^\dagger(\beta):=\lambda_2(\beta,B)-\frac{1}{d-1}$ on $(0,\infty)$ is $\beta_{\operatorname{dis}}(B)$ if $\dagger=\free$ and $\beta_{\operatorname{ord}}(B)$ if $\dagger=\wired$. Meanwhile, it is straightforward to check that $\lim_{\beta\to 0}r^{\free}_B(\beta)<0$ and $\lim_{\beta\to\infty}r_B^{\wired}(\beta)<0$. Since the functions $r_B^{\free}$ and $r_B^{\wired}$ are continuous with respect to $\beta\in (0,\infty)$, it follows that $\lambda^{\free}_2<\frac{1}{d-1}$ for all $0<\beta<\beta_{\operatorname{dis}}(B)$ and $\lambda_2^{\wired}<\frac{1}{d-1}$ for all $\beta>\beta_{\operatorname{ord}}(B)$. This concludes the claim.
    \end{proof}
    Turning back to the proof of Lemma~\ref{lem-infinite-product}, for any $\dagger\in \{\free,\wired\}$, we have
    \[
    \Big|\log\prod_{k\ge 3}e^{-\theta_k\delta_k^\dagger}\Big|=\Big|\sum_{k\ge 3}\theta_k\delta_k^\dagger\Big|\le \sum_{k\ge 3}\frac{(d-1)^k}{2k}\cdot (q-1)\lambda_2^k<\infty\,,
    \]
    and thus the infinite product $\prod_{k\ge 3}e^{-\theta_k\delta_k^\dagger}$ converges. Moreover, we fix $t\in (\max\{\lambda_2^{\free},\lambda_2^{\wired}\},\frac{1}{d-1})$. For any $\varepsilon>0$, we pick $K_0=K_0(\varepsilon)$ such that 
    \[
      \sum_{k>K_0}\theta_kt^k<\tfrac{\varepsilon}{2}\,,\quad\text{and}\quad
\sum_{k>K_0}\delta_k^\dagger t^{-k}<\log(1+\varepsilon)\,,\quad\dagger\in \{\free,\wired\}\,.
    \]
    For $G_n\sim \mathcal G_{n,d}$, it is known in \cite{bollobas1980probabilistic} that $\mathbb{E}[X_{n,k}]\le 2\theta_k$ for any $k\ge 3$. By Markov inequality and our choice of $K_0$, it holds with probability at least $1-\varepsilon$ that $X_{n,k}\le t^{-k},\forall k>K_0$. Given this is true, we have for any $K\ge K_0$,
    \[
    \prod_{k>K}(1+\delta_k^\dagger)^{X_{n,k}}\le \exp\Big(\sum_{k>K}\delta_k^\dagger t^{-k}\Big)<1+\varepsilon\,.
    \]
    This concludes the proof. 
\end{proof}

\subsection{Proof of Lemma~\ref{lem-non-recostruction}}\label{appendix-lem-non-reconstruction}
\begin{proof}

    It suffices to prove that for any $(\beta,B)\in \mathbb{R}_+^2\setminus \mathsf R_=$, the free and wired Potts measure on the infinite $(d-1)$-array tree with inverse temperature $\beta$ and external field $B$ is non-reconstructive. Recall that $x^{\free},x^{\wired}$ are the smallest and the largest root of the equation \eqref{eq-p(x)}. 
    The free and wired measures can be viewed as broadcasting tree processes on the $(d-1)$-array tree with transition matrices given by
    \[
    M_{11}^\dagger=\frac{e^\beta x^\dagger}{e^\beta x^\dagger+q-1}\,,\quad M_{1j}^\dagger=\frac{1}{e^\beta x^\dagger+q-1}\,,\quad \forall 2\le j\le q
    \]
    \[
    M_{i1}^\dagger=\frac{x^\dagger}{e^\beta+x^\dagger+q-2}\,,\quad M_{ij}^\dagger=\frac{e^{\beta\delta_{i,j}}}{e^\beta+x^\dagger+q-1}\,,\quad \forall 2\le i,j\le q\,,
    \]
    where $\dagger=\free$ and $\wired$, respectively. 
    Following the coupling argument as in the proof of \cite[Lemma 53]{GSVY16}, it suffices to show the following:
    \[
    \kappa^\dagger:=\frac{1}{2}\max_{i\neq j\in [q]}\sum_{k\in [q]}|M_{ik}-M_{jk}|<\frac{1}{d-1}\,.
    \]
    However, one can easily verify that
    \[
    \kappa^\dagger=\frac{1}{2}\sum_{k\in [q]}|M_{ik}-M_{2k}|=\frac{e^\beta x^\dagger}{e^\beta x^\dagger+q-1}-\frac{x^\dagger}{e^\beta+x^\dagger+q-2}=\lambda_2^\dagger\,,
    \]
    which is less than $\frac{1}{d-1}$ by Claim~\ref{claim-lambda_2}. This concludes the proof.
\end{proof}

\subsection{Proof of Lemma~\ref{lem-Delta_{d+1}}}\label{appendix-lem-Delta_{d+1}}

\begin{proof}
    Fix $d,q\ge 3$ and $(\beta,B)\in \mathsf R_c$. We denote 
    \[
P=\sum_{j=1}^q\check{\nu}^{\free}(j)e^{\beta\delta_{j,1}}\cdot \Big(\sum_{i,j=1}^q\check{\nu}^{\wired}(i)\check{\nu}^{\wired}(j)e^{\beta\delta_{i,j}}\Big)^{1/2}\,,
\]
\[
    Q=\sum_{j=1}^q\check{\nu}^{\free}(j)e^{\beta\delta_{j,2}}\cdot \Big(\sum_{i,j=1}^q\check{\nu}^{\wired}(i)\check{\nu}^{\wired}(j)e^{\beta\delta_{i,j}}\Big)^{1/2}\,,
    \]
    \[
    R=\sum_{j=1}^q\check{\nu}^{\wired}(j)e^{\beta\delta_{j,1}}\cdot \Big(\sum_{i,j=1}^q\check{\nu}^{\free}(i)\check{\nu}^{\free}(j)e^{\beta\delta_{i,j}}\Big)^{1/2}\,,
    \]
    \[
    S=\sum_{j=1}^q\check{\nu}^{\wired}(j)e^{\beta\delta_{j,2}}\cdot \Big(\sum_{i,j=1}^q\check{\nu}^{\free}(i)\check{\nu}^{\free}(j)e^{\beta\delta_{i,j}}\Big)^{1/2}\,,
    \]
    and $C_1=e^B,C_2=q-1$. It is straightforward to check that
    \[
    \Delta_{\ell,\free}>(<)\Delta_{\ell,\wired}\quad\iff \quad C_1P^\ell+C_2Q^\ell>(<)C_1R^\ell+C_2S^\ell\,.
    \]
    Additionally, since $\beta$ is the critical temperature, we have
    \[
    \Delta_{d,\free}=\Delta_{d,\wired}\quad\iff\quad C_1P^d+C_2Q^d=C_1R^d+C_2S^d\,.
    \]
    We denote the common sum $C_1P^d+C_2Q^d=C_1R^d+C_2S^d$ as $T$. In what follows we show that for $t>0$, 
    \begin{equation}\label{eq-t<>1}
    C_1P^{dt}+C_2Q^{dt}>(<)C_1R^{dt}+C_2S^{dt}\quad\iff\quad t<(>)1\,,
    \end{equation}
    and the desired result follows. 

    We first show that $0<S^d<Q^d<\frac{T}{C_1+C_2}<P^d<R^d<\frac{T}{C_1}$. As $\nu^\dagger(1)>\nu^\dagger(2)=\cdots=\nu^\dagger(q)$, it is straightforward to check that $P>Q$ and $R>S$. Therefore, it remains to show 
    \[
    P<R\iff \Phi(\check{\nu}^{\free})<\Phi(\check{\nu}^{\wired})\,,\quad \Phi(\check{\nu}):=\frac{\sum_{j=1}^q\check{\nu}(j)e^{\beta\delta_{j,1}}}{ \big(\sum_{i,j=1}^q\check{\nu}(i)\check{\nu}(j)e^{\beta\delta_{i,j}}\big)^{1/2}}\,.%<\frac{\sum_{j=1}^q\check{\nu}^{\wired}(j)e^{\beta\delta_{j,1}+h\delta_{j,1}}}{ \Big(\sum_{i,j=1}^q\check{\nu}^{\wired}(i)\check{\nu}^{\wired}(j)e^{\beta\delta_{i,j}+h\delta_{i,1}+h\delta_{j,1}}\Big)^{1/2}}\,,
    \]
    Note that $    \Phi(\check{\nu}^\dagger)=F\left(\check{\nu}(1)\right)$ for $\dagger\in \{\free,\wired\}$, 
    where 
    \[
    F(x):=\frac{(e^{\beta+B}-1)x+1}{\sqrt{e^{\beta}\big(x^2+\frac{(1-x)^2}{q-1}\big)+1}}\,.
    \]
    We compute that
    \[
    F'(x)=\frac{-(e^{\beta+B}+q-1)(e^\beta-1)x+(e^{\beta+B}-1)(q-1)+e^{\beta+B}(e^\beta-1)}{\big(e^{\beta}\big(x^2+\frac{(1-x)^2}{q-1}\big)+1\big)^{3/2}}\,.
    \]
    Since the numerator $N(x)$ of $F'(x)$ is linear with negative slope, and $N(1)=(q-1)e^\beta(e^B-1)\ge 0$, we conclude that $F'(x)>0$ for $x\in [0,1)$. Thus, $F$ is strictly increasing on $[0,1]$, and hence $\Phi(\check{\nu}^{\free})=F(\check{\nu}^{\free})<F(\check{\nu}^{\wired})=\Phi(\check{\nu}^{\wired})$, the claim follows. 

    Given the above relation, for any fixed $t>0$, we consider the function
    \[
    G(x)=C_1x^t+C_2\left(\frac{T-C_1x}{C_2}\right)^t\,,\quad x\in\left(\frac{T}{C_1+C_2},\frac{T}{C_1}\right)\,.
    \]
    We compute that
    \[
    G'(x)=C_1tx^{t-1}-C_1t\left(\frac{T-C_1x}{C_2}\right)^{t-1}\,,
    \]
    which is negative (resp. positive) for $0<t<1$ (resp. $t>1$). This yields that $G(P^d)>G(R^d)$ for $0<t<1$ and $G(P^d)<G(R^d)$ for $t>1$, thereby verifying \eqref{eq-t<>1}.
\end{proof}

\bibliographystyle{alpha}
\small
\bibliography{ref}

\end{document}